\documentclass{report}
\usepackage{appendix}
\usepackage{epsfig}
\usepackage{graphicx}
\usepackage[namelimits]{amsmath}
\usepackage{amsthm}
\usepackage{amsfonts}
\usepackage{amssymb}
\usepackage[all]{xy}
\usepackage{fancyhdr}
\usepackage{textcomp}

\newcommand{\Ber}{\mathrm{Ber}}
\newcommand{\ev}{\mathrm{ev}}
\newcommand{\odd}{\mathrm{odd}}
\newcommand{\tild}{\raise.17ex\hbox{$\scriptstyle\sim$}}

\newtheorem{thm}{Thm}[section]
\newtheorem{satz}[thm]{Theorem}
\newtheorem{lemma}[thm]{Lemma}

\newtheorem{prop}[thm]{Proposition}

\theoremstyle{definition}
\newtheorem{def1}[thm]{Definition}
\newtheorem{bsp}[thm]{Example}

\theoremstyle{remark}
\newtheorem{rem}[thm]{Remark}

\title{On flag domains in the supersymmetric setting \\ \vspace{15pt} revised reprint of the author's PhD Thesis at the Ruhr-Universit\"{a}t Bochum}
\author{Christopher Graw\thanks{partially supported by the SFB-TR 12: Universality in mesoscopic systems}}

\begin{document}


\maketitle
\tableofcontents

\chapter{Introduction and summary of results}

The subject of this paper is the geometry of flag domains in the supersymmetric setting. 
Classically a flag domain is an open submanifold $D$ of a $G$-flag manifold $Z$ endowed with a transitive action of a real form $G_\mathbb{R}$ of $G$. Flag domains are important, e.g. for the study of moduli spaces and the representation theory of real reductive Lie groups.

A detailed analysis of the orbit structure of $G_\mathbb{R}$ in $Z$ and of the geomtery of flag domains was provided by J.A. Wolf in \cite{W}. In particular he proved that there is a finite number of $G_\mathbb{R}$-orbits in $Z$, hence there are always open $G_\mathbb{R}$-orbits. Moreover he proved that every open $G_\mathbb{R}$-orbit $D$ contains a unique orbit $C_0$ of a maximal compact subgroup $K_\mathbb{R}$, which is a complex submanifold. This \emph{base cycle} is of fundamental importance, in particular for the understanding of the cohomology of line bundles on $D$.

The classical foundation of our work here consists of three topics which are covered in the classical case in Chapters 4,5 and 14 of \cite{FHW}:

\begin{itemize}
\item The question of existence of a $G_\mathbb{R}$-invariant volume form on a flag domain $D$ (Measurability).
\item The study of global holomorphic functions on $D$.
\item A Double Fibration Transform realizing the cohomology of line bundles of $D$ as a subrepresentation of a space of functions on a space of cycles.
\end{itemize}

In this paper these three topics are discussed in the framework of flag domains in flag supermanifolds, i.e. $G$ is a Lie supergroup and $Z$ and $D$ are complex supermanifolds.
Quite often such an approach amounts to formal generalizations of definitions and arguments to the supersymmetric case. However in all three cases considered here there appear phenomena which differ from the classical case.

The following chapter is devoted to reviewing a number of results from the theory of classical flag domains and the necessary background on the theory of supersymmetry that is used throughout the subsequent chapters. 

In the third chapter the $G_\mathbb{R}$-orbits with maximal odd dimension and the measurable flag domains are classified. In the super-symmetric case it is possible that a $G_\mathbb{R}$-orbit $D$ has open base $D_{\bar{0}}$, but its odd dimension is not maximal. The odd dimension is characterized by a codimension formula similar to the one in \cite{W} and the conditions for maximal odd dimension are classified in all cases.

Measurability in the supersymmetric sense is defined by requiring the existence of a $G_\mathbb{R}$-invaraint Berezinian form, the natural analogon of a volume form in the classical case. In that case measurability is characterized in root-theoretic as well as geometric terms in \cite{W}. One of the equivalent conditions is that the stabilizer subgroup $L_\mathbb{R}$ is a real reductive group. It turns out that this is not necessarily the case in the supersymmetric setting. This motivates the introduction of two distinct notions of weak and strong measurability. Both strongly and weakly measurable flag domains are classified in Chapter 3.
In most cases the conditions for weak or strong measurability can be expressed in terms of three 
symmetry conditions which represent the symmetries of the extended Dynkin diagrams of type $A(m,n)$.
 Most parts of this chapter coincide with the pre-publication in \cite{G}.

The fourth chapter deals with the question of the existence of global holomorphic superfunctions on a flag domain $D$. Classically there is a projection from $D$ onto a hermitian symmetric domain (possibly reduced to a point) which induces an isomorphism on global functions(see \cite{W}). This projection is also of great importance in the supersymmetric case. Moreover, even if there are no non-constant holomorphic functions on the base $D_{\bar{0}}$, there are sometimes global odd functions on $D$ and even on the flag supermanifold $Z = G/P$ itself. In the latter case these are described in \cite{V}. They are given by non-trivial $G_{\bar{0}}$-submodules of the cotangent space $(\mathfrak{g}_1/\mathfrak{p}_1)^*$. This characterization of global odd functions extends to flag domains and the resulting spaces of global odd functions on flag domains are classified in Chapter 4.

In both the third and the fourth chapter the results are summarized in tables at the end of the respective chapter.

Cycle spaces and the corresponding Double Fibration Transform in the supersymmetric case are studied in the fifth chapter. Classically the cycle space $\mathcal{M}$ is the connected component of the set of all $G$-translates of the base cycle $C$ which are contained in $D$. In the supersymmetric case the definition of the cycle space needs some more elaboration. The main reason for this is that commuting involutions (i.e. the antiholomorphic involution $\tau$ defining $G_\mathbb{R}$ and the Cartan involution $\theta$) and compact real forms are not available in most cases. It is therefore necessary to allow one of the two maps $\tau$ and $\theta$ to be an automorphism of order 4. This leads to two different types of cycle spaces which may be defined in a unified way using the notion of universal domain from \cite{FHW}. 

After clarifying the cycle space notion, the Double Fibration Transform relating the cohomology of line bundles of $D$ with spaces of sections of line bundles on the cycle space $\mathcal{M}$ is constructed in analogy with the construction in the classical case, given for example in \cite{WZ}. As in the classical case the question of injectivity and image of the Double Fibration Transform are of major importance. 
Classically these questions can be answered using the Bott-Borel-Weil Theorem and a number of results of this type are given in \cite{WZ}. The main point is that sufficient negativity of a weight $\lambda$ leads to injectivity of the respective Double Fibration Transform for the line bundle $\mathcal{E}_\lambda$ on $D$.

In the supersymmetric case the Bott-Borel-Weil theory is not yet fully developed and the known results show that is much richer in content. For example unlike in the classical case it is possible to have more than one non-vanishing cohomology group for the highest weight representation of an integral dominant weight.  In the second part of Chapter 5 the known results on the Bott-Borel-Weil theory for Lie superalgebras are used to obtain injectivity conditions for the Double Fibration Transform including in one case the possibility of two distinct non-trivial Double Fibration Transforms for the same weight $\lambda$ whose target spaces are twisted duals of each other.     

\chapter{Background on the geometry of flag domains}

Geometry is here understood to mean the classical geometry of flag domains as well as the supergeometry.
 
\section{The classical theory of flag domains}

This review of the classical theory of flag domains is based largley on \cite{W} and \cite{FHW}. 
The sections on injectivity and image of the Double Fibration Transform 
are based on the coverage in \cite{WZ}.
 
First the basic definitions of real forms, flag manifolds and flag domains are introduced and  a number of results from \cite{W} are reviewed: First of all, the open orbits of a real form $G_\mathbb{R}$ are distinguished among the finite number of $G_\mathbb{R}$-orbits by a codimension formula. Then the existence of the base cycle, the unique closed complex submanifold of a flag domain fixed by a maximal compact subgroup $K_\mathbb{R} \subseteq G_\mathbb{R}$, is discussed. Furthermore the global holomorphic functions on a flag domain $D$ are computed using the projection onto the bounded symmetric domain subordinate to $D$ and the notion of measurablity of a flag domain is defined and characterized in both root-theoretic and geometric terms.

After this the construction of the group-theoretic cycle space and the Double Fibration Transform following the coverage in \cite{FHW} and \cite{WZ} is discussed. The cycle space $\mathcal{M}$ is the connected component of the moduli space of all $G$-translates of the base cycle which lie inside $D$. The Double Fibration Transform relates the cohomology groups $H^p(D, \mathcal{O}(\mathbb{E}))$ of holomorphic vector bundles on $D$ with sections of certain associated vector bundles $\mathcal{O}(\mathbb{E}^\prime)$ on $\mathcal{M}$. Of particular importance are the injectivity of the DFT and a concrete description of its image inside $H^0(\mathcal{M}, \mathcal{O}(\mathbb{E}^\prime))$. Injectivity conditions and a concrete description of the image are then obtained making use of the Bott-Borel-Weil-Theorem.     

\subsection{Real forms, flag spaces and flag domains}

Let $G$ be a complex semisimple Lie group and $\mathfrak{g}$ be its Lie algebra.

\begin{def1}
 
A real form of $\mathfrak{g}$ is a real Lie subalgebra $\mathfrak{g}_\mathbb{R}$ which is the fixed point set
of a $\mathbb{C}$-antilinear involution $\tau: \mathfrak{g} \rightarrow \mathfrak{g}$.

\end{def1}

It is a basic result that the involution $\tau$ descends to an involution of $G$, 
which will also be denoted by $\tau$, and the connected component $G_\mathbb{R}$ of its fixed point set is 
called a real form of $G$.

Note that in the classical case every complex semisimple Lie group has a compact real form. As it turns out this is not the case
for complex Lie supergroups.

A Lie subalgebra $\mathfrak{b} \subseteq \mathfrak{g}$ is called a Borel subalgebra if it is a maximal solvable Lie subalgebra of $\mathfrak{g}$.
Its normalizer $B = \{g \in G : Ad(g)B = B\}$ is called a Borel subgroup of $G$. A subalgebra $\mathfrak{p}$ containing a
Borel subalgebra is called a parabolic subalgebra of $\mathfrak{g}$. Its normalizer $P$ is called a parabolic subgroup of $G$.
The following theorem due to Tits characterizes parabolic subgroups:

\begin{satz}[\cite{T1},\cite{T2}]
 Let $G$ be a complex semisimple Lie group, $P$ a complex algebraic subgroup and $G_u$ a compact real form of $G$. 
 Then the following are equivalent:

\begin{enumerate}
 \item $P$ is a parabolic subgroup.
 \item $Z = G/P$ is a compact complex manifold.
 \item $Z$ is a complex projective variety.
 \item $Z$ is a $G_u$-homogeneous compact K\"{a}ehler manifold.
 \item $Z$ is the projective space orbit of an extremal highest weight vector in an irreducible
       finite-dimensional representation of $G$.
 \item $Z$ is a $G$-equivariant quotient manifold of $\hat{Z} = G/B$ for some Borel subgroup $B \subseteq G$. 
\end{enumerate}
 
\end{satz}

Let $\mathfrak{h}$ be a Cartan subalgebra of $\mathfrak{g}$, $\Sigma(\mathfrak{g},\mathfrak{h})$
the corresponding root system and $\Pi$ a set of simple roots. For $J \subseteq \Pi$ let $\vert J \vert$ denote the 
set of all roots which are sums of elements of $J$. Then for every parabolic subalgebra $\mathfrak{p} \subseteq \mathfrak{g}$ containing the Borel subalgebra
$\mathfrak{b} = \mathfrak{h} \oplus \sum_{\alpha \in \Sigma^+} \mathfrak{g}^\alpha$ there is a unique subset $J \subseteq \Pi$ such that
$\mathfrak{p} = \sum_{\alpha \in \Sigma^+} \mathfrak{g}^\alpha \oplus \sum_{\alpha \in \vert J \vert} \mathfrak{g}^{-\alpha}$.
Denote $\Phi = \Sigma^+ \cup -\vert J \vert, \Phi^r = \vert J \vert \cup - \vert J \vert, \Phi^n = \Sigma^+ \setminus \vert J \vert $ 
and $\Phi^c = \Sigma \setminus \Phi$. Then $\Phi^r$ is the set of roots constituing the Levi component of $\mathfrak{p}$ and
$\Phi^n$ is the set of roots constituing the nilpotent radical of $\mathfrak{p}$. Conversely given $\Phi \subseteq \Sigma$ as before it is possible
to construct the parabolic subalgebra $\mathfrak{p}_\Phi = \mathfrak{h} \oplus \sum_{\alpha \in \Phi} \mathfrak{g}^\alpha$.
The corresponding parabolic subgroup of $G$ is then denoted $P_\Phi$.

The quotient manifold $Z = G/P$ is called a $G$-flag manifold. Real forms of $G$ act on $G$-flag manifolds. Their orbit structure
was analysed in depth in \cite{W}. The first main results is:

\begin{satz}

Let $Z$ be a $G$-flag manifold and $G_\mathbb{R}$ a real form of $G$. Then $G_\mathbb{R}$ has finitely many orbits
in $Z$, in particular $Z$ contains open $G_\mathbb{R}$-orbits. 

\end{satz}

\begin{def1}

Let $G_\mathbb{R}$ and $Z$ as before. An open $G_\mathbb{R}$-orbit $D \subseteq Z$ is called a flag domain.

\end{def1}

Note that the defining involution $\tau$ of $\mathfrak{g}_\mathbb{R}$ induces an involution of the root system $\Sigma(\mathfrak{g},\mathfrak{h})$, 
if $\mathfrak{h}$ is a $\tau$-invariant Cartan subalgebra. Moreover a $\tau$-invariant Cartan subalgebra $\mathfrak{h}$ always exists. 
The action on the root system can be used to obtain the following codimension formula:

\begin{satz}[2.12. in \cite{W}]
 
Let $M = G_\mathbb{R} \cdot z_0 \subseteq Z = G/P$  a real group orbit, $P_{z_0} = \mathrm{Stab}_G(z_0)$ and $\Sigma(\mathfrak{g},\mathfrak{h})$ a root system 
of $\mathfrak{g}$ such that $\mathfrak{h}$ is $\tau$-invariant and $P_{z_0} = P_\Phi$ for a certain $\Phi \subseteq \Sigma$. Then $\mathrm{codim}_Z M = \vert \Phi^c \cap \tau \Phi^c \vert$.
In particular, $M$ is open if and only if that intersection is empty.

\end{satz}

Let $\theta: \mathfrak{g}_\mathbb{R} \rightarrow \mathfrak{g}_\mathbb{R}$ be the Cartan involution,
$\theta^\mathbb{C}: \mathfrak{g} \rightarrow \mathfrak{g}$ its complexification and $\mathfrak{k} = \mathrm{Fix}(\theta^\mathbb{C})$.
Then $\theta^\mathbb{C}$ descends to $G$. Let $K$ be the connected component of its fixed point set. Then the intersection $K_\mathbb{R} = K \cap G_\mathbb{R}$ is
a maximal compact subgroup of $G_\mathbb{R}$. There is a correspondence between the orbit structures of $K$ and $G_\mathbb{R}$ on $Z$
called the Matsuki correspondence. For flag domains the statement is the following:

\begin{satz}

Let $Z$ be a $G$-flag manifold and $D \subseteq Z$ a flag domain. Then $D$ contains a unique closed $K$-orbit $C_0$ called the base cycle.
It can also be characterized as the unique complex $K_\mathbb{R}$-orbit inside $D$.  

\end{satz}

\subsection{Holomorphic functions}

The holomorphic functions on $D$ depend largely on the base cycle $C_0$. The connection is described in [W, 5.6 and 5.7]
using the term of the bounded symmetric domain subordinate to $D$. It is constructed as follows:

As $G$ is semisimple, $\mathfrak{g} = \mathfrak{g}_1 \oplus \ldots \oplus \mathfrak{g}_k$ is a sum of simple ideals 
and an analogous decomposition exists for $\mathfrak{p}$. Moreover $Z = G_1/P_1 \times \ldots \times G_k/P_k$
and $D = G_{1,\mathbb{R}}/L_{1,\mathbb{R}} \times \ldots \times G_{k,\mathbb{R}}/L_{k,\mathbb{R}}$. Let $K_{i,\mathbb{R}} 
\subseteq G_{i,\mathbb{R}}$ maximal compact and consider the following conditions:

\begin{enumerate}
 \item $L_{i,\mathbb{R}}$ is compact, thus contained in $K_{i,\mathbb{R}}$
 \item $G_{i,\mathbb{R}}/K_{i,\mathbb{R}}$ is a hermitian symmetric space
 \item $G_{i,\mathbb{R}}/L_{i,\mathbb{R}} \rightarrow G_{i,\mathbb{R}}/K_{i,\mathbb{R}}$ is
       holomorphic for one of the two invariant complex structures on $G_{i,\mathbb{R}}/K_{i,\mathbb{R}}$.
\end{enumerate}

If all these conditions are satisfied, set $M_i = K_{i,\mathbb{R}}$. Otherwise set $M_i = G_{i,\mathbb{R}}$.
Then $\tilde{D} = G_{1,\mathbb{R}}/M_1 \times \ldots \times G_{k,\mathbb{R}}/M_k$ is the bounded symmetric domain
subordinate to $D$ and the following holds:

\begin{satz}[5.7 in \cite{W}]

The canonical homomorphism \newline $H^0(\pi): H^0(\tilde{D}, \mathcal{O}) \rightarrow H^0(D,\mathcal{O})$ induced by the projection $\pi: D \rightarrow \tilde{D}$ is an isomorphism . 

\end{satz}

\begin{bsp}
 
Let $G = SL_n(\mathbb{C})$ and $G_\mathbb{R} = SU(p,q)$. Furthermore let $D^+$ be the set of all positive $p$-planes in $\mathbb{C}^n$
and $D^-$ be the set of all negative $q$-planes in $\mathbb{C}^n$. Then $D^+$ and $D^-$ are open $G_\mathbb{R}$-orbits
in the respective Grassmannians $\mathrm{Gr}_p(\mathbb{C}^n)$ and $\mathrm{Gr}_q(\mathbb{C}^n)$ and hermitian symmetric spaces.
Moreover an open $G_\mathbb{R}$-orbit $D = G_\mathbb{R}/L_\mathbb{R}$ allows non-constant global holomorphic funtions if and only if it projects onto
either $D^+$ or $D^-$. If it exists this projection is actually a $G_\mathbb{R}$-equivariant proper holomorphic map
$p: G_\mathbb{R}/L_\mathbb{R} \rightarrow G_\mathbb{R}/K_\mathbb{R}$ which is in fact a holomorphically trivial fibre bundle. 

\end{bsp}

\subsection{Measurable open orbits}

A flag domain $D$ is called measurable, if it possesses a $G_\mathbb{R}$-invariant volume element. A characterization of
measurable flag domains in algebraic and geometric terms was given in \cite{W}. The following formulation of the theorem
is taken from \cite{FHW}.

\begin{satz}[4.5.1 in \cite{FHW}]\label{WMeas}
 
Let $D = G_\mathbb{R} \cdot z$ be an open orbit in the complex flag manifold $Z = G/P$. Then the following are equivalent:

\begin{enumerate}
 \item $D$ is measurable
 \item $G_\mathbb{R} \cap P$ is the $G_\mathbb{R}$-centralizer of a torus subgroup of $G_\mathbb{R}$.
 \item $D$ has a $G_\mathbb{R}$-invariant, possibly indefinite, K\"{a}hler metric, thus a $G_\mathbb{R}$-invariant measure
obtained from the volume form of that metric.
 \item $\tau \Phi^r = \Phi^r$ and $\tau \Phi^n = \Phi^{-n}$
 \item $\mathfrak{p} \cap \tau\mathfrak{p}$ is a complex reductive Lie algebra.
 \item $\mathfrak{p}\cap\tau\mathfrak{p} $ is the Levi component of $\mathfrak{p}$
 \item $\tau \mathfrak{p}$ is $G$-conjugate to the parabolic subalgebra $\mathfrak{p}^r + \mathfrak{p}^{-n}$ opposite to $\mathfrak{p}$ 
\end{enumerate}

\end{satz}

A classification of measurable flag domains in the classical case is included in the classification of measurable flag superdomains in Chapter 3. 

\subsection{Cycle spaces and the Double Fibration Transform}

Let $D \subseteq Z$ be a flag domain and $C_0 \subseteq D$ the base cycle. As it is a closed complex submanifold of $Z$ its
stabilizer $J = \mathrm{Stab}_G(C_0) = \{ g \in G : gC_0 = C_0\}$ is a closed complex Lie subgroup of $G$. Let $\mathcal{M}_Z = G/J$.

\begin{def1}
  $\mathcal{M} = \{ gC_0 \in \mathcal{M}_Z : gC_0 \subseteq D \}^\circ$ is the (group-theoretic) cycle space of $D$.
\end{def1}

\begin{bsp}

Let $\mathfrak{j}$ and $\mathfrak{k}$ be the respective Lie algebras of $J$ and $K$.
If $g_\mathbb{R}$ is supposed to be simple, there are three distinct cases of cycle spaces depending on whether $\mathfrak{k}$ is a maximal subalgebra of $\mathfrak{g}$
or not and whether $\mathfrak{j} = \mathfrak{k}$ or $\mathfrak{j}$ is a subalgebra of $\mathfrak{g}$ properly containing $\mathfrak{k}$.
They are the following ones (see page 57 in \cite{FHW}):

\begin{description}
 \item[Hermitian holomorphic case:] $\mathcal{B} = G_\mathbb{R}/K_\mathbb{R}$ is a hermitian symmetric space of non-compact type, 
the complement $\mathfrak{s}$ of $\mathfrak{k}$ in $\mathfrak{g}$ splits into a direct sum $\mathfrak{s}^+ \oplus \mathfrak{s}^-$
of $K$-modules, $\mathfrak{j}$ is either of the maximal parabolic subalgebras $\mathfrak{p}^+ = \mathfrak{k} + \mathfrak{s}^+$
and $\mathfrak{p}^- = \mathfrak{k} + \mathfrak{s}^-$ and $\mathcal{M}$ is either $\mathcal{B}$ or $\bar{\mathcal{B}}$ respectively.
 \item[Hermitian non-holomorphic case:] $\mathcal{B} = G_\mathbb{R}/K_\mathbb{R}$ is again a hermitian symmetric space of non-compact type,
but $\mathfrak{j} = \mathfrak{k}$. Then $\mathcal{M}$ is isomorphic to $\mathcal{B} \times \bar{\mathcal{B}}$.
 \item[Generic case:] $G_\mathbb{R}/K_\mathbb{R}$ does not possess a $G_\mathbb{R}$-invariant complex structure, $\mathfrak{k}$
is a maximal subalgebra of $\mathfrak{g}$ and $\mathcal{M}$ is the universal domain $\mathcal{U}$ introduced in Chapter 6 of \cite{FHW}
(see pages 77-81).
\end{description}

The hermitian cases can only occur for $G_\mathbb{R} = SU(p,q),SO^*(2n),Sp_{2m}(\mathbb{R})$ or $SO(n,2)$.

\end{bsp}

The flag domain $D$ and the cycle sapce $\mathcal{M}$ fit into a double fibration

\[ \xymatrix{ & \mathfrak{X} \ar[dl]_\mu \ar[dr]^\nu & \\ D & & \mathcal{M} } \]

\noindent where $\mathfrak{X} = \{ (z,C) \in D \times \mathcal{M} : z \in C \}$ is called the universal family.

In this text a vector bundle on a complex manifold $X$ is always understood to be a locally free $\mathcal{O}_X$-module $\mathcal{E}$.
If $\mathcal{E} = \mathcal{O}(\mathbb{E})$, the sheaf of germs of sections of a locally trivializable holomorphic submersion $p: \mathbb{E} \rightarrow X$, then $\mathbb{E}$ is called the total space of the vector bundle $\mathcal{E}$.

Let $\mathbb{E} \rightarrow D$ be a locally trivializable holomorphic submersion and $\mathcal{E} = \mathcal{O}(\mathbb{E})$. Then the double fibration gives rise to a Double Fibration Transform relating the the cohomology
of $\mathcal{O}(\mathbb{E})$ with the global sections of certain associated sheaves on $\mathcal{M}$. It is constructed as follows:

\subsubsection{Pullback}

The first step is to pull back the cohomology from $D$ to $\mathfrak{X}$ along $\mu$. If $\mathcal{G}$ is an arbitrary sheaf 
on $D$ then $\mu^{-1} \mathcal{G}$ is the sheafification of the presheaf given by 
\[ U \mapsto \varinjlim_{V \supseteq f(U)} \mathcal{G}(V) \]
For every $r \geq 0$ and every sheaf $\mathcal{G}$ on $D$, $\mu$
induces a map \newline $\mu^r : H^r(D,\mathcal{G}) \rightarrow H^r(\mathfrak{X},\mu^{-1}\mathcal{G})$.

\noindent The condition for these maps to be injective is purely topological:

\begin{def1}

Let $q \geq 0$. The fibration $\mu: \mathfrak{X} \rightarrow D$ satisfies the Buchdahl $q$-condition if the fibre
$F$ of $\mu$ is connected and $H^r(F,\mathbb{C}) = 0$ for $0 < r < q$.   

\end{def1}

This has the following effect on the pullback maps $\mu^r$

\begin{satz}(\cite{Bu})\label{Buchdahl}
 Let $q \geq 0$ and $\mu: \mathfrak{X} \rightarrow D$ be given. If $\mu$ satisfies the Buchdahl $q$-condition, then
$\mu^r$ is an isomorphism for $r < q$ and injective for $r = q$. If the fibres of $\mu$ are cohomologically acyclic,
then $\mu^r$ is an isomorphism for all $r$.
\end{satz}

Note that in the cases under consideration according to [FHW], Theorem 14.5.2 and Proposition 14.6.1 on page 212f., the fibres of $\mu$ will actually always be contractible so 
$\mu^r$ will always be an isomorphism.  

Now let $\mu^*\mathcal{O}(\mathbb{E}) = \mu^{-1}\mathcal{O}(\mathbb{E}) \otimes_{\mu^{-1}\mathcal{O}_D} \mathcal{O}_\mathfrak{X}$ denote the pullback sheaf.  
It is the sheaf of germs of holomorphic section of the projection $\mu^*\mathbb{E} \rightarrow \mathfrak{X}$. The Extension of Scalars induces morphisms 
$i_r: H^r(\mathfrak{X}, \mu^{-1}\mathcal{O}(\mathbb{E})) \rightarrow H^r(\mathfrak{X}, \mu^*\mathcal{O}(\mathbb{E}))$. Let $j_r$ be the composition $ = i_r \mu^r: H^r(D, \mathcal{O}(\mathbb{E})) \rightarrow H^r(\mathfrak{X}, \mu^*\mathcal{O}(\mathbb{E}))$.

\subsubsection{Pushdown}

The second step is to push $H^r(\mathfrak{X}, \mu^* \mathcal{O}(\mathbb{E}))$ down from $\mathfrak{X}$ to $\mathcal{M}$ along $\nu$. 
To this end one needs to make use of the Grauert Direct Image Theorem:

\begin{satz}

Let $f: X \rightarrow Y$ be a proper map of complex manifolds and $\mathcal{S}$ a coherent 
$\mathcal{O}_X$-module. Then the $p^{th}$ direct image $\mathcal{R}^p f_* \mathcal{S}$ is a coherent $\mathcal{O}_Y$-module
for all $p \geq 0$. 

\end{satz}

Returning to the case of the double fibration note that $\nu$ is indeed a proper holomorphic map so the Grauert Direct Image Theory can be applied.
Moreover one of the major results of \cite{FHW} is that the cycle space $\mathcal{M}$ is actually a Stein space in all possible cases (Theorem 11.3.1 and Theorem 11.3.7 on page 171f.).
This implies $H^q(\mathcal{M},\mathcal{R}^p \nu_* \mu^* \mathcal{O}(\mathbb{E})) = 0$ for all $q > 0, p \geq 0$.
Consequently the Leray spectral sequence (see \cite{Wei}, chapter 5) collapses to yield isomorphisms $\mathcal{R}^r \nu_* : H^r(\mathfrak{X}, \mu^* \mathcal{O}(\mathbb{E})) \cong
H^0(\mathcal{M}, \mathcal{R}^r \nu_* \mu^* \mathcal{O}(\mathbb{E}))$.

\begin{def1}

The Double Fibration Transform is the composition

\[ \mathcal{P} = \mathcal{R}^r \nu_* j_r : H^r(D,\mathcal{O}(\mathbb{E})) \rightarrow H^0(\mathcal{M}, \mathcal{R}^r \nu_* \mu^* \mathcal{O}(\mathbb{E})) \] 

\end{def1}

In order for this Double Fibration Transform to be useful two conditions need to be satisfied:

\begin{itemize}
 \item $\mathcal{P}$ needs to be injective.
 \item There needs to be some concrete characterization of the image of $\mathcal{P}$.
\end{itemize}

\subsubsection{Injectivity and Image of the Double Fibration Transform}\label{BBW}

As the pushdown map $\mathcal{R}^r \nu_*$ is always an isomorphism in all cases under consideration and the fibres of $\mu$ are contractible, the injectivity 
question is reduced to the question of injectivity of the coefficient map $i_r: H^r(\mathfrak{X}, \mu^{-1}\mathcal{O}(\mathbb{E})) \rightarrow H^r(\mathfrak{X}, \mu^*\mathcal{O}(\mathbb{E}))$.  
This question can be answered by considering a suitable resolution of $\mu^{-1}\mathcal{O}(\mathbb{E})$ by $\mathcal{O}_\mathfrak{X}$-modules.
The following material is based on pages 530ff. in \cite{WZ}. 

Let $x \in \mathfrak{X}$ and $T^{1,0}_{\mu,x} = T^{1,0}(\mu^{-1}(\mu(x)))$ be the holomorphic tangent space at $x$ of the
fibre of $\mu$ containing $x$. The disjoint union of all these spaces defines a subbundle $T^{1,0}_\mu \mathfrak{X}$ of
the holomorphic tangent bundle $T^{1,0}\mathfrak{X}$. The sheaf of germs of $\mu$-relative holomorphic $p$-forms
is then given by $\Omega_\mu^p = \mathcal{O}(\bigwedge^p (T^{1,0}_\mu \mathfrak{X})^*)$. For every $p \geq 0$ there is
a surjective map $\pi$ from $\Omega_\mathfrak{X}^p$ onto $\Omega_\mu^p$ given point-wise by restriction of differential forms
to $\bigwedge^p T_{\mu,x}$. This allows to define the relative exterior differential $\partial_\mu: \Omega_\mu^p \rightarrow \Omega_\mu^{p+1}$
by $\partial_\mu(\eta) = \pi(\partial(\omega))$ where $\omega$ is an arbitrary element of $\pi^{-1}(\eta)$.
This defines a complex $\Omega_\mu^*$ of $\mathcal{O}_\mathfrak{X}$-modules. Moreover $\mu^{-1}\mathcal{O}_D$ can
be identified as the kernel of $\partial_\mu: \mathcal{O}_\mathfrak{X} \rightarrow \Omega_\mu^1$, hence
$\Omega_\mu^*$ is a resolution of $\mu^{-1} \mathcal{O}_D$ by $\mathcal{O}_\mathfrak{X}$-modules (Lemma 2.12. on p.530 in \cite{WZ}).
Tensoring with $\mu^{-1}\mathcal{O}(\mathbb{E})$ over $\mu^{-1}\mathcal{O}_D$ then yields the desired resolution of $\mu^{-1}\mathcal{O}(\mathbb{E})$: 

\[ \xymatrix{ 0 \ar[r] & \mu^{-1}\mathcal{O}(\mathbb{E}) \ar[r] & \mu^*\mathcal{O}(\mathbb{E}) \ar[r] & \Omega_\mu^1(\mathcal{O}(\mathbb{E})) \ar[r] & \ldots \ar[r] &
 \Omega_\mu^m(\mathcal{O}(\mathbb{E})) \ar[r] & 0}\]

According to chapter 5 in \cite{Wei}, this resolution gives rise to two spectral sequences $^\prime E$ and $^{\prime\prime}E$ including the pages $^\prime E_2^{p,q} = H^p(\mathfrak{X}, \mathcal{H}^q(\mathfrak{X}), \Omega^\bullet(\mathcal{O})(\mathbb{E}))$ and $^{\prime\prime}E_2^{p,q} = H_d^q(H^p(\mathfrak{X},\Omega^\bullet(\mathcal{O}(\mathbb{E}))))$ and both converging to the hypercohomology $\mathbb{H}^{p+q}(\mathfrak{X},\Omega^\bullet(\mathcal{O}(\mathbb{E})))$.

Here $\mathcal{H}^q(\mathfrak{X}, \Omega^\bullet(\mathcal{O}(\mathbb{E})))$ is the sheaf given locally by

\[ \mathcal{H}^q(\mathfrak{X}, \Omega^\bullet(\mathcal{O}(\mathbb{E})))(U) = \frac{ \mathrm{Ker} (d: \Gamma(U, \Omega^q(\mathcal{O}(\mathbb{E}))) \rightarrow \Gamma(U, \Omega^{q+1}(\mathcal{O}(\mathbb{E})))) }
{\mathrm{Im} (d: \Gamma(U, \Omega^{q-1}(\mathcal{O}(\mathbb{E}))) \rightarrow \Gamma(U, \Omega^{q}(\mathcal{O}(\mathbb{E}))))} \]

Exactness of the relative de Rham complex yields $\mathcal{H}^q(\mathfrak{X}, \Omega^\bullet(\mathcal{O}(\mathbb{E}))) = 0$
for all $q > 0$. Consequently $^\prime E$ collapses at the $E_2$-page yielding an isomorphism between the hypercohomology
and the cohomology of $\mu^{-1}\mathcal{O}(\mathbb{E})$. The latter also appears in $^{\prime\prime}E_2$ as the kernel
of the map $d_0: H^p(\mathfrak{X}, \mu^*\mathcal{O}(\mathbb{E})) \rightarrow H^p(\mathfrak{X}, \Omega_\mu^1(\mathcal{O}(\mathbb{E})))$. 
This helps to give a concrete description of the image of $\mathcal{P}$.

The Double Fibration Transform is certainly injective if $H^p(\mathfrak{X}, \mu^*\mathcal{O}(\mathbb{E}))$
survives to the $^{\prime\prime}E_\infty$-page. A sufficient condition for this is that $^{\prime\prime}E_2^{r,q} = H_d^q(H^r(\mathfrak{X},\Omega^\bullet(\mathcal{O}(\mathbb{E})))) = 0$
for all $r < p, 1 \leq q \leq m$.

\noindent Pushing down to $\mathcal{M}$ this is equivalent to the vanishing of $H^0(\mathcal{M}, R^r \nu_* \Omega_\mu^q(\mathcal{O}(\mathbb{E})))$
for all $r < p, 1 \leq q \leq m$.

Suppose $\mathbb{E}$ is the total space of a $G_\mathbb{R}$-homogeneous vector bundle with fibre $E$. Then $\Omega_\mu^q(\mathcal{O}(\mathbb{E}))$
is the sheaf of germs of holomorphic sections of the homogeneous vector bundle on $\mathfrak{X}$ with fibre
$\bigwedge (\mathfrak{p}^{op}/(\mathfrak{p}^{op} \cap \mathfrak{j}))^* \otimes E$ and $H^0(\mathcal{M}, R^r \nu_* \Omega_\mu^q(\mathcal{O}(\mathbb{E})))$
is the space of global sections of the $G_\mathbb{R}$-homogeneous vector bundle on $\mathcal{M}$ with fibre 
$H^r(C_0,\mathcal{O}(K_\mathbb{R} \times_{(K_\mathbb{R} \cap L_\mathbb{R})}(\bigwedge (\mathfrak{p}^{op}/(\mathfrak{p}^{op} \cap \mathfrak{j}))^* \otimes E)))$. 

So the vanishing of all these spaces for $r < p, 1 \leq q \leq m$ implies injectivity of the double fibration transform.

In \cite{WZ} the authors give a number of vanishing results using the Bott-Borel-Weil Theorem:

\begin{satz}[Bott-Borel-Weil-Theorem \cite{Bott}]

Let $G$ be a complex \\ semisimple Lie group, $G_u$ a compact real form and $\lambda$ a $G_u$-dominant weight. Let $\rho = \frac{1}{2} \sum_{\alpha \in \Sigma^+(\mathfrak{g},\mathfrak{h})} \alpha$, half the sum of the positive roots.

\begin{enumerate}
 \item If $\lambda + \rho$ is singular, then $H^k(Z, \mathcal{E}_\lambda) = 0$ for all $k \geq 0$.
 \item If $\lambda + \rho$ is regular, then there is a unique element $w \in W$ of the Weyl group such that
       $\langle w(\lambda + \rho), \psi \rangle > 0$ for every simple root $\psi$. Let $q = l(w)$, the length of $w$.
       Then $H^k(Z, \mathcal{E}_\lambda) = 0$ for $k \neq q$ and $H^q(Z, \mathcal{E}_\lambda)$ is the irreducible $G_u$-module
       with highest weight $w \cdot \rho = w(\lambda + \rho) - \rho$.
\end{enumerate}
 
\end{satz} 
Let $\mathfrak{s}$ be the complementary subspace to $\mathfrak{k}$ in $\mathfrak{g}$ and $\mathfrak{r}^+$ be the
unipotent radical of $\mathfrak{p}$. The final result in the hermitian holomorphic case is the following:

\begin{satz}

Let $\lambda \in \mathfrak{h}^*$ be a weight and $\mathbb{E} = G_\mathbb{R} \times_{L_\mathbb{R}} E_\lambda$ where
$E_\lambda$ is the representation space of the highest weight representation with highest weight $\lambda$. If
$\lambda$ is sufficiently negative, e.g. 

\[ \langle \mu + \beta, \gamma \rangle < 0 \] 

\noindent where $\beta$ is a sum of elements in $\Sigma(\mathfrak{l} \cap \mathfrak{s}^+, \mathfrak{h})$
and $\gamma \in \Sigma(\mathfrak{k} \cap \mathfrak{r}^+, \mathfrak{h})$,  

\noindent then the Double Fibration Transform $\mathcal{P}: H^p(D,\mathcal{O}(\mathbb{E})) \rightarrow H^0(\mathcal{M}, R^p \nu_* \mu^* \mathcal{O}(\mathbb{E}))$ 
is injective.

\end{satz}

In the hermitian non-holomorphic case the condition is slightly different:

\begin{satz}

Let $\lambda \in \mathfrak{h}^*$ be a weight and $\mathbb{E} = G_\mathbb{R} \times_{L_\mathbb{R}} E_\lambda$ where
$E_\lambda$ is the representation space of the highest weight representation with highest weight $\lambda$. If
$\lambda$ is sufficiently negative, e.g. 

\[ \langle \mu + \beta + \rho_\mathfrak{k}, \gamma \rangle < 0 \] 
\noindent where $\beta$ is a sum of elements in  $\Sigma({\mathfrak{p}^{op}} \cap \mathfrak{s}, \mathfrak{h})$
and $\gamma \in \Sigma(\mathfrak{q}^c \cap \mathfrak{k})$, 

\noindent then the Double Fibration Transform $\mathcal{P}: H^p(D,\mathcal{O}(\mathbb{E})) \rightarrow H^0(\mathcal{M}, R^p \nu_* \mu^* \mathcal{O}(\mathbb{E}))$ 
is injective.

\end{satz}

The image of the double fibration transform is contained in the space of global sections of a certain vector bundle $\mathcal{E}^\prime$ on $\mathcal{M}$ and can be identified with the solution of a system of partial differential equations,
which are actually the characterization of the image of $\mathcal{P}$ as the kernel of the map 
$d_0: H^0(\mathcal{M}, \mathcal{R}^p \nu_* \mu^* \mathcal{O}(\mathbb{E}))$ \\ $\rightarrow H^0(\mathcal{M}, \mathcal{R}^p \nu_* \Omega_\mu^1(\mathcal{O}(\mathbb{E})))$.  
The target space of the Double Fibration Transform is also accesible to characterization using Bott-Borel-Weil theory(see \cite{WZ}, Sections 4 and 6).

\section{Supermanifolds and Super Vector Bundles}

In this section a series of important notions from the theory of supermanifolds are introduced
in order to allow a discussion of the topics from the preceding sections in the context of flag domains in flag supermanifolds.

The basic notions are those of Super vector spaces and Lie superalgebras. A super vector space
is a vector space endowed with a $\mathbb{Z}/2$-grading and Lie superalgebras possess a bracket which is graded-antisymmetric and satisfies the graded Jacobi identity. For further use in the subsequent chapters of this thesis the classification of complex simple Lie superalgebras due to V.G. Kac is reviewed.

After this supermanifolds in the sense of Kostant, Berezin and Leites are introduced and the notion of a split supermanifold is defined. Then Lie supergroups and Super Harish-Chandra Pairs
are introduced and an equivalence between these two categories is established. Afterwards real forms and even real forms of complex Lie supergroups and homogeneous superspaces are introduced.
Here the flag supermanifolds, quotients of complex Lie supergroups $G$ by parabolic subsupergroups $P$, are of particular importance. As in the classical case open submanifolds with a transitive action of a real form $G_\mathbb{R}$ of $G$ are called flag domains, but unlike in the classical case they need not always exist for every real form in every flag supermanifold.

In order to be able to discuss the Double Fibration Transform in the context of flag domains in flag supermanifolds super vector bundles are introduced and some results from their general theory are stated. In particular, the complex of differential forms on a super manifold and a relative complex for a holomorphic submersion are constructed. 

The chapter closes with a discussion of a spectral sequence, developed by A. L. Onishchik and E. G. Vishnyakova, which relates the cohomology of a super vector bundle $\mathcal{E}$ with the cohomology of an associated $\mathbb{Z}$-graded super vector bundle $\mathrm{gr} \mathcal{E}$.
In many cases the cohomology of the latter super vector bundle can be computed much more easily than that of $\mathcal{E}$ itself. This allows to give an upper bound for the cohomology of super vector bundles in those cases.    

\subsection{Super vector spaces and Lie superalgebras}

Before discussing the generalization of the above theory to the supersymmetric setting, a brief introduction
of the notions and techniques that will be used seems in order. In order to avoid confusion the sheaf of holomorphic
functions on a complex manifold will from now on be denoted by the letter $\mathcal{F}$. The material in this section
is standard, the presentation is similar to the one in \cite{V} and \cite{OV1}. Some results from these papers will be used and extended
later on in this text. The basic objects are super vector spaces and Lie superalgebras:

\begin{def1}
\begin{enumerate}
\item A super vector space is a $\mathbb{Z}/2$-graded vector space $V = V_{\bar{0}} \oplus V_{\bar{1}}$. The vector subspaces
$V_{\bar{0}}$ and $V_{\bar{1}}$ are called the even and the odd part of $V$ respectively. If $v \in V_j$ then its parity is $\vert v \vert = j$. A linear map of super vector spaces is called
even or odd respectively, if it preserves or interchanges the gradation.
\item The parity changed vector space $\Pi V$ is the vector space $V$ with the gradation $(\Pi V)_{\bar{0}} = V_{\bar{1}}, (\Pi V)_{\bar{1}} = V_{\bar{0}}$. 
The parity change is the natural odd isomorphism $\Pi: V \rightarrow \Pi V$.
\item A Lie superalgebra is a super vector space $\mathfrak{g}$ together with a bilinear mapping $[\cdot,\cdot]: \mathfrak{g} \times \mathfrak{g} \rightarrow \mathfrak{g}$
 that is skew-supersymmetric and satisfies the graded Jacobi identity, i.e. $[X,Y] = -(-1)^{\vert X \vert \vert Y \vert}[Y,X]$) and
$[X,[Y,Z]] = $ \\ $ [[X,Y],Z] + (-1)^{\vert X \vert \vert Y \vert} [Y,[X,Z]]$ for all $X,Y,Z \in \mathfrak{g}$.
\end{enumerate}
\end{def1}

\begin{rem}

In order to be able to distinguish between $\mathbb{Z}/2$-gradings and $\mathbb{Z}$-gradings, filtrations
or numberings, the two elements of $\mathbb{Z}/2$ are denoted ${\bar{0}}$ and ${\bar{1}}$ instead
of $0$ and $1$. 

\end{rem}

\begin{bsp}

Let $\mathfrak{g}_{\bar{0}}$ be a Lie algebra and $V$ be a $\mathfrak{g}_{\bar{0}}$-module. Then one may define
the structure of a Lie superalgebra on $\mathfrak{g} = \mathfrak{g}_{\bar{0}} \oplus V$ via the bracket

\[ [X,Y] = \begin{cases} [X,Y] & X,Y \in \mathfrak{g}_{\bar{0}} \\ X(Y) & X \in \mathfrak{g}_{\bar{0}}, Y \in \mathfrak{g}_{\bar{1}} \\ 0 &  X,Y \in \mathfrak{g}_{\bar{1}}
\end{cases} \]

\noindent Lie superalgebras of this type are called split Lie superalgebras.

\end{bsp}

\begin{bsp}

Let $\mathfrak{g} = \mathfrak{gl}_{n \vert m}(\mathbb{C})$ be the vector space of all complex $(n+m) \times (n+m)$-matrices.
One may identify $\mathfrak{g}$ with $\mathrm{End}(\mathbb{C}^n \oplus \Pi \mathbb{C}^m)$, where $\mathbb{C}^n$ and $\mathbb{C}^m$
are considered as purely even super vector spaces. Then there is a natural 
$\mathbb{Z}/2$-grading on $\mathfrak{g}$:

\[ \mathfrak{g}_{\bar{0}} = \mathrm{Hom}(\mathbb{C}^n, \mathbb{C}^n) \oplus \mathrm{Hom}(\Pi\mathbb{C}^m, \Pi\mathbb{C}^m),
\mathfrak{g}_{\bar{1}} = \mathrm{Hom}(\mathbb{C}^n, \Pi\mathbb{C}^m) \oplus \mathrm{Hom}(\Pi\mathbb{C}^m, \mathbb{C}^n)\] 

\noindent Using this grading the bracket on homogeneous elements is defined to be  $[X,Y]$ \\ $= XY - (-1)^{\vert X \vert \vert Y \vert} YX$.
Apart from the natural $\mathbb{Z}/2$-grading there are also natural $\mathbb{Z}$- and $\mathbb{Z}/2 \times \mathbb{Z}/2$-gradings
on $\mathfrak{g}$. They are given as follows:

\[ \mathfrak{g}_0 = \mathfrak{g}_{\bar{0}} = \mathfrak{g}_{\bar{0}\bar{0}} \oplus \mathfrak{g}_{\bar{1}\bar{1}} = \mathrm{Hom}(\mathbb{C}^n, \mathbb{C}^n) \oplus \mathrm{Hom}(\Pi\mathbb{C}^m, \Pi\mathbb{C}^m) \]
\[ \mathfrak{g}_1 = \mathfrak{g}_{\bar{0}\bar{1}} = \mathrm{Hom}(\Pi\mathbb{C}^m, \mathbb{C}^n), \mathfrak{g}_{-1} = \mathfrak{g}_{\bar{1}\bar{0}} = \mathrm{Hom}(\mathbb{C}^n, \Pi\mathbb{C}^m) \]

Like the commutator the notions of trace and transpose also have distinct analogues in the $\mathbb{Z}/2$-graded case:

The super-trace and super-transpose of a block matrix are respectively given by

\[ \begin{pmatrix} A & B \\ C & D \end{pmatrix}^{st} = \begin{pmatrix} A^T & -C^T \\ B^T & D^T \end{pmatrix}, \ \mathrm{str} \begin{pmatrix} A & B \\ C & D \end{pmatrix} = \mathrm{tr} A - \mathrm{tr} D\]

The respective analogue of the determinant, the Berezinian, will only be needed for purely even matrices(i.e. $B = C = 0$).
It is then defined to be $\frac{\det A}{\det D}$. Its derivative is the super-trace. 

\end{bsp}

\subsection{Classification of complex simple Lie superalgebras}

The classification of complex simple Lie superalgebras is due to Kac(see \cite{Kac}). He distinguishes between two classes
of simple Lie superalgebras, the classical Lie superalgebras and those of Cartan type. In this text only the classical 
Lie superalgebras are considered.

\begin{satz}

Up to isomorphism there are the following classical simple complex Lie superalgebras:

\begin{enumerate}
\item The simple complex Lie algebras.
\item The class $A(n,m)$ realized by $\mathfrak{sl}(n+1,m+1)(\{m,n\} \neq \{0,1\})$ in the case $m \neq n$ and by $\mathfrak{psl}(n+1,n+1)(n \geq 1)$ in the case $m = n$.
\item The class $B(m,n)$ realized by $\mathfrak{osp}(2m+1,2n)$, the orthosymplectic Lie algebra given by all matrices self-adjoint with respect to a non-degenerate even supersymmetric (i.e. $B(X,Y) = (-1)^{\vert X \vert \vert Y \vert} B(Y,X)$) bilinear form.
\item The class $D(m,n)$ realized by $\mathfrak{osp}(2m,2n), n > 1$.
\item The class $C(m)$ realized by $\mathfrak{osp}(2m,2)$.
\item The class $P(n)$ realized by $\mathfrak{sp}_n$, the algebra of all matrices self-adjoint with respect to a non-degenerate odd skew-supersymmetric (i.e. $B(X,Y) = -(-1)^{\vert X \vert \vert Y \vert} B(Y,X)$) bilinear form.
\item The class $Q(n)$ realized by $\mathfrak{psq}_n$, where $\mathfrak{q}_n$ is the subalgebra of all elements of $\mathfrak{gl}_{n \vert n}(\mathbb{C})$ commuting with an odd involution $\Pi$.
\item The exceptional family $D(1,2,\alpha)$.
\item The exceptional Lie algebra $G(3)$.
\item The exceptional Lie algebra $F(4)$.
\end{enumerate}

\end{satz}

The Lie superalgebras of types $A - G$ are the basic classical Lie superalgebras. There are two types of
classical Lie superalgebras which come equipped with distinct natural $\mathbb{Z}$-gradings:

\begin{def1}

\begin{enumerate}
 \item A classical Lie superalgebra $\mathfrak{g}$ is of type I if the adjoint representation of $\mathfrak{g}_{\bar{0}}$
       on $\mathfrak{g}_{\bar{1}}$ is completely reducible. It allows a natural $\mathbb{Z}$-grading 
       $\mathfrak{g} = \mathfrak{g}_{-1} \oplus \mathfrak{g}_0 \oplus \mathfrak{g}_1$ where $\mathfrak{g}_0 = \mathfrak{g}_{\bar{0}}$ and
       $\mathfrak{g}_{\pm 1}$ are the irreducible components of $\mathfrak{g}_{\bar{1}}$. The simple classical Lie superalgebras
       of type I are $A(n,m), C(m)$ and $P(n)$.
 \item A classical Lie superalgebra is of type II if the adjoint representation of $\mathfrak{g}_{\bar{0}}$
       on $\mathfrak{g}_{\bar{1}}$ is irreducible. In that case $\mathfrak{g}$ allows a natural $\mathbb{Z}$-grading
       $\mathfrak{g} = \mathfrak{g}_{-2} \oplus \mathfrak{g}_{-1} \oplus \mathfrak{g}_0 \oplus \mathfrak{g}_1 \oplus \mathfrak{g}_2$
       with $\mathfrak{g}_{\bar{1}} = \mathfrak{g}_{-1} \oplus \mathfrak{g}_1$. The simple classical Lie superalgebras of type II are
       $\mathrm{Osp}(m \vert 2n)$ for $n > 1, Q(n)$ and the exceptional simple Lie superalgebras.
\end{enumerate}

\end{def1}

\begin{bsp}[The irreducible components of $\mathfrak{g}_{\bar{1}}$ for classical Lie superalgebras of type I]\label{TypeIcomp}

Let $\mathfrak{g}$ be a complex simple classical Lie superalgebra of type I. Then $\mathfrak{g}$
is isomorphic to either of $\mathfrak{sl}_{n \vert m}(\mathbb{C})$, $\mathfrak{osp}(2 \vert 2m)$ or $\mathfrak{p}(n)$ for some $n,m \in \mathbb{N}$. The irreducible $G_{\bar{0}}$-submodules of $\mathfrak{g}_{\bar{1}}$ are given for these three classes of Lie superalgebras as follows:

\begin{enumerate}
\item If $\mathfrak{g} = \mathfrak{sl}_{n \vert m}(\mathbb{C})$, then the irreducible components
of $\mathfrak{g}_{\bar{1}}$ are the subspaces $\mathfrak{g}_1$ and $\mathfrak{g}_{-1}$ introduced in the last example.
\item Suppose $\mathfrak{g} = \mathfrak{osp}(2 \vert 2m)$. Then an element $M$ of $\mathfrak{g}$ has the following block matrix form:

\[ M = \begin{pmatrix} a & 0 & u_1 & u_2 \\ 0 & -a & v_1 & v_2 \\ v_2^T & u_2^T & X & Y \\ -v_1^T & -v_2^T & Z & -X^T \end{pmatrix}, \begin{matrix} a \in \mathbb{C} \\ u_i, v_i \in \mathbb{C}^m \\ X,Y,Z \in \mathbb{C}^{m \times m} \\ Y = Y^T, Z = Z^T \end{matrix} \]

\noindent and the irreducible components of $\mathfrak{g}_{\bar{1}}$ are $\mathfrak{g}_1 = \{ M \in \mathfrak{g}_{\bar{1}} : v_1 = v_2 = 0 \}$ and $\mathfrak{g}_{-1} = \{ M \in \mathfrak{g}_{\bar{1}} : u_1 = u_2 = 0 \}$.
\item Finally let $\mathfrak{g} = \mathfrak{p}(n)$. Then an arbitrary element $M$ has the following block form:

\[ M = \begin{pmatrix} X & Y \\ Z & -X^T \end{pmatrix}, \begin{matrix} X,Y,Z \in \mathbb{C}^{n \times n} \\ Y = -Y^T, Z = Z^T \end{matrix} \]

\noindent and the irreducible components of $\mathfrak{g}_{\bar{1}}$ are $\mathfrak{g}_1 = \{ M \in \mathfrak{g}_{\bar{1}} : Z = 0 \} $ and $\mathfrak{g}_{-1} = \{ M \in \mathfrak{g}_{\bar{1}} : Y = 0 \}$.

\end{enumerate} 

\end{bsp}

\subsection{Supermanifolds}

In this text supermanifolds are always understood in the sense of Kostant, Berezin and Leites:

\begin{def1}

A complex supermanifold is a ringed space $X = (X_{\bar{0}}, \mathcal{O}_X)$ with a $\mathbb{Z}/2$-graded structure sheaf, 
which is locally isomorphic to an open subset of $\mathbb{C}^{n \vert m}$. 
The underlying manifold(or base) of this linear model is $\mathbb{C}^n$ and its structure sheaf is 
given by $\mathcal{O}_{\mathbb{C}^{n \vert m}}(U) = \mathcal{F}_{\mathbb{C}^n}(U) \otimes \bigwedge \mathbb{C}^m$ for $U \subseteq \mathbb{C}^n$ open.
The $\mathbb{Z}/2$-grading is given by

\[ (\mathcal{O}_{\mathbb{C}^{n \vert m}})_{\bar{0}} = \mathcal{F}_{\mathbb{C}^n} \otimes \bigoplus_{k \in \mathbb{N}} \bigwedge^{2k} \mathbb{C}^m , \
 (\mathcal{O}_{\mathbb{C}^{n \vert m}})_{\bar{1}} = \mathcal{F}_{\mathbb{C}^n} \otimes \bigoplus_{k \in \mathbb{N}} \bigwedge^{2k+1} \mathbb{C}^m \]

Smooth supermanifolds, analytic supermanifolds, etc. are defined analogously. The natural numbers $n$ and $m$ are the even and odd dimension of $X$ respectively and the dimension of $X$
is $\dim X = n \vert m$. 

A morphism of supermanifolds is a morphism of ringed spaces that preserves the $\mathbb{Z}/2$-grading, i.e.
homogeneous elements are mapped to homogeneous elements of the same parity.

\end{def1}

\begin{bsp}
 
Let $X_{\bar{0}}$ be a complex manifold and $\mathcal{E}$ a vector bundle on $X_{\bar{0}}$.
Then $X = (X_{\bar{0}}, \bigwedge \mathcal{E})$ is a super manifold. Let $U \subseteq X_{\bar{0}}$ be a coordinate neighbourhood
as well as a trivializing neighbourhood for $\mathcal{E}$ and let $\theta_1, \ldots, \theta_m$ be an $\mathcal{F}(U)$-basis of $\mathcal{E}(U)$.
Then every superfunction on $U$ can be expressed in the following way:

\[ f = \sum_{\varepsilon \in (\mathbb{Z}/2)^m} f_\varepsilon \theta_1^{\varepsilon_1} \ldots \theta_m^{\varepsilon_m} , f_\varepsilon \in \mathcal{F}(U) \ \forall \varepsilon \in (\mathbb{Z}/2)^m  \] 

A supermanifold which is isomorphic to a supermanifold of this type is called split. 
Note that in this case the $\mathbb{Z}/2$-grading of $\mathcal{O}_X$ is inherited from the $\mathbb{Z}$-grading
of $\bigwedge \mathcal{E}$.

\end{bsp}

A very important feature of complex supermanifolds is that while every smooth supermanifold is split (Theorem 4.2.2 on page 188 in \cite{Ma}) this is generally not true for complex supermanifolds.
A low-dimensional example of a non-split complex supermanifold is the Grassmannian $\mathrm{Gr}_{1 \vert 1}(\mathbb{C}^{2 \vert 2})$(see \cite{Ma}).

Even though not every complex supermanifold is split it is always possible to assign to each complex supermanifold $X = (X_{\bar{0}}, \mathcal{O}_X)$
a split supermanifold which is locally isomorphic to $X$. It is constructed in a functorial way as follows:

Let $\mathcal{J} \subseteq \mathcal{O}_X$ be the ideal generated by all odd elements. Then define $(\mathrm{gr} \mathcal{O})_p =
\mathcal{J}^p/\mathcal{J}^{p+1}$ and $\mathrm{gr} \mathcal{O} = \bigoplus_{p \geq 0} (\mathrm{gr} \mathcal{O})_p$.
The manifold $\mathrm{gr} X = (X_{\bar{0}}, \mathrm{gr} \mathcal{O})$ is a split supermanifold which is locally isomorphic
to $X$. It is called the associated graded supermanifold or the retract of $X$.

Moreover if $f = (f_{\bar{0}},f^*): (X_{\bar{0}}, \mathcal{O}_X) \rightarrow (Y_{\bar{0}}, \mathcal{O}_Y)$ is a morphism of supermanifolds, 
the associated morphism of graded manifolds is $\mathrm{gr} f = (f_{\bar{0}}, \mathrm{gr} f^*)$ where $\mathrm{gr} f^*$ is
defined on $(\mathrm{gr}\mathcal{O}_Y)_p$ by $\mathrm{gr} f^* ( X + \mathcal{J}_Y^{p+1}) = f^*X + \mathcal{J}_X^{p+1}$.
Note that this is well-defined, as a morphism of supermanifolds preserves the $\mathbb{Z}/2$-grading and therefore maps
$\mathcal{J}_Y^{p+1}$ into $\mathcal{J}_X^{p+1}$.

\begin{def1}

 Let $x \in X_{\bar{0}}$ and let $\mathfrak{m}_x$ be the unique maximal ideal of the local ring $\mathcal{O}_{X,x}$.
Then the tangent space to $X$ at $x$ is the super vector space $T_x X = (\mathfrak{m}_x/\mathfrak{m}_x^2)^*$.
Its even part $(T_x X)_{\bar{0}}$ is the usual tangent space $T_x X_{\bar{0}}$.

\end{def1}

Note that the tangent spaces to $X$ and $\mathrm{gr} X$ always coincide. Moreover the tangent spaces of split supermanifolds are given as follows: 
Assume $X = (X_{\bar{0}}, \bigwedge \mathcal{E})$ and let $\mathbb{E}$ be the total space of $\mathcal{E}$. 
Then there is a natural identification $\mathbb{E}_x \cong (T_x X)_{\bar{1}}^*$.

\subsection{Lie supergroups and Super Harish-Chandra Pairs}

Lie supergroups $G$ are group objects within the category of supermanifolds, i.e. they come equipped
with three morphisms $m: G \times G \rightarrow G, i: G \rightarrow G$ and $e : \{pt\} \rightarrow G$
satisfying the usual group properties.

\noindent An important tool in the study of Lie supergroups and their actions, which will be defined in the next section,
is the equivalence between the categories of Lie supergroups and Super Harish-Chandra Pairs:

\begin{def1}

1. A Super Harish-Chandra Pair(or SHCP) is a pair $(G_{\bar{0}},\mathfrak{g})$ such that $G_{\bar{0}}$ is a Lie group, $\mathfrak{g}$ is a Lie
superalgebra, $\mathfrak{g}_{\bar{0}} = \mathrm{Lie}(G_{\bar{0}})$ and there is a representation $\mathrm{Ad}$ of $G_{\bar{0}}$ on $\mathfrak{g}$ such that
its derivative $\mathrm{ad}$ conincides with the adjoint representation of $\mathfrak{g}$ when restricted to $\mathfrak{g}_{\bar{0}}$.

\medskip
\noindent 2. A morphism of SHCP $f = (f_{\bar{0}}, F): (G_{\bar{0}},\mathfrak{g}) \rightarrow (H_{\bar{0}}, \mathfrak{h})$ is a pair of maps such that
$(df_{\bar{0}})_e = F \vert_{\mathfrak{g}_{\bar{0}}}$ and $\mathrm{Ad}(f_{\bar{0}}(g)) \circ F = F \circ \mathrm{Ad}(g) \ \forall g \in G_0$.

\end{def1}

\begin{bsp}
\label{SplitLSA}
Let $G$ be a Lie supergroup, $\mathfrak{g}$ its Lie superalgebra and $\mathfrak{g}^\prime$ be the Lie superalgebra
with the same underlying vector space as $\mathfrak{g}$, but with a modified bracket, which is constantly zero on $\mathfrak{g}_{\bar{1}} \times \mathfrak{g}_{\bar{1}}$. 
Then the Lie supergroup $\mathrm{gr} G$ is given by the SHCP $(G_0, \mathfrak{g}^\prime)$(see \cite{V}, Theorem 2).

\end{bsp}

\noindent The equivalence of categories is constructed as follows:

\noindent Let $(G_{\bar{0}},\mathfrak{g})$ be a SHCP and let $\mathcal{U}(\mathfrak{g})$ and $\mathcal{U}(\mathfrak{g}_{\bar{0}})$ be the respective
universal envelopping algebras. Then define a sheaf $\mathcal{O}_{(G_{\bar{0}},\mathfrak{g})}$ on $G_{\bar{0}}$ by

\[ \mathcal{O}_{(G_{\bar{0}},\mathfrak{g})}(U) = \mathrm{Hom}_{\mathcal{U}(\mathfrak{g}_{\bar{0}})}(\mathcal{U}(\mathfrak{g}), \mathcal{F}(U)) \]

\noindent This sheaf is isomorphic to $\mathcal{F} \otimes \bigwedge \mathfrak{g}_{\bar{1}}^*$ by virtue of the natural symmetrization map 
$\sigma: \bigwedge \mathfrak{g}_{\bar{1}} \rightarrow \mathcal{U}(\mathfrak{g})$.

Conversely, every Lie supergroup defines a SHCP in an obvious way. These two assignments are indeed functorial
and yield the desired equivalence of categories.

Using this equivalence of categories it is possible to prove that complex Lie supergroups are always split(Theorem 3.4.2 in \cite{CF}).
This resulst can also be derived as a consequence of the fact that the right-invariant vector fields yield an isomorphism
$\mathcal{O}_G(G_{\bar{0}}) \cong \mathcal{F}(G_{\bar{0}}) \otimes \bigwedge \mathfrak{g}_{\bar{1}}^*$, 
where $\mathfrak{g} = T_e G$ is the Lie superalgebra of the Lie supergroup $G$ (see Remark 3.6. in \cite{Kos}).
Note that even though every complex Lie supergroup $G$ is split, its split model will in general have a different Lie superalgebra
by virtue of Example \ref{SplitLSA}.

\subsection{Real forms and even real forms}

Let $G$ be a complex Lie supergroup and $\mathfrak{g}$ its Lie superalgebra. 

\begin{def1}
  1. A real form of $\mathfrak{g}$ is a real Lie subsuperalgebra $\mathfrak{g}_\mathbb{R} \subseteq \mathfrak{g}$, which
       is the fixed point set of a $\mathbb{C}$-antilinear involution $\tau: \mathfrak{g} \rightarrow \mathfrak{g}$. If such an
       involution $\tau$ is given it descends to an involution $\tau: G \rightarrow G$ and its fixed point Lie subsupergroup
       $G_\mathbb{R}$ is called a real form of $G$. 
\medskip

\noindent  2. An even real form of $\mathfrak{g}$ is an even real Lie subalgebra $\mathfrak{g}_\mathbb{R} \subseteq \mathfrak{g}_{\bar{0}}$
       defined by a $\mathbb{C}$-antilinear automorphism $\tau: \mathfrak{g} \rightarrow \mathfrak{g}$ satisfying $\tau^2(X)
       = (-1)^{\vert X \vert} X$ for homogeneous $X \in \mathfrak{g}$. The automorphism $\tau$ again lifts to an automorphism 
       of $G$ and its fixed point set $G_\mathbb{R}$ is a real form of $G_{\bar{0}}$. 
\medskip

\noindent  3. A real form or even real form $\mathfrak{g}_\mathbb{R}$ is called compact if its even part $\mathfrak{g}_{\bar{0}\mathbb{R}}$
       is a compact Lie algebra.
\end{def1}

\begin{bsp}
 
Let $\mathfrak{g} = \mathfrak{sl}_{n \vert m}(\mathbb{C})$, the subspace of all elements of $\mathfrak{gl}_{n \vert m}(\mathbb{C})$ 
with vanishing super-trace and let $\tau(X) = - \bar{X}^{st}$. It defines an even real form
of $\mathfrak{g}$ with $\mathfrak{g}_\mathbb{R} = \mathfrak{s}(\mathfrak{u}(n) \oplus \mathfrak{u}(m))$. If one instead chooses
$\tau(X) = -i^{\vert X \vert} \bar{X}^t$, then $\tau$ is an actual involution and its fixed point set is 
$\mathfrak{g}_\mathbb{R} = \mathfrak{su}(n \vert m)$.

\end{bsp}

The classification of real forms of complex simple Lie superalgebras was developed in \cite{Kac}, \cite{Par} and \cite{Ser}. 
A table of this classification is given in Chapter 3 and will be used throughout the main text. 

One striking difference
to the classical case is that compact real forms occur very rarely in the supersymmetric case, e.g. if 
$G = \mathrm{Osp}_{n \vert 2m}(\mathbb{C})$ then there are no compact real forms for $n \geq 2, m \geq 1$
as the base $G_{\bar{0}\mathbb{R}}$ of every real form $G_\mathbb{R}$ contains a factor isomorphic to either $Sp_{2m}(\mathbb{R})$
or $SO^*(n)$ if $n$ is even. Compact even real forms on the other hand exist in abundance.

Let $\mathfrak{g}$ be a basic classical simple Lie superalgebra. 
Then, by a theorem due to Kac(see \cite{Kac}), $\mathfrak{g}$ allows an even non-degenerate supersymmetric 
(i.e. $B(X,Y) = (-1)^{\vert X \vert \vert Y \vert} B(Y,X)$ for all homogeneous $X,Y \in \mathfrak{g}$)
 $\mathrm{ad}$-invariant bilinear form $B: \mathfrak{g} \times \mathfrak{g} \rightarrow \mathfrak{g}$. 
Let $\mathfrak{g}_\mathbb{R}$ be a real form of $\mathfrak{g}$, $\theta: \mathfrak{g}_{\bar{0}\mathbb{R}} \rightarrow \mathfrak{g}_{\bar{0}\mathbb{R}}$
the Cartan involution and $\theta^\mathbb{C}: \mathfrak{g}_{\bar{0}} \rightarrow \mathfrak{g}_{\bar{0}}$ its $\mathbb{C}$-linear extension to $\mathfrak{g}_{\bar{0}}$.  

\begin{satz}

Unless $\mathfrak{g}_\mathbb{R} = \mbox{}^0\mathfrak{pq}(n)$ or $\mathfrak{us}\pi(n)$, 
there exists an essentially unique extension of $\theta$ to an order 4 automorphism of $\mathfrak{g}$ such that

\begin{enumerate}
 \item $\theta^2(X) = (-1)^{\vert X \vert} X$ for all homogeneous elements
 \item The bilinear form $B_\theta = B(\cdot, \theta\cdot)$ is symmetric and non-degenerate. 
\end{enumerate}

\noindent The map $\theta:\mathfrak{g} \rightarrow \mathfrak{g}$ is called the Cartan isomorphim of $\mathfrak{g}_\mathbb{R}$
 
\end{satz}

\begin{proof}
 
The respective automorphisms $\theta$ can be extracted from Table 7 in $\cite{Ser}$. They are precisely
those automorphisms from that list which have a compact fixed point set.

\end{proof}

In particular the two maps $\tau$ and $\theta$ commute and their composition $\sigma$ defines a compact even real form
of $\mathfrak{g}$. On the other hand, given an even real form $\mathfrak{g}_\mathbb{R}$ defined by an oder 4 automorphism $\tau$ 
there need not be a $\mathbb{C}$-linear involution $\theta: \mathfrak{g} \rightarrow \mathfrak{g}$ that commutes with $\tau$
and restricts to the Cartan involution on $\mathfrak{g}_\mathbb{R}$. 

\begin{bsp}

Let $\mathfrak{g} = \mathfrak{sl}_{n \vert m}(\mathbb{C})$ and $\mathfrak{g}_\mathbb{R}$ a real form. Then even though $\theta$
restricts to the classical Cartan involution on $\mathfrak{g}_{\bar{0}\mathbb{R}}$, $B_\theta$ is not negative definite.
In fact it is negative definite on $\mathfrak{g}_{\bar{0}\mathbb{R}} \cap \mathfrak{g}_{\bar{0}\bar{0}}$ and positive
definite on $\mathfrak{g}_{\bar{0}\mathbb{R}} \cap \mathfrak{g}_{\bar{1}\bar{1}}$. This is due to the fact that $B_\theta$
incorporates the super-trace rather than the classical trace. 

\end{bsp}

\subsection{Homogeneous superspaces}

This text uses the notion of homogeneous superspace introduced by Kostant in \cite{Kos}.

An action of a Lie supergroup $G$ on a supermanifold $X$ is a morphism $\nu: G \times X \rightarrow X$ satisfying
the usual axioms for group actions.

Suppose $G$ is a Lie supergroup and $H$ is a closed subsupergroup and let $\pi_{\bar{0}}: G_{\bar{0}} \rightarrow G_{\bar{0}}/H_{\bar{0}}$ be the canonical projection. 
Then the quotient is the supermanifold $G/H = (G_{\bar{0}}/H_{\bar{0}}, \mathcal{O}_{G/H})$, where

\[ \mathcal{O}_{G/H}(U) = \{ f \in \mathcal{O}_G(\pi_{\bar{0}}^{-1}U) : R_w f = \langle w, 1_H \rangle f \ \forall w \in \mathcal{O}_H(H_{\bar{0}})^* \} \]

Here $R_w$ is the adjoint operator of the multiplication operator $m_w: \mathcal{O}_H(H_{\bar{0}})^*$ \\ $ \rightarrow \mathcal{O}_H(H_{\bar{0}})^*, u \mapsto uw$ and $1_H$ is the constant function with value 1 on $H$. 

It is shown in \cite{Kos} that this does indeed define the structure of a supermanifold with base $G_{\bar{0}}/H_{\bar{0}}$ and there
is a canonical projection $\pi: G \rightarrow G/H$. 

Moreover, as in the classical case, given an action of $G$ on a supermanifold $X$ one may define the notions of orbits
and stabilizers and obtain canonical isomorphism $G \cdot x \cong G/G_x$ for all $x \in X_{\bar{0}}$.

The tangent space of a split homogeneous superspace  $G/H = (G_{\bar{0}}/H_{\bar{0}},\bigwedge\mathcal{E})$ can be identified with the dual space 
of the fibre of $\mathcal{E}$:

\begin{lemma}[\cite{On}]\label{OnFunctions}

Let $G$ be a complex Lie supergroup and $M = G/H$ a complex split $G$-homogeneous supermanifold. 
Suppose $\mathcal{O}_M \simeq \bigwedge \mathcal{E}$. Then $\mathbb{E}$ is $G_{\bar{0}}$-equivariantly isomorphic to 
$G_{\bar{0}} \times_{H_{\bar{0}}} (\mathfrak{g}_{\bar{1}}/\mathfrak{h}_{\bar{1}})^*$. In particular $(T_xM)_{\bar{1}} \cong \mathbb{E}_x^*$. 

\end{lemma}

This result carries over to open $G_\mathbb{R}$-orbits: Let $i: D = G_\mathbb{R}/L_\mathbb{R} \rightarrow G/H$ be the inclusion.
Then $i^*\mathbb{E} = G_{\bar{0}\mathbb{R}} \times_{L_{\bar{0}\mathbb{R}}} (\mathfrak{g}_{\bar{1}}/\mathfrak{h}_{\bar{1}})^*$ and given that 
$D$ actually has maximal odd dimension there is a canonical isomorphism between $(\mathfrak{g}_{\bar{1}}/\mathfrak{h}_{\bar{1}})^*$
and $(\mathfrak{g}_{\bar{1}\mathbb{R}}/\mathfrak{l}_{\bar{1}\mathbb{R}})^*$. Consequently $\mathcal{E} \vert_D$ is the sheaf of germs of holomorphic sections
of the homogeneous vector bundle $\mathcal{O}(G_{\bar{0}\mathbb{R}} \times_{L_{\bar{0}\mathbb{R}}} (\mathfrak{g}_{\bar{1}\mathbb{R}}/\mathfrak{l}_{\bar{1}\mathbb{R}})^*)$. 

Moreover if $G/H$ is any (possibly non-split) homogeneous superspace, then the superfunctions
on $\mathrm{gr} G/H$ are precisely the sections of the homogeneous bundle $G_{\bar{0}} \times_{H_{\bar{0}}} \bigwedge (\mathfrak{g}_{\bar{1}}/\mathfrak{h}_{\bar{1}})^*$(Theorem 2 in \cite{V}).

\subsection{Parabolic Lie superalgebras, flag supermanifolds and flag domains}\label{OnPSA}

The strong correspondence between parabolic subsuperalgebras of reductive Lie superalgebras and flag spaces was elaborated
in \cite{OnI}. Throughout this text several results from that paper are used without proof. 

A Lie superalgebra $\mathfrak{g}$ is called reductive, if $\mathfrak{g}_{\bar{0}}$ is reductive and the adjoint representation of $\mathfrak{g}_{\bar{0}}$
on $\mathfrak{g}$ is algebraic. A Lie supergroup $G$ is called reductive if $\mathfrak{g} = \mathrm{Lie}(G)$ is reductive.

Assume that $\mathfrak{g}$ is a reductive Lie superalgebra and that $\mathfrak{t}$ is a Cartan subalgebra of $\mathfrak{g}_{\bar{0}}$.
Let $\Sigma(\mathfrak{g},\mathfrak{t}) \subseteq \mathfrak{t}^*$ be the corresponding root system and $\mathfrak{h} = \mathfrak{g}^0$. It is called the
Cartan subalgebra of $\mathfrak{g}$.

\begin{rem}

If $\mathfrak{g}$ is simple, then $\mathfrak{h} = \mathfrak{t}$ unless $\mathfrak{g}$ is of type $Q$. In that case
$\mathfrak{h} = \mathfrak{t} \oplus \Pi \mathfrak{t}$ and all root spaces are $1 \vert 1$-dimensional.  

\end{rem}

Let $x \in \mathfrak{t}(\mathbb{R}) = \{ y \in \mathfrak{t} : \alpha(y) \in \mathbb{R} \ \forall \alpha \in \Sigma \}$ 
and define $\mathfrak{p}(x) = \mathfrak{h} \oplus \sum_{\alpha \in \Sigma, \alpha(x) \geq 0} \mathfrak{g}^\alpha$.

\begin{def1}

A subsuperalgebra $\mathfrak{p} \subseteq \mathfrak{g}$ is called parabolic with repect to $\mathfrak{t}$ 
if it is of the form $\mathfrak{p}(x)$ for some $x \in \mathfrak{t}$. It is called parabolic if it is parabolic
with respect to an arbitrary Cartan subalgebra of $\mathfrak{g}_{\bar{0}}$. A Lie subsupergroup $P \subseteq G$ is called parabolic
if $P_{\bar{0}}$ is a parabolic subgroup of $G_{\bar{0}}$ and $\mathfrak{p} = \mathrm{Lie}(P)$ is a parabolic subsuperalgebra of
$\mathfrak{g} = \mathrm{Lie}(G)$.
 
\end{def1}

In \cite{OnI} the parabolic subalgebras of the non-exceptional simple classical Lie superalgebras are characterized and realized
as stabilizers of flags in standard vector spaces. This characterization is as follows:

\begin{itemize}
 \item If $\mathfrak{g} = \mathfrak{sl}_{n \vert m}(\mathbb{C})$, then its parabolic subalgebras are stabilizers of
       flags of supervector subspaces in $\mathbb{C}^{n \vert m}$.
 \item If $\mathfrak{g} = \mathfrak{osp}_{n \vert 2m}(\mathbb{C})$ and $B$ is the non-deg. super-symmetric bilinear form
       on $\mathbb{C}^{n \vert 2m}$ left invariant by $\mathfrak{g}$, then the parabolic subalgebras are stabilizers
       of flags of $B$-isotropic supervector subspaces of $\mathbb{C}^{n \vert 2m}$.
 \item If $\mathfrak{g} = \mathfrak{p}_n(\mathbb{C})$ and $\omega$ is the non-deg. super-skewsymmetric bilinear form
       on $\mathbb{C}^{n \vert n}$ left invariant by $\mathfrak{g}$, then the parabolic subalgebras are stabilizers
       of flags of $\omega$-isotropic supervector subspaces of $\mathbb{C}^{n \vert n}$.
 \item If $\mathfrak{g} = \mathfrak{q}_n(\mathbb{C})$ and $\Pi : \mathbb{C}^{n \vert n} \rightarrow \mathbb{C}^{n \vert n}$
       is the parity change operator, then the parabolic subalgebras are stabilizers of flags of $\Pi$-invariant supervector
       subspaces of $\mathbb{C}^{n \vert n}$.
\end{itemize}

This characterization motivates the following definition:

\begin{def1}

Let $G$ be a complex reductive Lie supergroup. 

 \begin{enumerate}
  \item A $G$-flag supermanifold is a homogeneous space $Z = G/P$ where $P$ is a parabolic subsupergroup of $G$.
  \item If $G_\mathbb{R}$ is a real form of $G$, then a flag domain in $Z$ is an orbit $D = G_\mathbb{R} \cdot z, z \in Z_{\bar{0}}$
	that is an open subsupermanifold of $Z$.
  \item If $G_\mathbb{R}$ is an even real form of $G$, then a flag domain is an open subamnifold $D = (D_0, i^*\mathcal{O}_Z)$
	where $D_{\bar{0}}$ is an open $G_\mathbb{R}$-orbit in $Z_0$ and $i: D_{\bar{0}} \rightarrow Z_{\bar{0}}$ is the inclusion map.
 \end{enumerate}
\end{def1}

It is necessary to define flag domains of even real forms this way, because, as $G_\mathbb{R}$ is a purely even group,
$G_\mathbb{R}$-orbits in $Z$ always have odd dimension equal to zero. The way they are defined  the classical theory 
always grants the existence of flag domains for even real forms. If $G_\mathbb{R}$ is a real form, then
$G_\mathbb{R}$-orbits  with open base are supermanifolds, which are not purely even. In fact their odd dimension
satisfies $\frac{1}{2} \dim_{\bar{1}} Z \leq \dim_{\bar{1}} G_\mathbb{R} \cdot z \leq \dim_{\bar{1}} Z$.  
Moreover $G_\mathbb{R}$-orbits of maximal odd dimension need not always exist, as will be shown during the discussion of measurability.
Consequently the first thing to check for such a flag domain is whether its odd dimension is maximal. 
In fact it will also turn out that it is sensible to consider only flag domains of even real forms, which satisfy
a condition that is very similar to the conditions for maximal odd dimension in the case of orbits of real forms.

The following chapters will present a discussion on the question of generalizing the results about flag domains exhibited
in the first section to the theory of flag supermanifolds and their flag domains. But before venturing into this discussion
a recollection of the basic definitions and some important results from the theory of supervector bundles is required to provide all
necessary background for the last chapter. 

\subsection{Super vector bundles, differential forms and Berezinians}

The material in this section is reviewed in more detail in \cite{Ma} and \cite{AH}. 

\begin{def1}

A super vector bundle on a supermanifold $X$ is a locally free $\mathbb{Z}/2$-graded $\mathcal{O}_X$-module $\mathcal{E}$. Let $U \subseteq X_{\bar{0}}$ open. 
An element of $\mathcal{E}(U)$ is called a local section of the super vector bundle $\mathcal{E}$.

\end{def1}

\begin{bsp}[Tangent and Cotangent bundles] 

The tangent bundle $\mathcal{T}_X$ of a super manifold is the sheaf $\mathrm{Der} \mathcal{O}$ of derivations of the struture sheaf and the cotangent sheaf $\mathcal{T}_X^*$ is its dual sheaf. 
If $U \subseteq X_{\bar{0}}$ is a coordinate neighbourhood with even coordinates $x_1, \ldots, x_n$ and odd coordinates $\xi_1, \ldots, \xi_m$, then $\mathcal{T}_X(U)$ is isomorphic to the free $\mathcal{O}_X(U)$-module with basis $\frac{\partial}{\partial x_1}, \ldots, \frac{\partial}{\partial x_n},\frac{\partial}{\partial \xi_1},\ldots, \frac{\partial}{\partial \xi_m}$ and  $\Omega_X^1(U)$  is the  free $\mathcal{O}_X(U)$-module
spanned by the dual basis $dx_1, \ldots, dx_n,$ $ d\xi_1, \ldots, d\xi_m$. Analogous to the classical case, sections of the tangent sheaf are called super vector fields and sections of the cotangent bundle are called differential 1-forms. 

Let $\Omega_X^p = \bigwedge^p \Omega_X^1$ be the sheaf of germs of differential $p$-forms. Unlike the even differentials, the odd differentials commute. 
Therefore $\Omega_X^p \neq 0$ for all $p \geq 0$, if the odd dimension of $X$ is bigger than zero. 
However, as in the classical case, the sheaves of differential forms consitute a resolution of the sheaf of locally constant functions



\[ \xymatrix{ 0 \ar[r] & \mathbb{C} \ar[r] & \mathcal{O}_X \ar[r]^\partial & \Omega_X^1 \ar[r]^\partial & \ldots \ar[r]^\partial & \Omega_X^n \ar[r]^\partial & \ldots } \]

and if $\mu: X \rightarrow Y$ is a submersion then there is a relative complex: As in the classical case $\mu$
induces an injective homomorphism $T^*\mu: \Omega_Y^1 \rightarrow \Omega_X^1$. Let $\Omega_\mu^1$ be the quotient sheaf, $\pi: \Omega_X^1 \rightarrow \Omega_\mu^1$ the projection 
and $\Omega_\mu^p = \bigwedge^p \Omega_\mu^1$. The differentials $d$ of the complex $\Omega_X^*$ induce differentials $d_\mu$ on $\Omega_\mu^*$.
As in the classical case, if $\eta$ is a local section of $\Omega_\mu^p$ and $\omega$ is a local section of $\Omega_x^p$ satisfying $\pi(\omega) = \eta$,
then $d_\mu(\eta) = \pi(d\omega)$. The complex of $\mu$-relative differential forms is then a resolution of $\mu^{-1}\mathcal{O}_Y$: 

\[ \xymatrix{ 0 \ar[r] & \mu^{-1}\mathcal{O}_Y \ar[r] & \mathcal{O}_X \ar[r]^{\partial_\mu} & \Omega_\mu^1 \ar[r]^{\partial_\mu} & \ldots \ar[r]^{\partial_\mu} & \Omega_\mu^n \ar[r]^{\partial_\mu} & \ldots } \]

\end{bsp}
Another important super vector bundle is the Berezinian bundle. Berezinian forms are the supersymmetric analogue of volume forms and the basic object for integration on supermanifolds. 
For a super vector space $V = V_{\bar{0}} \oplus V_{\bar{1}}$, the Berezinian module $\mathrm{Ber}(V)$ is constructed as follows:

Consider the symmetric algebra $S(\Pi V \oplus V^*)$. The parity change operator $\Pi: V \rightarrow \Pi V$ defines an odd element of this algebra, so its square is zero and therefore one may consider the complex

\[  \xymatrix{ S(\Pi V \oplus V^*) \ar[r]^{\cdot \Pi} & S(\Pi V \oplus V^*) } \]

The Berezinian module of $V$ is the homology of this complex. If $\mathcal{E}$ is a super vector bundle with fibre isomorphic to $V$ and cocycle $f_{ij}$, then the Berezinian bundle $\mathrm{Ber} \mathcal{E}$ is the super vector bundle
with fibre $\mathrm{Ber}(V)$ and cocycle $\mathrm{Ber}(f_{ij})$, the Berezinian of the endomorphism $f_{ij}$(see \cite{Ma}).
It is a line bundle with the same parity as the odd dimension of $V$.

For a supermanifold $X$ let $\mathrm{Ber}(X) = \mathrm{Ber}(T^*X)$. If $U \subseteq X_{\bar{0}}$ is a coordinate neighbourhood then a basis element of $\mathrm{Ber}(X)(U)$ is $ f dx_1 \wedge \ldots \wedge dx_n \wedge \frac{\partial}{\partial \xi_1} \wedge \ldots \wedge \frac{\partial}{\partial \xi_m} $ for some $f \in \mathcal{O}_X(U)$.
Note that unlike in the classical case, the basic object for integration is not a differential form of any degree.

\subsection{A spectral sequence of Onishchik's and Vishnyakova's}

Given a super vector bundle $\mathcal{E}$ on $X$ there are several ways to associate to $\mathcal{E}$ a $\mathbb{Z}$-graded super vector bundle on $\mathrm{gr} X$(see \cite{OV1} and \cite{OV2}):

Let $\mathcal{J} \subseteq \mathcal{O}_X$ as above. Then $\mathcal{S} = \mathcal{E}/\mathcal{J}\mathcal{E}$
is the sheaf of sections of a $\mathbb{Z}/2$-graded vector bundle on $X_{\bar{0}}$ with the same local bases as $\mathcal{E}$.

Define $\mathcal{E}^{(p)} = \mathcal{J}^p\mathcal{E}, (\mathrm{gr} \mathcal{E})_p = \mathcal{E}^{(p)}/\mathcal{E}^{(p+1)}$ and $\mathrm{gr} \mathcal{E} = \bigoplus_{p \geq 0} (\mathrm{gr} \mathcal{E})_p$. Then $\mathrm{gr}\mathcal{E}$ is a $\mathbb{Z}$-graded super vector bundle on $\mathrm{gr} X$.
Note however that this $\mathbb{Z}$-grading is usually not compatible with the $\mathbb{Z}/2$-grading. 

To obtain a $\mathbb{Z}$-graded super vector bundle whose $\mathbb{Z}$-grading fits in with the $\mathbb{Z}/2$-grading let $^\prime\mathcal{E}^{(p)} = \mathcal{J}^p \mathcal{E}_{\bar{0}} + \mathcal{J}^{p-1}\mathcal{E}_{\bar{1}}$, 
$\tilde{\mathcal{E}}_p = \ ^\prime\mathcal{E}^{(p)}/ \ ^\prime\mathcal{E}^{(p+1)}$ and $\tilde{\mathcal{E}} = \bigoplus_{p \geq 0} \tilde{\mathcal{E}}_p$. 

Both of these associated graded vector bundles can be used to construct a spectral sequence converging to the cohomology of $\mathcal{E}$. 
This construction was worked out in detail in \cite{OV1} and \cite{OV2}. Here only the relevant results will be reviewed.

The first step towards the construction of the spectral sequence is the characterization of non-isomorphic super vector bundles
with a fixed retract. This characterization is similar to the characterization of non-isomorphic supermanifold structures
on a given complex manifold and uses quasi-derivations.

\begin{def1}

\begin{enumerate}
 \item Let $X$ be a complex supermanifold, $U \subseteq X_{\bar{0}}$ open, $\mathcal{E}$ a super vector bundle on $X$ and $\Gamma$ an even vector field on $U$. 
A sheaf homomorphism $A: \mathcal{E}\vert_U \rightarrow \mathcal{E}\vert_U$ is called a $\Gamma$-derivation, if $A(fv) = \Gamma(f)v + f A(v)$
for all $f \in \mathcal{O}_X(U), v \in \mathcal{E}(U)$. $A$ is called a quasi-derivation if it is a $\Gamma$-derivation for
some even vector field $\Gamma$. The sheaf of quasi-derivations is denoted by $\mathcal{QD}er \mathcal{E}$. 
 \item Let $X, U$ and $\mathcal{E}$ as before and let $\Psi$ be an automorphism of $\mathcal{O}(U)$. Then a sheaf homomorphism
$a: \mathcal{E} \rightarrow \mathcal{E}$ is called a $\Psi$-morphism if $a(fv) = \Psi(f) a(v)$ for all $f \in \mathcal{O}(U),
 v \in \mathcal{E}(U)$. $a$ is called a quasi-automorphism, if $a$ is an automorphism and $\Psi$-morphism for some automorphism $\Psi$.
The sheaf of quasi-automorphisms is denoted $\mathcal{QA}ut \mathcal{E}$.
\end{enumerate}
 
\end{def1}

Note that these sheaves are sheaves of Lie superalgebras and groups respectively. They both allow natural double filtrations 
as follows:

\[ \mathcal{QA}ut_{(p)(q)} \mathcal{E}(U) = \{ a \in \mathcal{QA}ut \mathcal{E}(U) : a(v) \equiv v \ \mathrm{mod} \mathcal{E}^{(p)}(U),   \] 
\[ \Psi(f) \equiv f \ \mathrm{mod} \mathcal{J}^q(U) \textnormal{ for all} \ f \in \mathcal{O}(U), v \in \mathcal{E}(U) \}, \ p,q \geq 0 \]

\[ \mathcal{QD}er_{(p)(q)} \mathcal{E}(U) = \{ A \in \mathcal{QD}er \mathcal{E}(U) : A(\mathcal{E}^{(r)}(U)) \subseteq \mathcal{E}^{(p + r)}(U),   \] 
\[ \Gamma(\mathcal{J}^s(U)) \subseteq \mathcal{J}^{q + s}(U) \textnormal{ for all} \ r,s \in \mathbb{Z} \}, \ p,q \geq 0 \]

Moreover for a graded super vector bundle $\mathrm{gr} \mathcal{E}$ consider the subsheaves
$\mathcal{QA}ut_0 \mathcal{E} \subseteq \mathcal{QA}ut \mathcal{E}$ and $\mathcal{QD}er_{k,k} \mathcal{E} \subseteq \mathcal{QD}er_{(k)(k)} \mathcal{E}$,
which consist of all quasi-automorphisms and quasi-derivations which are compatible with the $\mathbb{Z}$-gradings of
$\mathrm{gr} \mathcal{O}$ and $\mathrm{gr} \mathcal{E}$.

Also note that there is a mapping
\[\exp: \mathcal{QD}er_{(p)(q)} \mathcal{E} \rightarrow \mathcal{QA}ut_{(p)(q)} \mathcal{E}\]

which is an isomorphism of sheaves for $p = 1, q = 2$.
The main characterization result is the following:

\begin{satz}[Theorem 2 in \cite{OV1}]
 
Let $(X, \mathcal{O}_{gr})$ be a split supermanifold, $\mathcal{S}$ be a sheaf of $\mathbb{Z}/2$-graded $\mathcal{F}$-modules
and $\mathcal{E}_{gr} = \mathcal{O}_{gr} \otimes_{\mathcal{F}} \mathcal{S}$. Then there is a bijection between the sets
$\{ [\mathcal{E}] : \mathrm{gr} \mathcal{O} = \mathcal{O}_{gr}, \mathrm{gr} \mathcal{E} = \mathcal{E}_{gr} \}$ and \\ 
$H^1(X, \mathcal{QA}ut_{(1),(2)} \mathrm{gr} \mathcal{E})/H^0(X, \mathcal{QA}ut_0 \mathrm{gr} \mathcal{E})$.

\end{satz}

Let $\mathcal{U} = \{U_i , i \in I\}$ be an open Stein cover of $X_{\bar{0}}$ such that each $U_i$ is a coordinate neighbourhood and 
$\mathcal{E}$ is trivializable on $U_i$. Then $\mathrm{gr} \mathcal{E}$  is also trivializable
on $U_i$. Let $\tau_{ij}$ and $\mathrm{gr} \tau_{ij}$ be the respective cocycles. Then the theorem states
that $\tau_{ij}$ is obtained from $\mathrm{gr} \tau_{ij}$ by means of a quasi-automorphism of $\mathrm{gr} \mathcal{E}(U_{ij})$
and that its quasi-isomorphism class is unique up to a natural action of the group $H^0(X, \mathcal{QA}ut_0 \mathrm{gr} \mathcal{E})$.
Also note that this quasi-automorphism is actually a $\Psi$-automorphism for $\Psi$ a representative
of the class \\ $[\Psi] \in H^1(X, \mathcal{A}ut_{(2)} \mathrm{gr} \mathcal{O})/H^0(X, \mathcal{A}ut_0 \mathrm{gr} \mathcal{O})$
which correpsonds to the supermanifold structure $\mathcal{O}$.

To obtain a spectral sequence relating the cohomology of $\mathcal{E}$ and $\mathrm{gr} \mathcal{E}$
consider the Stein cover $\mathcal{U}$ and the $\check{\textnormal{C}}$ech complex $C^*(\mathcal{U}, \mathcal{E})$. 
It is a filtered complex by virtue of the filtration of $\mathcal{E}$. This also gives rise to a filtration
of $H^*(X, \mathcal{E})$ and therefore L\'{e}ray's theorem (see chapter 5 in \cite{Wei}) yields a spectral sequence $E$ with

\[ E_0^{p,q} = C^{p+q}(X, \mathrm{gr} \mathcal{E}^{(p)}) , \ E_1^{p,q} = H^{p+q}(X, \mathrm{gr} \mathcal{E}^{(p)}) \] 

and converging to $E_\infty^{p,q} = \mathrm{gr}_p H^{p+q}(X, \mathcal{E})$. In particular, if $X$ is compact then
\[ \dim H^k(X, \mathcal{E}) = \sum_{p + q = k} \dim E_\infty^{p,q} \]

This result will be of great use later on to obtain certain vanishing results for the cohomology
of super vector bundles.

Furthermore Vishnyakova and Onishchik also describe the first non-vanishing differential in the spectral sequence. 
To this end consider the following map:

\[ \mu_k: \mathcal{QA}ut_{(k),(2)} \mathrm{gr} \mathcal{E} \rightarrow \mathcal{QD}er_{k,k} \mathcal{E} , 
a \mapsto \bigoplus_{q \in \mathbb{Z}} \mathrm{pr}_{k+q} \circ A \circ \mathrm{pr}_q \]

where $a = \exp(A)$ and $\mathrm{pr}_q: \mathrm{gr}\mathcal{E} \rightarrow (\mathrm{gr}\mathcal{E})_p$ is the canonical projection.

Let $a$ be a representative of the class in $H^1(X,\mathcal{QA}ut_{(1),(2)} \mathrm{gr} \mathcal{E})$ corresponding
to $\mathcal{E}$ and let $k = \max \{ l \in \mathbb{N} : a \in \mathcal{QA}ut_{(k),(2)} \mathrm{gr} \mathcal{E} \}$.

\begin{satz}[Theorem 7 in \cite{OV1}]
 
Let $a$ be a representative of the class in $H^1(X,\mathcal{QA}ut_{(1),(2)} \mathrm{gr} \mathcal{E})$ corresponding
to $\mathcal{E}$ and let $k = \max \{ l \in \mathbb{N} : a \in \mathcal{QA}ut_{(k),(2)} \mathrm{gr} \mathcal{E} \}$.
Let $d_r : E_r \rightarrow E_r$ be the differentials of the spectral sequence constructed above. Then $d_r = 0$
for $1 \leq r < k$ and $d_k = \mu_k(a)$.

\end{satz}

If one works with the split super vector bundle $\tilde{\mathcal{E}}$ instead of $\mathrm{gr}\mathcal{E}$ these
results can be slightly improved as $\mathcal{QA}ut_{(2p-1)(q)}\tilde{\mathcal{E}} = \mathcal{QA}ut_{(2p)(q)}\tilde{\mathcal{E}}$
for all $p \geq 0$ and therefore the distinct quasi-isomorphism classes of super vector bundles with retract $\tilde{\mathcal{E}}$
are parametrized by $H^1(X, \mathcal{QA}ut_{(2),(2)} \tilde{\mathcal{E}})/H^0(X, \mathcal{QA}ut_0 \tilde{\mathcal{E}})$.

\begin{bsp}
 
Let $X = \mathrm{Gr}_{1 \vert 1}(\mathbb{C}^{2 \vert 2})$. It is a non-split supermanifold with base $\mathbb{P}^1 \times \mathbb{P}^1$.
Let $\{U_0,U_1\}$ be the standard open cover of $\mathbb{P}^1$ with respective coordinates $z$ and $\zeta$ and 
$ \{ V_1 = U_0 \times U_0, V_2 = U_0 \times U_1, V_3 = U_1 \times U_0, V_4 = U_1 \times U_1 \}$ be the corresponding product 
cover of $X_0$. 

The split model $(X_{\bar{0}}, \mathrm{gr} \mathcal{O})$ is given by the vector bundle $2 \mathcal{F}(-1,-1)$. Let $\xi_i, \eta_i$
be the bases of $2 \mathcal{F}(-1,-1)\vert_{V_i}$. Then the cocyle $\omega$ of derivations defining
the non-split structure $\mathcal{O}$ is given by

\[ \omega_{21} = \zeta_2^{-1} \xi_2\eta_2 \frac{\partial}{\partial z_1}, \omega_{31} = - \zeta_1^{-1} \xi_3\eta_3 \frac{\partial}{\partial z_2}, \]
\[ \omega_{41} = - \zeta_1^{-2} \zeta_2^{-1} \xi_4\eta_4 \frac{\partial}{\partial z_1} + \zeta_1^{-1} \zeta_2^{-2} \xi_4 \eta_4 \frac{\partial}{\partial z_2} \]

Especially the transition functions on $V_{14}$ are given by

\[ \tau_{41}(z_1) = \zeta_1^{-1} + \zeta_1^{-2}\zeta_2^{-1} \xi_4\eta_4 = \mathrm{gr} \ \tau_{41}(z_1) + \omega_{41}(z_1) \]
\[ \tau_{41}(z_2) = \zeta_2^{-1} - \zeta_2^{-2}\zeta_1^{-1} \xi_4\eta_4 = \mathrm{gr} \ \tau_{41}(z_2) + \omega_{41}(z_2) \]
\[ \tau_{41}(\xi_1) = -\zeta_1^{-1}\zeta_2^{-1} \xi_4 = \mathrm{gr} \ \tau_{41}(\xi_1) \]
\[ \tau_{41}(\eta_1) = -\zeta_1^{-1}\zeta_2^{-1} \eta_4 = \mathrm{gr} \ \tau_{41}(\eta_1) \]

so the automorphism of $\mathrm{gr}\mathcal{O}$ defining the non-split structure $\mathcal{O}$ is $\Psi = \exp \omega = \mathrm{id} + \omega$

Let $\mathcal{E} = \Omega_X^1$. Its basis in $V_1$ is $dz_1,dz_2,d\xi_1,d\eta_1$ and the transition functions on $V_{41}$ are

\[ \tau_{41}(dz_1) = d(\zeta_1^{-1} + \zeta_1^{-2}\zeta_2^{-1} \xi_4\eta_4) =\] 
\[ [- \zeta_1^{-2} d\zeta_1] + [- 2\zeta_1^{-3}\zeta_2^{-1}\xi_4\eta_4 d\zeta_1 - \zeta_1^{-2} \zeta_2^{-2} \xi_4 \eta_4 d\zeta_2 
+ \zeta_1^{-2}\zeta_2^{-1} \eta_4 d\xi_4 + \zeta_1^{-2}\zeta_2^{-1} \xi_4 d\eta_4] = \]
\[ \mathrm{gr} \ \tau_{41}(dz_1) + A(dz_1) \]

\[ \tau_{41}(dz_2) = d(\zeta_2^{-1} - \zeta_2^{-2}\zeta_1^{-1} \xi_4\eta_4) =\] 
\[ [- \zeta_2^{-2} d\zeta_2] + [2\zeta_2^{-3}\zeta_1^{-1}\xi_4\eta_4 d\zeta_2 + \zeta_2^{-2} \zeta_1^{-2} \xi_4 \eta_4 d\zeta_1 
- \zeta_2^{-2}\zeta_1^{-1} \eta_4 d\xi_4 - \zeta_2^{-2}\zeta_1^{-1} \xi_4 d\eta_4] = \]
\[ \mathrm{gr} \ \tau_{41}(dz_2) + A(dz_2) \]

\[ \tau_{41}(d\xi_1) = -d(\zeta_1^{-1}\zeta_2^{-1} \xi_4) = [- \zeta_1^{-1}\zeta_2^{-1} d\xi_4] + [ \zeta_1^{-2}\zeta_2^{-1} \xi_4 d\zeta_1 + \zeta_1^{-1}\zeta_2^{-2} \xi_4 d\zeta_2] \]
\[ = \mathrm{gr} \ \tau_{41}(d\xi_1) + A(d\xi_1)\]

\[ \tau_{41}(d\eta_1) = -d(\zeta_1^{-1}\zeta_2^{-1} \eta_4) = [- \zeta_1^{-1}\zeta_2^{-1} d\eta_4] + [\zeta_1^{-2}\zeta_2^{-1} \eta_4 d\zeta_1 + \zeta_1^{-1}\zeta_2^{-2} \eta_4 d\zeta_2] \]
\[ = \mathrm{gr} \ \tau_{41}(d\eta_1) + A(d\eta_1)\]

The first bracket in each line is the respective transition function for $\mathrm{gr} \Omega^1$ and the second bracket is
the contribution of the quasi-derivation $A_{14}$. Analogous formulae can be obtained for the other two-fold intersections $V_{ij}$,
but for the further computation this component of the cocycle $A$ is sufficient. Note that $a = \exp A = \mathrm{id} + A$
is a quasi-automorphism corresponding to the quasi-isomorphism class of $\Omega^1$.

Now consider $H^*(X_{\bar{0}}, \mathrm{gr} \Omega^1)$. $(\mathrm{gr} \Omega^1)_0$ is the direct sum of line bundles isomorphic to
$\mathcal{F}(-2,0), 2\mathcal{F}(-1,-1)$ and $\mathcal{F}(0,-2)$ with respective generators $dz_1,d\xi_1,d\eta_1$ and $dz_2$.
Moreover $(\mathrm{gr} \Omega^1)_1 \cong (\mathrm{gr} \Omega^1)_0 \otimes 2\mathcal{F}(-1,-1)$ and $(\mathrm{gr} \Omega^1)_2 =
(\mathrm{gr} \Omega^1)_0 \otimes \mathcal{F}(-2,-2)$. This yields the following non-vanishing cohomology:

\[ H^1(X_{\bar{0}}, (\mathrm{gr} \Omega^1)_0) \cong \mathbb{C} \{z_1^{-1} dz_1, z_2^{-1} dz_2\} \]
\[ H^2(X_{\bar{0}}, (\mathrm{gr} \Omega^1)_0) = \mathbb{C}\left\lbrace \frac{\xi_1 d\xi_1}{z_1^{-1}z_2^{-1}}, \frac{\xi_1 d\eta_1}{z_1^{-1}z_2^{-1}}, \frac{\eta_1 d\xi_1}{z_1^{-1}z_2^{-1}},\frac{\eta_1 d\eta_1}{z_1^{-1}z_2^{-1}}\right\rbrace \]
\[ H^2(X_{\bar{0}}, (\mathrm{gr} \Omega^1)_2) \cong \mathbb{C}^{14}\]

Now, as $a$ yields the transition functions of $\Omega^1$, it is necessarily \\ a $\Psi$-automorphism, hence $A$ is an $\omega$-derivation
and one may compute:

\[ A(z_1^{-1} dz_1) = \omega_{41}(z_1^{-1}) \mathrm{gr} \tau_{41}(dz_1) + \mathrm{gr}\tau_{41}(z_1^{-1})A(dz_1) = \]
\[  \zeta_1^{-2} \zeta_2^{-1} \xi_4\eta_4 \mathrm{gr} \tau_{41}\left(\frac{\partial}{\partial z_1}(z_1^{-1})\right) \zeta_1^{-2} d\zeta_1 + \]
\[\zeta_1 [- 2\zeta_1^{-3}\zeta_2^{-1}\xi_4\eta_4 d\zeta_1 - \zeta_1^{-2} \zeta_2^{-2} \xi_4 \eta_4 d\zeta_2 
+ \zeta_1^{-2}\zeta_2^{-1} \eta_4 d\xi_4 + \zeta_1^{-2}\zeta_2^{-1} \xi_4 d\eta_4] =  \]
\[ [-3 \zeta_1^{-2}\zeta_2^{-1}\xi_4\eta_4d\zeta_1 - \zeta_1^{-1}\zeta_2^{-2}\xi_4\eta_4d\zeta_2] + 
[\zeta_1^{-1}\zeta_2^{-1}\eta_4d\xi_4 + \zeta_1^{-1}\zeta_2^{-1}\xi_4d\eta_4] = \]
\[ \mu_2(a)(z_1^{-1} dz_1) + \mu_1(a)(z_1^{-1} dz_1) \]

\noindent An analogous computation yields $\mu_1(a)(z_2^{-2} dz_2) = -\zeta_1^{-1}\zeta_2^{-1}\eta_4d\xi_4$ \\ $+ \zeta_1^{-1}\zeta_2^{-1}\xi_4d\eta_4$.
On the other intersections $V_{12}$ and $V_{13}$ one obtains the following results:

\[ \mu_1(a)(z_1^{-1} dz_1)_{12} =  -z_1^{-1}\zeta_2^{-1}\xi_2d\eta_2 -z_1^{-1}\zeta_2^{-1}\eta_2d\xi_2 , \mu_1(a)(z_2^{-1} dz_2)_{12} = 0 \]
\[ \mu_1(a)(z_1^{-1} dz_1)_{13} = 0 , \mu_1(a)(z_2^{-1} dz_2)_{13} =  z_2^{-1}\zeta_1^{-1}\xi_2d\eta_2 + z_2^{-1}\zeta_1^{-1}\eta_2d\xi_2   \]

The final result is that $E_1^{0,1} = H^1(X_{\bar{0}}, (\mathrm{gr} \Omega^1)_0)$ does not survive to the $E_2$-page. Consequently
the spectral sequence collapses at the $E_2$-page to yield

\[ H^0(X_{\bar{0}}, \Omega^1) = H^1(X_{\bar{0}}, \Omega^1) = 0 , H^2(X_{\bar{0}}, \Omega^1) \cong \mathbb{C}^{16} \]

In particular the cohomology of $\Omega^1$ vanishes below top degree, even though the cohomology of $\mathrm{gr} \Omega^1$ does not.

\end{bsp}

\chapter{Flag domains I: Classification and Measurability}

The main problems adressed in this chapter are the questions of maximal odd dimension of an orbit of a real form $G_\mathbb{R}$ of a complex reductive Lie supergroup $G$ in a $G$-flag supermanifold $Z$ and its measurability. Clearly a $G_\mathbb{R}$-orbit $D$ which is open in the categorical sense needs to have an open base $D_{\bar{0}}$. However, the odd dimension of $D$ need not be maximal in general. As in the classical case there is a codimension formula which characterizes the odd dimension of $D$ completely in terms of root combinatorics. As the conditions for maximal odd dimension are very strongly linked to to those for measurability, both these questions are discussed together in this chapter. 

In analogy with the classical case a flag domain is measurable if it possesses an invariant Berezinian density, the natural generalization of an invariant volume element. The problem of its existence can be stated in terms of the isotropy representation at the neutral point which is a representation of a classical Lie group on a Lie superalgebra (Theorem 4.13 in \cite{AH}). Therefore the solution of the problem amounts to classical representation theory. 

As it turns out measurability of a flag domain $D$ does not always imply measurability of its base $D_{\bar{0}}$. This motivates the introduction of two notions of weak and strong measurability.

Strong measurabilty is then characterized in root-theoretic terms analogous to the characterization of measurable flag domains in the classical case due to J. A. Wolf.

After this the measurable flag domains are classified case by case.

The conditions for maximal odd dimension and weak or strong measurability are given in most cases by three symmetry conditions which represent symmetries of the extended Dynkin diagram of type $A(n,n)$. 

Furthermore the characterization results are extended to flag domains of even real forms and the (strongly or weakly) measurable ones are classified by identifying the actions of the defining automorphisms $\tau$ of real forms and even real forms on the root system.

All classification results are summarized in two tables at the end of the chapter.  Most parts of this chapter coincide with the pre-publication in \cite{G}.

\section{Flag types and symmetry conditions}

In this chapter $G$ is always a simple complex Lie supergroup and $P$ is a parabolic subgroup.
So $Z = G/P$ is a flag supermanifold as defined in Chapter 1. By virtue of the 1-to-1 correspondence between conjugacy classes
of parabolic subalgebras and flag types (see section \ref{OnPSA}), $Z = Z(\delta)$ is determined up to ismorphism by the dimension sequence(or flag type)

\[ \delta = 0 \vert 0 < d_0^1 \vert d_1^1 < \ldots < d_0^k \vert d_1^k < n \vert m \]

\noindent and a flag of type $\delta$ is a sequence

\[ 0 \leq V_1 \leq \ldots \leq V_k \leq \mathbb{C}^{n \vert m} \]

\noindent of graded subspaces satisfying $\dim V_i = d_0^i \vert d_1^i$. 
Given a dimension sequence $\delta$, the base $Z(\delta)_{\bar{0}}$ is the set of all flags of type $\delta$.
By virtue of the above mentioned 1-1 correspondence there is, for every given flag $z \in Z(\delta)_{\bar{0}}$, a unique parabolic
subalgebra $\mathfrak{p} \leq \mathfrak{g}$, maximal among those stabilizing $z$, and the conjugacy class of $\mathfrak{p}$ in $\mathfrak{g}$
depends only on $\delta$. Therefore we may define $Z(\delta) =  (Z(\delta)_{\bar{0}} , \varphi^* \mathcal{O}_{G/P})$, where
$\varphi : Z(\delta)_{\bar{0}} \rightarrow^\simeq G_{\bar{0}}/P_{\bar{0}}$ is the canonical isomorphism. The identification of $Z(\delta)$
with $G/P$ will be used implicitly from now on. 

Recall that a Berezinian form is a global section of the Berezinian bundle $\Ber(X) = \Ber(T^*X)$ and the supersymmetric analogue
of a volume form. 
Measurable flag domains are defined in analogy with the classical case:

\begin{def1}

A flag domain $D$ is called measurable if it allows a $G_\mathbb{R}$-invariant Berezinian form.

\end{def1}

The following three symmetry conditions are central to the characterization of measurability in the supersymemtric setting:

\begin{def1}

Let $\delta$ be a dimension sequence. Then:

\begin{enumerate}
 \item $\delta$ is even-symmetric, if $d_0 \vert d_1 \in \delta$, if and only if $(n- d_0) \vert (m - d_1) \in \delta$
 \item $\delta$ is odd-symmetric, if $n=m$ and $d_0 \vert d_1 \in \delta$, if and only if $(n-d_1) \vert (n-d_0) \in \delta$
 \item $\delta$ is $\Pi$-symmetric , if $n = m$ and $d_0 = d_1$ for all $d_0 \vert d_1 \in \delta$.
\end{enumerate}
\end{def1}
 
Note that in the first and second case, as the product ordering on $\{0, \ldots, n\} \times \{0, \ldots, m\}$ is not total, unless $mn = 0$,
not every given dimension sequence can be enlarged to a symmetric dimension sequence. The sequence $\delta$ is called symmetrizable if it can be enlarged to
a symmetric dimension sequence.

\begin{rem}
The three symmetry conditions are related to symmetries of the extended Dynkin diagram of type $A(n,m)$  in the following way:
Choose the standard Borel subalgebra corresponding to the flag type $ \delta = 0 \vert  0 < 1 \vert 0  < \ldots < n \vert 0 < \ldots < n \vert m $. The extended Dynkin diagram contains a vertex for each simple root and an additional vertex for the lowest root. The rules for edges are the same as for the usual Dynkin diagram.  The underlying graph of the extended Dynkin diagram of $A(n,m)$  is therefore a cycle on $n+m+2$ vertices, two of which correspond to odd roots. Removing these two vertices
yields the disjoint union of the Dynkin diagrams $A_n$ and $A_m$. Now assume the vertices are labelled by the simple roots and a parabolic subalgebra $\mathfrak{p}$ of $A(n,m)$ is given. A diagram automorphism of $A(n,m)$ must either fix the two odd roots or interchange them. Thus there are three possible forms of automorphisms: a reflection $r_{\bar{0}}$
fixing one or two even roots, a reflection $r_{\bar{1}}$ fixing the two odd roots and the antipodal map $s$. The last two of these require $m = n$.

Let $J$ be the set of simple roots such that $\mathfrak{g}^{-\alpha} \subseteq \mathfrak{p}$ and let $Z(\delta) = G/P$. Then $\delta$  is even-symmetric if and only if $r_{\bar{0}}(J) = J$.
It is odd-symmetric if and only if $r_{\bar{1}}(J) = J$ and it is $\Pi$-symmetric, if $s(J) = J$.  
\end{rem}

It turns out that, unlike in the classical case, measurability is not equivalent to $P \cap \tau P$ being complex reductive:

\begin{bsp}

Let $n > 2$, $G = PSL_{n \vert n}(\mathbb{C})$, $G_\mathbb{R} = PSL_{n \vert n}(\mathbb{R})$ and $Z = \mathrm{Gr}_{1 \vert 1}(\mathbb{C}^{n \vert n})$, i.e. $Z = Z(\delta)$ for
$\delta = 0 \vert 0 < 1 \vert 1 < n \vert n$.
Then $D_{\bar{0}}$ is the product of two copies of the open $SL_n(\mathbb{R})$-orbit in $\mathbb{P}^{n-1}(\mathbb{C})$, which is not measurable. 
In a suitable basis of $\mathbb{C}^{n \vert n}$, the involution $\tau$ defining the real form acts by $\tau(X) = A\overline{X}A$, where $A$ is the unit antidiagonal matrix. Therefore

\[ \mathfrak{p} \cap \tau\mathfrak{p} = \left\lbrace 
\begin{pmatrix} 
a_1 & v_1^T & 0 & \vline & a_2 & v_2^T & 0 \\ 
0 & X_1 & 0 & \vline & 0 & X_2 & 0 \\
0 & w_1^T & b_1 & \vline & 0 & w_2^T & b_2 \\ \hline
a_3 & v_3^T & 0 & \vline & a_4 & v_4^T & 0 \\ 
0 & X_3 & 0 & \vline & 0 & X_4 & 0 \\
0 & w_3^T & b_3 & \vline & 0 & w_4^T & b_4
\end{pmatrix} \in \mathfrak{g} : \begin{matrix} a_i,b_i \in \mathbb{C} \\ v_i,w_i \in \mathbb{C}^{n-2} \\ X_i \in \mathfrak{gl}_{n-2}(\mathbb{C})  \end{matrix} \right\rbrace \]

Let $\mathfrak{h}$ be the Cartan subalgebra of diagonal matrices. Then $\mathfrak{h}^*$  is generated by the functionals $x_i - x_j, x_i - y_j$ and $y_i - y_j$, where
$x_i$ is the $i^{th}$ diagonal entry of the upper left block and $y_j$ is the $j^{th}$ diagonal entry of the lower right block. So the root system $\Sigma(\mathfrak{g},\mathfrak{h})$
is given by 

\[ \Sigma(\mathfrak{g},\mathfrak{h}) = \{ x_i - x_j \vert 1 \leq i,j \leq n , i \neq j \} \cup  \{ y_i - y_j \vert 1 \leq i,j \leq n , i \neq j \}\]
\[ \cup  \{ \pm x_i \mp y_j \vert 1 \leq i,j \leq n \} \] 

The root spaces contributing to the quotient $\mathfrak{g}/(\mathfrak{p} \cap \tau\mathfrak{p})$
are precisely the ones corresponding to the roots $(x_1 - x_j),(x_1 - y_j),(y_1-x_j),(y_1-y_j) (j \neq 1)$ and
$(x_n - x_j),(x_n - y_j),(y_n-x_j),(y_n-y_j) (j \neq n)$. 
As their graded sum is zero, $\mathfrak{p}$ acts on $\mathfrak{g}/\mathfrak{p}$ and therefore on $(\mathfrak{g}/\mathfrak{p})^*$  with trace zero. Consequently, the
isotropy action of $P_{\bar{0}\mathbb{R}}$ on  $\Ber(\mathfrak{g}/\mathfrak{p})^*$ is trivial. This implies
the existence of a $G_\mathbb{R}$-invariant Berezinian form on $D$.

To sum it up, $D$ carries a $G_\mathbb{R}$-invariant Berezinian form, even though $P_{\bar{0}} \cap \tau P_{\bar{0}}$ is not reductive and therefore $D_{\bar{0}}$
is not measurable according to Theorem \ref{WMeas}.
 
\end{bsp}

\section{Measurable flag domains}

The fact that there are flag domains in the supercase which allow an invariant Berezinian, but whose real stabiliser subgroups are not reductive,
shows that the characterization theorem from \cite{W} does not generalize verbatim. This fact motivates the introduction of two different notions
of measurability. The first of these is the naive generalization and the other one is a notion designed to fulfill the requirements of a characterization theorem
analogous to the one given in \cite{W}. 

\begin{def1}
 
Let $G$ be a complex reductive Lie supergroup, $G_\mathbb{R}$ a real form of $G$, $Z = G/P$ a flag supermanifold
and $D \subseteq Z$ an open $G_\mathbb{R}$-orbit. Then:

\begin{enumerate}
 \item $D$ is strongly measurable, if $D$ carries a $G_\mathbb{R}$-invariant Berezinian density and $D_{\bar{0}}$ is measurable.
 \item $D$ is weakly measurable, if $D$ carries a $G_\mathbb{R}$-invariant Berezinian density and $D_{\bar{0}}$ is not necessarily measurable.
\end{enumerate}

\end{def1}

Starting with a real supergroup orbit $D \subseteq Z$ such that $D_{\bar{0}}$ is open there are three things to check:

\begin{itemize}
 \item Is the odd dimension of $D$ maximal? 
 \item Is $D$ strongly measurable?
 \item Is $D$ weakly measurable?
\end{itemize}

From now on let $G$ be a classical complex reductive Lie supergroup, $G_\mathbb{R} = \mathrm{Fix}(\tau)$ a real form of $G$, $P$ a parabolic subsupergroup of $G$,
$\mathfrak{g}, \mathfrak{g}_\mathbb{R}$ and $\mathfrak{p}$ the respective Lie superalgebras, $\mathfrak{h}$ a Cartan subalgebra of $\mathfrak{g}$,
$\Sigma = \Sigma(\mathfrak{g},\mathfrak{h})$ a root system and $\Phi^r = \{ \alpha \in \Sigma : \mathfrak{g}^\alpha \oplus \mathfrak{g}^{-\alpha} \in \mathfrak{p} \},
 \Phi^n = \{ \alpha \in \Sigma^+ : \mathfrak{g}^{-\alpha} \not\in \mathfrak{p} \}, \Phi^c = \{ \alpha \in \Sigma^- : \mathfrak{g}^\alpha \not\in \mathfrak{p} \}$.

A useful criterion for maximal odd dimension is given by the following codimension formula:

\begin{satz}
 
Let $Z = G/P$ be a flag supermanifold, $M \subseteq Z$ a real group orbit. Then

\[ \mathrm{codim}_Z(M) = \vert \Phi^n \cap \tau\Phi^n \vert \] 

\noindent Especially, $M$ is open, if and only if $\Phi^n \cap \tau \Phi^n$ is empty.

\end{satz}

\begin{proof}
This codimension formula is a generalization of Theorem 2.12 in \cite{W} and the proof is analogous to the one given there:

The main point is that $\mathfrak{p}$ decomposes as $\mathfrak{p} = \mathfrak{p}^r \oplus \mathfrak{p}^n$, where $\mathfrak{p}^r$ and $\mathfrak{p}^n$ are respectively the reductive part and the nilpotent radical of $\mathfrak{p}$. Then $\mathfrak{p} \cap \tau\mathfrak{p} = \mathfrak{p}^r \cap \tau\mathfrak{p}^r \oplus \mathfrak{p}^n \cap \tau\mathfrak{p}^r \oplus \mathfrak{p}^r \cap \tau\mathfrak{p}^n \oplus \mathfrak{p}^n \cap \tau\mathfrak{p}^n$ and $(\mathfrak{g}_\mathbb{R} \cap \mathfrak{p})^\mathbb{C} =  \mathfrak{p} \cap \tau \mathfrak{p}$.

This implies $\dim_\mathbb{R} (\mathfrak{g}_\mathbb{R} \cap \mathfrak{p}) = \dim_\mathbb{C}  \mathfrak{p} \cap \tau \mathfrak{p} = \dim_\mathbb{C}( \mathfrak{p}^r \cap \tau\mathfrak{p}^r \oplus \mathfrak{p}^n \cap \tau\mathfrak{p}^r \oplus \mathfrak{p}^r \cap \tau\mathfrak{p}^n \oplus \mathfrak{p}^n \cap \tau\mathfrak{p}^n) = \dim_\mathbb{C}  \mathfrak{p}^r \cap \tau\mathfrak{p}^r + \dim_\mathbb{C}  \mathfrak{p}^n \cap \tau\mathfrak{p}^r + \dim_\mathbb{C} \mathfrak{p}^r \cap \tau\mathfrak{p}^n + \dim_\mathbb{C}  \mathfrak{p}^n \cap \tau\mathfrak{p}^n = \dim_\mathbb{C} \mathfrak{p}^r + \dim_\mathbb{C}  \mathfrak{p}^n \cap \tau\mathfrak{p}^n$.

The comdimension of $M$ in $Z$ is then $2 \dim_\mathbb{R} \mathfrak{g}_\mathbb{R} - 2 \dim_\mathbb{C} \mathfrak{p} - \dim_\mathbb{R} \mathfrak{g}_\mathbb{R} + \dim_\mathbb{C} \mathfrak{p}^r + \vert \Phi^n \cap \tau\Phi^n \vert$
and it turns out that all terms except for the last cancel out.
\end{proof}

Strong measurability is characterized by the following theorem which is analogous to the characterization of measurability 
in the classical setting (see Theorem 6.1.(1a);(2a)-(2d) and 6.7. in \cite{W}).  Weak measurability occurs in some exceptional cases 
where certain symmetries of the root system lead to the cancellation of even and odd roots.

\begin{satz}
Let $G$ be a complex reductive Lie supergroup, $G_\mathbb{R}$ a real form of $G$ and $P$ a parabolic subgroup of $G$
and $D \cong G_\mathbb{R}/L_\mathbb{R}$ an open orbit in $G/P$. Then the following are equivalent.

\begin{enumerate}
 \item $D$ is strongly measurable.
 \item $\mathfrak{p} \cap \tau\mathfrak{p}$ is reductive.
 \item $\mathfrak{p} \cap \tau\mathfrak{p}$ is the reductive part of $\mathfrak{p}$.
 \item $\tau \Phi^r = \Phi^r$ and $\tau \Phi^n = \Phi^c$.
 \item $\tau\mathfrak{p}$ is $G_{\bar{0}}$-conjugate to $\mathfrak{p}^{op}$.
\end{enumerate}
 
\end{satz}

\begin{proof}

Equivalence of 2.-5. follows form the fact that the reductive part of $\mathfrak{p} \cap \tau\mathfrak{p}$ is $\mathfrak{p}^r \cap \tau\mathfrak{p}^r$ and
its nilpotent radical is $\mathfrak{p}^r \cap \tau\mathfrak{p}^n + \mathfrak{p}^n \cap \tau\mathfrak{p}^r$.  It remains to show the equivalence of 1. and 2.

If $D$ is strongly measurable, then the Berezinian bundle $\Ber (G_\mathbb{R}/L_\mathbb{R})$ has a non-zero $G_\mathbb{R}$-invariant global section. 
By virtue of Theorem 4.13 in \cite{AH} this is equivalent to $\Ber((\mathfrak{g}_\mathbb{R}/\mathfrak{l}_\mathbb{R})^*)$ being a trivial $L_\mathbb{R}$-module.
This in turn is equivalent to $\Ber(\mathfrak{g}_\mathbb{R}/\mathfrak{l}_\mathbb{R})$ being a trivial $L_\mathbb{R}$-module.
Consequently the Harish-Chandra pair $\mathfrak{l}_\mathbb{R}$ and its complexification $\mathfrak{p} \cap \tau \mathfrak{p}$
act on $\mathfrak{g}_\mathbb{R}/\mathfrak{l}_\mathbb{R}$ and $\mathfrak{g}/(\mathfrak{p} \cap \tau \mathfrak{p})$ respectively
by operators with vanishing supertrace. As $\mathrm{ad}(\mathfrak{g}_{\bar{1}}) \subseteq \mathfrak{gl}(\mathfrak{g})_{\bar{1}}$, all of these
operators will have supertrace equal to zero. Hence one only needs to consider the action of $\mathfrak{p}_{\bar{0}} \cap \tau \mathfrak{p}_{\bar{0}}$
on $\mathfrak{g}/(\mathfrak{p} \cap \tau \mathfrak{p})$.

By assumption, $D_{\bar{0}}$ is measurable and consequently $\mathfrak{p}_{\bar{0}} \cap \tau \mathfrak{p}_{\bar{0}}$ is reductive. Therefore  $\mathfrak{p}_{\bar{0}} \cap \tau \mathfrak{p}_{\bar{0}}$
acts on $\mathfrak{p}_{\bar{0}} \cap \tau \mathfrak{p}_{\bar{0}}$ and on $\mathfrak{p}_{\bar{1}} \cap \tau \mathfrak{p}_{\bar{1}}$ with trace zero. 
As in the proof of the analogous theorem
in the classical case, one observes that $\mathfrak{p}_{\bar{1}} \cap \tau \mathfrak{p}_{\bar{1}}$ splits into 
$\mathfrak{p}_{\bar{1}}^r \cap \tau \mathfrak{p}_{\bar{1}}^r \oplus \mathfrak{p}_{\bar{1}}^r \cap \mathfrak{p}_{\bar{1}}^n  \oplus \mathfrak{p}_{\bar{1}}^n \cap \tau \mathfrak{p}_{\bar{1}}^r$ 
and the trace zero condition forces $\mathfrak{p}_{\bar{1}}^r \cap \mathfrak{p}_{\bar{1}}^n  \oplus \mathfrak{p}_{\bar{1}}^n \cap \tau \mathfrak{p}_{\bar{1}}^r = 0$. 
Consequently, $P \cap \tau P$ is reductive.

Conversely, if $P \cap \tau P$ is reductive, then its real form $L_\mathbb{R}$ is real reductive and therefore acts trivially on $\Ber((\mathfrak{g}_\mathbb{R}/\mathfrak{l}_\mathbb{R})^*)$. 
Then again by virtue of \cite{AH}, there exists a $G_\mathbb{R}$-invariant Berezinian form on $D$. 
Moreover, $P_{\bar{0}} \cap \tau P_{\bar{0}}$ is reductive and therefore $D_{\bar{0}}$ is measurable by the classical characterization theorem.

\end{proof}

\begin{rem}

Using the above characterization, non-measurable (neither \\ strongly nor weakly) open orbits can be constructed as follows:
According to \cite{OnI} any parabolic subalgebra $\mathfrak{p}$ of a basic Lie superalgebra $\mathfrak{g}$ is of the form

\[ \mathfrak{p} = \mathfrak{b} \oplus \bigoplus_{ \alpha \in \vert -J \vert } \mathfrak{g}^\alpha \]

\noindent where $\mathfrak{b}$ is a Borel subalgebra, $J \subseteq \Pi$ is a set of simple roots and $\vert -J \vert$ is the set of all negative roots
which are sums of elements of $-J$. Now, if $\mathfrak{p}$ is given as above, open $G_\mathbb{R}$-orbits in $G/P$ will be measurable
if and only if $\tau(J) = J$. So, in order to construct a non-measurable open orbit, one only needs to pick a subset $J \subseteq \Pi$
such that $\tau(J) \neq J$ and let $\mathfrak{p}$ be given by the above formula.  

\end{rem}

The following lemma gives a useful tool to determine whether a given flag domain is strongly measurable or not:

\begin{lemma}
 
A flag domain $D$ is strongly measurable if and only if for every $\alpha \in \Sigma: \alpha \in \Phi \Leftrightarrow \tau(-\alpha) \in \Phi$

\end{lemma}

\begin{proof}
 
Let $D$ be strongly measurable. Then $\alpha \in \Phi \cap \tau \Phi$ if and only if $- \alpha \in \Phi \cap \tau \Phi$.
Suppose $\alpha \in \Phi$. If $\tau(\alpha) \in \Phi$, then $-\alpha \in \Phi \cap \tau \Phi$ and therefore $\tau(-\alpha) \in \Phi$.
If $\tau(\alpha) \not\in \Phi$, then $\alpha \in \Phi^n$ and, as $D$ is strongly measurable, $\tau(\Phi^n) = -\Phi^n$, and therefore
$\tau(-\alpha) \in -(-\Phi^n) = \Phi^n$, especially $\tau(-\alpha) \in \Phi$. The same argument applied to $\tau(-\alpha)$ yields the converse.
So if $D$ is strongly measurable, $\alpha \in \Phi$ if and only if $\tau(-\alpha) \in \Phi$.

Now suppose $\alpha \in \Phi \Leftrightarrow \tau(-\alpha) \in \Phi$. If $\alpha \in \tau(\Phi)$, then $\tau(\alpha) \in \Phi$
and therefore $\tau(-\tau(\alpha)) = -\alpha \in \Phi$. As $\tau(-\alpha) \in \Phi$ was given, this implies $-\alpha \in \Phi \cap \tau \Phi$.
If $\alpha \not\in \tau(\Phi)$, then $\tau(\alpha) \not\in \Phi$ and therefore $- \alpha = \tau(-\tau(\alpha)) \not\in \Phi$.
So $\alpha \in \Phi \cap \tau\Phi$ if and only if $-\alpha \in \Phi \cap \tau\Phi$ and therefore $D$ is strongly measurable as claimed.

\end{proof}

Note that a $G_\mathbb{R}$-orbit $D$ is measurable if the super-trace of the natural representation of $\mathfrak{p}$ on $(\mathfrak{g}/\mathfrak{p})^*$
is zero. Therefore strong measurability requires $\alpha \in \Phi \cap \tau \Phi \Leftrightarrow -\alpha \in \Phi \cap \tau \Phi$
while weak measurability requires that the even and odd roots cancel each other. 


\section{Classification of measurable open orbits}

The measurability of open real orbits will be analyzed case by case according to the classification of complex simple Lie superalgebras.
The classification of real forms of the complex simple Lie superalgebras was done in \cite{Kac}, \cite{Par} and \cite{Ser}.
The following table is based on Table 3 in \cite{Ser} with a slight change of notation:

\noindent \begin{tabular}{p{1.7cm} | p{2.3cm}  | p{3.8cm} | p{5.5cm}}
 & real form & even part & defining involution \\ \hline 
$\mathfrak{sl}_{n \vert m}(\mathbb{C})$ & $\mathfrak{sl}_{n+1 \vert m+1}(\mathbb{R})$ & $\mathfrak{sl}_{n}(\mathbb{R}) \oplus \mathfrak{sl}_{m}(\mathbb{R}) [\oplus \mathbb{R}]$ & $\tau(X) = \bar{X}$ \\ \hline
       & $\mathfrak{sl}_{k \vert l}(\mathbb{H})$ & $\mathfrak{sl}_{k}(\mathbb{H}) \oplus \mathfrak{sl}_{l}(\mathbb{H}) [\oplus \mathbb{R}]$ & $\tau(X) = -J_{k \vert l}\bar{X}J_{k \vert l}$ \\ \hline
       & $\mathfrak{su}(p, q \vert r, s)$ & $\mathfrak{su}(p,q) \oplus \mathfrak{su}(r,s) [\oplus i\mathbb{R}]$ & $\tau(X) = -i^{\vert X \vert} I_{p,q \vert r,s} X^\dagger I_{p,q \vert r,s}$ \\ \hline
$\mathfrak{psl}_{n \vert n}(\mathbb{C})$ & $^0\mathfrak{pq}(n)$ & $\mathfrak{sl}_n(\mathbb{C})$ & $\tau(X) = \Pi(\bar{X})$  \\ \hline
       & $\mathfrak{us\pi}(n)$ & $\mathfrak{sl}_n(\mathbb{C})$ & $\tau(X) = -\Pi(\bar{X}^{st})$ \\ \hline
$\mathfrak{osp}(m \vert 2n)$ & $\mathfrak{osp}(p,q \vert 2n)$ & $\mathfrak{so}(p,q) \oplus \mathfrak{sp}_{2n}(\mathbb{R})$ & $\tau(X) = I_{p,q \vert 2n}\bar{X}I_{p,q \vert 2n}$ \\ \hline
$\mathfrak{osp}(2m \vert 2n)$ & $\mathfrak{osp}(2m \vert 2r, 2s)$ & $\mathfrak{so}^*(2m) \oplus \mathfrak{sp}(2r,2s)$ & $\tau(X) = \mathrm{Ad}(\mathrm{d}(J_m,I_{r,s},I_{r,s}))(\bar{X})$ \\ \hline 
P(n)   & $\mathfrak{s\pi}_\mathbb{R}(n)$ & $\mathfrak{sl}_n(\mathbb{R})$  & $\tau(X) = \bar{X}$ \\ \hline
P(2n)   & $\mathfrak{s\pi}_\mathbb{H}(n)$ & $\mathfrak{sl}_n(\mathbb{H})$ & $\tau(X) = -J_{n \vert n}\bar{X}J_{n \vert n}$ \\ \hline
Q(n)   & $\mathfrak{psq}_\mathbb{R}(n)$ & $\mathfrak{sl}_n(\mathbb{R})$  & $\tau(X) = \bar{X}$\\ \hline
Q(2n)   & $\mathfrak{psq}_\mathbb{H}(n)$ & $\mathfrak{sl}_n(\mathbb{H})$  & $\tau(X) = -J_{n \vert n}\bar{X}J_{n \vert n}$\\ \hline          
Q(n)   & $\mathfrak{upsq}(p,q)$ & $\mathfrak{su}(p,q)$  & $\tau(X) = -i^{\vert X \vert} I_{p,q \vert p,q} X^\dagger I_{p,q \vert p,q}$\\                   
 \end{tabular}

\subsection{Type A}

\begin{satz}
 
Let $\mathfrak{g} = \mathfrak{sl}_{m \vert n}(\mathbb{C})(m \neq n)$ or $\mathfrak{g} = \mathfrak{psl}_{n \vert n}(\mathbb{C})$. 

\begin{enumerate}
 \item If $\mathfrak{g}_\mathbb{R} = \mathfrak{sl}_{m \vert n}(\mathbb{R})$ or $\mathfrak{g}_\mathbb{R} = \mathfrak{sl}_{k \vert l}(\mathbb{H})(n=2k,m=2l)$, then
  \begin{itemize}
   \item A $G_\mathbb{R}$-orbit $D$ with open base has maximal odd dimension if and only if $\delta$ is even-symmetrizable.
   \item A flag domain $D$ is strongly measurable if and only if $\delta$ is even-symmetric.
   \item If $m=n$ and $\delta$ is $\Pi$-symmetric then $D$ is weakly measurable.
   \item If $n=m$, the unique open $G_\mathbb{R}$-orbit in $\mathbb{P}(\mathbb{C}^{n \vert n})$ is weakly measurable.
  \end{itemize}
 \item If $\mathfrak{g}_\mathbb{R} = \mathfrak{su}(p, n-p \vert q, m-q)$ then $D$ always has maximal odd dimension and is measurable.
 \item If $\mathfrak{g}_\mathbb{R} = \ ^0\mathfrak{pq}(n)$ then:
  \begin{itemize}
   \item A $G_\mathbb{R}$-orbit $D$ with open base has maximal odd dimension if and only if $\delta$ is odd-symmetrizable.
   \item A flag domain $D$ is strongly measurable if and only if $\delta$ is odd-symmetric.
   \item If $\delta$ is $\Pi$-symmetric then $D$ is weakly measurable.
  \end{itemize}
 \item If $\mathfrak{g}_\mathbb{R} = \mathfrak{us\pi}(n)$, then A $G_\mathbb{R}$-orbit $D$ with open base has maximal odd dimension if and only if $\delta$ is $\Pi$-symmetric. 
       If it has maximal odd dimension, it is always strongly measurable.
\end{enumerate}

\end{satz}

The second part of this theorem is immediate, as the involution in question is $\tau(\alpha) = -\alpha$.

\begin{rem}
The statements on strong measurability can be proven by analysis of the extended Dynkin diagram, but the point of view  adopted in this text gives more information, in particular it allows to
prove statements on maximal odd dimension and weak measurability as well.
\end{rem}

\subsubsection*{Proof of Theorem}
 
If $\mathfrak{h}$ is a $\tau$-invariant Cartan subalgebra, then there is basis $e_1, \ldots, e_n,f_1, \ldots, f_m$ of $\mathbb{C}^{n \vert m}$, such that
all $e_i,f_j$ are common eigenvectors of $\mathfrak{h}$. Let

\[ x_i (\mathrm{diag}(\lambda_1,\ldots, \lambda_n, \mu_1, \ldots, \mu_n)) = \lambda_i \ \textnormal{and} \ 
y_j (\mathrm{diag}(\lambda_1,\ldots, \lambda_n, \mu_1, \ldots, \mu_n)) = \mu_j \]

\noindent Then $(y_j - x_i) \in \Phi \Leftrightarrow \exists X \in \mathfrak{p}: X(e_i) = f_j$ 
and analogous conditions hold for the even roots.

\subsubsection{The Case $\mathfrak{g}_\mathbb{R} = \mathfrak{sl}_{m \vert n}(\mathbb{R})$}

The open $G_{\bar{0}\mathbb{R}}$-orbit in $G_{\bar{0}}/B_{\bar{0}}$, where $B_{\bar{0}} = B_n^+ \times B_m^+$ is the product of the usual Borel subalgebras of
$SL_n(\mathbb{C})$ and $SL_m(\mathbb{C})$, is open, and projects onto the open $G_{\bar{0}\mathbb{R}}$-orbit in $Z_{\bar{0}}$. So 

\[ (y_j - x_i) \in \Phi \Leftrightarrow \exists X \in \mathfrak{p}: X(e_i) = f_j \Leftrightarrow j \leq \min \{d_1 : d_0 \vert d_1 \in \delta, i \leq d_0\} \]  

\noindent As the intersection $\mathfrak{h} \cap \mathfrak{sl}_{n \vert m}(\mathbb{R})$ is not maximally compact, the open orbit does not contain the neutral point of $G_{\bar{0}}/P_{\bar{0}}$. It is therefore useful to pass to the isomorphic real form
$\mathfrak{g}_\mathbb{R}^\prime = g \mathfrak{g}_\mathbb{R} g^{-1}$, where $g = c_1 \ldots c_{\lfloor \frac{n}{2} \rfloor} \tilde{c_1} \ldots \tilde{c}_{\lfloor \frac{m}{2} \rfloor}$ is
the product of commuting Cayley transforms in the subspaces $\langle e_i , e_{n-i+1} \rangle_\mathbb{C}$ and $\langle f_j , f_{m-j+1} \rangle_\mathbb{C}$.

\noindent The $\mathbb{C}$-antilinear involution defining $\mathfrak{g}_\mathbb{R}^\prime$ is $\tau: \mathfrak{g} \rightarrow \mathfrak{g}, X \mapsto \mathrm{Ad}(\mathrm{diag}(A_n,A_m))(\bar{X})$,
where $A_n$ and $A_m$ are the respective antidiagonal unit matrices.

Its action on $\Sigma(\mathfrak{g} : \mathfrak{h})$ is given by

\[ \tau(x_j - x_i) = x_{n-j+1} -x_{n-i+1}\]
\[ \tau(y_j - y_i) = y_{m-j+1} -y_{m-i+1}\]
\[ \tau(y_j - x_i) = y_{m-j+1} -x_{n-i+1}\]

\begin{rem}
 
Instead of considering a conjugate real form, one could consider a different base point $z$ in $G/B$, given by
the basis $e_1 + i e_n, e_2 + i e_{n-1} \ldots, e_2 - i e_{n-1}, e_1 - i e_n, f_1 + i f_m, \ldots, f_1 - i f_m$.
Then the stabilizer of $z$ is a parabolic subalgebra $\mathfrak{p}^\prime = g^{-1} \mathfrak{p} g$ of $\mathfrak{g}$
and $\mathfrak{g}_\mathbb{R} \cap g^{-1} \mathfrak{h} g$ is maximally compact and therefore $(G_\mathbb{R} \cdot z)_{\bar{0}}$ 
is open. 

\end{rem}

\begin{prop}
 
An $SL_{n \vert m}(\mathbb{R})$-orbit $D$ with open base has maximal odd dimension if and only if $\delta$ is even-symmetrizable.

\end{prop}

\begin{proof}
 
First, suppose $D$ does not have maximal odd dimension. Then there is an odd root $\alpha \in \Phi^c \cap \tau \Phi^c$. Without loss of generality,
one may assume $\alpha = y_j - x_i$. Let $\overline{d_{\bar{0}}} \vert \overline{d_{\bar{1}}} = \min\{ d_{\bar{0}} \vert d_{\bar{1}} \in \delta : i \leq d_{\bar{0}} \}$
and $\widetilde{d_{\bar{0}}} \vert \widetilde{d_{\bar{1}}} = \min \{ d_{\bar{0}} \vert d_{\bar{1}} \in \delta : n-i+1 \leq d_{\bar{0}} \}$. As $(y_j - x_i) \not\in \Phi$, $j > \overline{d_{\bar{1}}}$.
Moreover, as $\tau(y_j - x_i) = (y_{m-j+1} - x_{n-i+1}) \not\in \Phi$, also $m-j+1 > \widetilde{d_{\bar{1}}} \Leftrightarrow j \leq m-\widetilde{d_{\bar{1}}}$. 
Finally, as $n-i+1 \leq \widetilde{d_{\bar{0}}}$ implies $i > n-\widetilde{d_{\bar{0}}}$, one obtains the following inequalities:

\[ n-\widetilde{d_{\bar{0}}} < i \leq \overline{d_{\bar{0}}}, \quad \overline{d_{\bar{1}}} < j \leq m-\widetilde{d_{\bar{0}}} \]

\noindent So $\overline{d_{\bar{0}}} \vert \overline{d_{\bar{1}}}$ and $n-\widetilde{d_{\bar{0}}},m-\widetilde{d_{\bar{1}}}$ are not comparable and therefore $\delta$ is not symmetrizable.

Conversely, let $d_{\bar{0}} \vert d_{\bar{1}}, d_{\bar{0}}^\prime \vert d_{\bar{1}}^\prime \in \delta$ such that $n-d_{\bar{0}}^\prime < d_{\bar{0}}$ and $m-d_{\bar{1}}^\prime > d_{\bar{1}}$.
Then the choice $j = m-d_{\bar{1}}^\prime, i = d_{\bar{0}}$ satisfies $(y_j - x_i) \in \Phi^c \cap \tau\Phi^c$ and thus $D$ does not have maximal odd dimension.
 
\end{proof}

\begin{prop}
 
An $SL_{n \vert m}(\mathbb{R})$-flag domain $D$ is strongly measurable, if and only if $\delta$ is even-symmetric.

\end{prop}

\begin{proof}
 
First assume $\delta$ is symmetric and let $\alpha = (y_j - x_i) \in \Phi$. We need to show $\tau(-\alpha) = (x_{n-i+1} - y_{m-j+1}) \in \Phi$.
Let $\underline{d_{\bar{0}}} \vert \underline{d_{\bar{1}}} = \max \{ d_{\bar{0}} \vert d_{\bar{1}} \in \delta: i > d_{\bar{0}} \}$ and $\overline{d_{\bar{0}}} \vert \overline{d_{\bar{1}}} = \min \{ d_{\bar{0}} \vert d_{\bar{1}} \in \delta: i \leq d_{\bar{0}} \}$.
Furthermore, let $\widetilde{d_{\bar{0}}} \vert \widetilde{d_{\bar{1}}} = \min \{ d_{\bar{0}} \vert d_{\bar{1}} \in \delta: m-j+1 \leq d_{\bar{1}}\}$. Then $(x_{n-i+1} - y_{m-j+1}) \in \Phi$,
if and only if $n - i + 1 \leq \widetilde{d_{\bar{0}}}$. 

Now $m-j+1 > m - \overline{d_{\bar{1}}}$ and by symmetry, $n - \underline{d_{\bar{0}}} \vert m - \underline{d_{\bar{1}}}$ is the succesor of $n - \overline{d_{\bar{0}}} \vert m - \overline{d_{\bar{1}}}$ in $\delta$,
so $\widetilde{d_{\bar{0}}} \geq n - \underline{d_{\bar{0}}}$ and as $i > \underline{d_{\bar{0}}}$, this implies $n - i + 1 \leq n - \underline{d_{\bar{0}}} \leq \widetilde{d_{\bar{0}}}$, so $\tau(-\alpha) \in \Phi$.

Now assume $D$ is not strongly measurable, so there exists $\alpha = y_j - x_i \in \Phi$, such that $\tau(-\alpha) = x_{n-i+1} - y_{m-j+1} \not\in \Phi$.
Let $\underline{d_{\bar{0}}} \vert \underline{d_{\bar{1}}}, \overline{d_{\bar{0}}} \vert \overline{d_{\bar{1}}}$ and $\widetilde{d_{\bar{0}}} \vert \widetilde{d_{\bar{1}}}$ as before.
As $\tau(-\alpha) \not\in \Phi$, $n - i + 1 > \widetilde{d_{\bar{0}}}$. But $i > \underline{d_{\bar{0}}}$, so $n- \overline{d_{\bar{0}}} < n-i+1 \leq n-\underline{d_{\bar{0}}}$. Therefore $\widetilde{d_{\bar{0}}} < n - \underline{d_{\bar{0}}}$. 
Also $\widetilde{d_{\bar{1}}} \geq m-j+1 > m-j \geq m-\overline{d_{\bar{1}}}$. Altogether, this yields

\[ n-\overline{d_{\bar{0}}} \vert m-\overline{d_{\bar{1}}} < \widetilde{d_{\bar{0}}} \vert \widetilde{d_{\bar{1}}} < n-\underline{d_{\bar{0}}} \vert m-\underline{d_{\bar{1}}} \]

\noindent And, as $n-\widetilde{d_{\bar{0}}} \vert m - \widetilde{d_{\bar{1}}} \not\in \delta$, $\delta$ is not symmetric. The proof proceeds analagously for even roots.

\end{proof}

\begin{prop}
 
If $n=m$, $G_\mathbb{R} = PSL_{n \vert n}(\mathbb{R})$, and $\delta$ is $\Pi$-symmetric then $D$ is weakly measurable.

\end{prop}

\begin{proof}
 
Let $\alpha = (x_j - x_i) \in \Phi \cap \tau\Phi$, so $\tau(\alpha) = x_{n-j+1} - x_{n-i+1} \in \Phi$. $\Pi$-symmetry 
is equivalent to the fact, that if one of $x_j - x_i, y_j - y_i, x_j - y_i$ and $y_j - x_i$ is in $\Phi$, then so are the other three.
The same fact applied to $\tau(\alpha)$ yields, that they are all in $\Phi \cap \tau\Phi$(Note that this would fail, if $n \neq m$).
Conversely, if $\alpha$ is not in $\Phi \cap \tau\Phi$, then so neither will be the other three. Now, when computing the supertrace of the action of $\mathfrak{p}$ on $(\mathfrak{g}/\mathfrak{p})^*$, the contributions of these four
roots will add up to zero. As $\alpha$ was arbitrary, the supertrace will be zero and therefore $D$ is weakly measurable.

\end{proof}

\begin{prop}
 
Let $G_\mathbb{R} = PSL_{n \vert n}(\mathbb{R})$ or $G_\mathbb{R} = PSL_{k \vert k}(\mathbb{H})$ 
and $Z = \mathbb{P}(\mathbb{C}^{n \vert n})$. Then the $G_\mathbb{R}$-flag domains in $Z$ are weakly measurable.

\end{prop}

\begin{proof}
 
Suppose we are given a $\tau$-generic basis of $\mathbb{C}^n$. Then $\mathfrak{p}$ is the stabiliser of $e_1$ and $\tau \mathfrak{p}$ is
the stabiliser of $e_n$. Therefore

\[ \Sigma \setminus (\Phi \cap \tau\Phi) = \] 
\[ \{ x_j - x_1 : j > 1 \} \cup \{ x_j - x_n : j < n \} \cup \{ y_j - x_1 : 1 \leq j \leq n \} \cup \{ y_j - x_n : 1 \leq j \leq n \} \]

\noindent The graded sum of these roots is

\[ \sum_{j=2}^n (x_j - x_1) + \sum_{j=1}^{n-1} (x_j - x_n) - \sum_{j=1}^n (y_j - x_1) - \sum_{j=1}^n (y_j - x_n) \]

\[ = 2 \sum_{j=2}^{n-1} x_j - (n-2)(x_1 + x_n) - (2 \sum_{j=1}^n y_j - n(x_1 + x_n))\] 

\[ = 2 \sum_{j=1}^n x_j - 2 \sum_{j=1}^n y_j = 2 \mathrm{str} = 0    \]

\noindent Consequently $D$ is weakly measurable.

\end{proof}

\subsubsection{The Case $\mathfrak{g}_\mathbb{R} = \ ^0\mathfrak{pq}(n)$}

In this case $n =m$ and the defining involution of $\mathfrak{g}_\mathbb{R}$ is 

\[ \tau \begin{pmatrix} A & B \\ C & D \end{pmatrix} = \begin{pmatrix} \bar{D} & \bar{C} \\ \bar{B} & \bar{A} \end{pmatrix} \]

\noindent and the action of $ \tau $ on $\Sigma(\mathfrak{g} : \mathfrak{h})$ is given by

\[ \tau(x_j - x_i) = y_j - x_i \]
\[ \tau(y_j - x_i) = x_j - y_i \]

\noindent Moreover the $G_{\bar{0}\mathbb{R}}$-orbit through the neutral point in $G_{\bar{0}}/B_{\bar{0}}$  is open, where $B_{\bar{0}} = B_n^+ \times B_n^-$ is the product of the usual Borel subalgebra of
$SL_n(\mathbb{C})$ and its opposite. It projects onto the open $G_{\bar{0}\mathbb{R}}$-orbit in $Z_{\bar{0}}$. So 

\[ (y_j - x_i) \in \Phi \Leftrightarrow \exists X \in \mathfrak{p}: X(e_i) = f_j \Leftrightarrow n-j+1 \leq \min \{d_{\bar{1}} : d_{\bar{0}} \vert d_{\bar{1}} \in \delta, i \leq d_{\bar{0}}\} \]  

\begin{prop}
 
A $^0PQ(n))$-orbit $D$ with open base has maximal odd dimension if and only if $\delta$ is odd-symmetrizable.

\end{prop}

\begin{proof}
 
First suppose $D$ does not have maximal odd dimension. Then there is an odd root $\alpha \in \Phi^c \cap \tau \Phi^c$. Without loss of generality,
one may assume $\alpha = y_j - x_i$. Let $\overline{d_{\bar{0}}} \vert \overline{d_{\bar{1}}} = \min\{ d_{\bar{0}} \vert d_{\bar{1}} \in \delta : i \leq d_{\bar{0}} \}$
and $\widetilde{d_{\bar{0}}} \vert \widetilde{d_{\bar{1}}} = \min \{ d_{\bar{0}} \vert d_{\bar{1}} \in \delta : n-i+1 \leq d_{\bar{1}} \}$. As $(y_j - x_i) \not\in \Phi$, $n-j+1 > \overline{d_{\bar{1}}}$.
Moreover as $\tau(y_j - x_i) = (x_j - y_i) \not\in \Phi$, also $j > \widetilde{d_{\bar{0}}}$. 
Finally as $n-i+1 \leq \widetilde{d_{\bar{1}}}$ implies $i > n-\widetilde{d_{\bar{1}}}$, one obtains the following inequalities:

\[ n-\widetilde{d_{\bar{1}}} < i \leq \overline{d_{\bar{0}}}, \quad \widetilde{d_{\bar{0}}} < j \leq n-\overline{d_{\bar{1}}} \]

\noindent So $\overline{d_{\bar{0}}} \vert \overline{d_{\bar{1}}}$ and $\widetilde{d_{\bar{0}}} \vert \widetilde{d_{\bar{1}}}$ are not comparable and therefore $\delta$ is not symmetrizable.

Conversely, let $d_{\bar{0}} \vert d_{\bar{1}}, d_{\bar{0}}^\prime \vert d_{\bar{1}}^\prime \in \delta$ such that $n-d_{\bar{1}}^\prime < d_{\bar{0}}$ and $n-d_{\bar{1}} > d_{\bar{0}}^\prime$.
Then the choice $j = n - d_{\bar{1}}, i = d_{\bar{0}}^\prime$ satisfies $(y_j - x_i) \in \Phi^c \cap \tau\Phi^c$ and thus $D$ does not have maximal odd dimension.
 
\end{proof}

\begin{prop}
 
A $^0PQ(n)$-flag domain $D$ is strongly measurable, if and only if $\delta$ is odd-symmetric.

\end{prop}

\begin{proof}
 
First assume $\delta$ is symmetric and let $\alpha = (y_j - x_i) \in \Phi$. We need to show $\tau(-\alpha) = (y_i - x_j) \in \Phi$.
Let $\underline{d_{\bar{0}}} \vert \underline{d_{\bar{1}}} = \max \{ d_{\bar{0}} \vert d_{\bar{1}} \in \delta: i > d_{\bar{0}} \}$ and $\overline{d_{\bar{0}}} \vert \overline{d_{\bar{1}}} = \min \{ d_{\bar{0}} \vert d_{\bar{1}} \in \delta: i \leq d_{\bar{0}} \}$.
Furthermore let $\widetilde{d_{\bar{0}}} \vert \widetilde{d_{\bar{1}}} = \min \{ d_{\bar{0}} \vert d_{\bar{1}} \in \delta: j \leq d_{\bar{0}}\}$. Then $(y_i - x_j) \in \Phi$,
if and only if $ n-i+1 \leq \widetilde{d_{\bar{1}}}$. 

Now $n - \overline{d_{\bar{1}}} < j$ and by symmetry, $n - \underline{d_{\bar{1}}} \vert n - \underline{d_{\bar{0}}}$ is the succesor of $n - \overline{d_{\bar{1}}} \vert n - \overline{d_{\bar{0}}}$ in $\delta$,
so $\widetilde{d_{\bar{0}}} \geq n - \underline{d_{\bar{1}}}$ and as $i > \underline{d_{\bar{0}}}$, this implies $n - i + 1 \leq n - \underline{d_{\bar{0}}} \leq \widetilde{d_{\bar{1}}}$, so $\tau(-\alpha) \in \Phi$.

Now assume $D$ is not strongly measurable, so there exists $\alpha = y_j - x_i \in \Phi$, such that $\tau(-\alpha) = y_i - x-j \not\in \Phi$.
Let $\underline{d_{\bar{0}}} \vert \underline{d_{\bar{1}}}, \overline{d_{\bar{0}}} \vert \overline{d_{\bar{1}}}$ and $\widetilde{d_{\bar{0}}} \vert \widetilde{d_{\bar{1}}}$ as before.
As $\tau(-\alpha) \not\in \Phi$, $n - i + 1 > \widetilde{d_{\bar{1}}}$. But $i > \underline{d_{\bar{0}}}$, so $n- \overline{d_{\bar{0}}} < n-i+1 \leq n-\underline{d_{\bar{0}}}$. Therefore $\widetilde{d_{\bar{1}}} < n - \underline{d_{\bar{0}}}$. 
Also $\widetilde{d_{\bar{0}}} \geq j >  n-\overline{d}_{\bar{1}}$. Altogether this yields

\[ n-\overline{d_{\bar{1}}} \vert n-\overline{d_{\bar{0}}} < \widetilde{d_{\bar{0}}} \vert \widetilde{d_{\bar{1}}} < n-\underline{d_{\bar{1}}} \vert n-\underline{d_{\bar{0}}} \]

\noindent And as $n-\widetilde{d_{\bar{1}}} \vert n - \widetilde{d_{\bar{0}}} \not\in \delta$, $\delta$ is not symmetric. The proof proceeds analagously for even roots.

\end{proof}

\begin{prop}
 
If $\delta$ is $\Pi$-symmetric then $^0PQ(n)$-flag domains $D$ in $Z(\delta)$ are weakly measurable.

\end{prop}

\begin{proof}
 
Let $\alpha = (x_j - x_i) \in \Phi \cap \tau\Phi$, so $\tau(\alpha) = y_j - x_i \in \Phi$. $\Pi$-symmetry 
is equivalent to the fact that if one of $x_j - x_i, y_{n-j+1} - y_{n-i+1}, x_j - y_{n-i+1}$ and $y_{n-j+1} - x_i$ is in $\Phi$, then so are the other three.
The same fact applied to $\tau(\alpha)$ yields that they are all in $\Phi \cap \tau\Phi$.
Conversely if $\alpha$ is not in $\Phi \cap \tau\Phi$ then so neither will be the other three. 
Now when computing the supertrace of the action of $\mathfrak{p}$ on $(\mathfrak{g}/\mathfrak{p})^*$ the contributions of these four
roots will add up to zero. As $\alpha$ was arbitrary, the supertrace will be zero and therefore $D$ is weakly measurable.

\end{proof}

\subsubsection{The Case $\mathfrak{g}_\mathbb{R} = \mathfrak{us}\pi(n)$}

In this case $n = m$ and the defining involution is

\[ \tau \begin{pmatrix} A & B \\ C & D \end{pmatrix} = \begin{pmatrix} -D^\dagger & B^\dagger \\ -C^\dagger & -A^\dagger \end{pmatrix} \]

\noindent and the action of $\tau$ on $\Sigma(\mathfrak{g} : \mathfrak{h})$ is

\[ \tau(x_j - x_i) = y_i - y_j \]
\[ \tau(y_j - x_i) = y_i - x_j \]  

\noindent Here the $G_{\bar{0}\mathbb{R}}$-orbit through the neutral point in $G_{\bar{0}}/B_{\bar{0}}$, where $B_{\bar{0}} = B_n^+ \times B_n^+$, is open, and
projects onto $D_{\bar{0}} \subseteq Z_{\bar{0}}$.

\begin{prop}
 
A $US\Pi(n)$-orbit $D$ with open base has maximal odd dimension if and only if $\delta$ is $\Pi$-symmetric. If this is the case $D$ is strongly measurable.

\end{prop}

\begin{proof}
 
The involution $\tau$ acts trivially on the roots $\pm (y_i - x_i)$ for all $1 \leq i \leq n$. So for $\Phi \cap \tau \Phi$
to be empty, one needs $\pm (y_i - x_i) \in \Phi$ for all $1 \leq i \leq n$. This is equivalent to $\Pi$-symmetry.

$\Pi$-symmetry also implies that if one of $x_j - x_i, y_j - y_i, y_j - x_i$ and $x_j - y_i$ is in $\Phi$, then
so are the other three. Now $y_j - y_i = \tau(-(x_j - x_i))$ and $y_j - x_i = \tau(-(y_j - x_i))$. So $\Pi$-symmetry
also yields strong measurability. 

\end{proof}

\subsection{Types B,C and D} 

Now suppose we are given we are given a $\mathbb{C}$-supervector space $V^{k \vert 2m}$ with a non-degenerate even super-symmetric
bilinear form $S: V \times V \rightarrow \mathbb{C}$. If $\mathfrak{g}$ is of type $B(n,m)$, $k = 2n+1$, if it is of type $C(m)$, then $k = 2$
and if it is of type $D(n,m)$, then $k = 2n$. One can always choose a basis $e_1, \ldots, e_k,f_1, \ldots, f_{2m}$ of $V$ such that

\[ S(e_i,f_l) = 0 , S(e_i,e_j) = \delta_{i,k-j}, S(f_l,f_a) = \delta_{l,2m-a} \forall 1 \leq i,j \leq k, 1 \leq l,a \leq 2m \]  

The simple superalgebras of types $B(n,m),C(m)$ and $D(n,m)$ are the orthosymplectic Lie superalgebras $\mathfrak{g} = \mathfrak{osp}(k \vert 2m)$. They consist of all matrices self-adjoint with respect to $S$.
Using the bilinear form $S$ it is possible to restrict the variety of possible flag types:

\begin{lemma}

If $G = \mathrm{Osp}(n \vert 2m)$, then every dimension sequence $\delta$ is even-symmetric

\end{lemma}

\begin{proof}
This is due to the fact that if $\mathfrak{g}$ stabilizes a subspace $W \subseteq V$, then it also stabilizes the orthosymplectic complement $W^\perp$.
\end{proof}

\begin{satz}\label{MaxOddOsp}

Let $G = \mathrm{Osp}(n \vert 2m)$, $G_\mathbb{R} = \mathrm{Osp}(2p,2q+1 \vert 2m), \mathrm{Osp}(2p+1,2q+1 \vert 2m)$ or $\mathrm{Osp}^*(n \vert 2r,2s)$ and $Z = G/P$ be a $G$-flag manifold. Then all $G_\mathbb{R}$-orbits in $Z$ with open base have maximal odd dimension. Moreover they are strongly measurable.

\end{satz}
 
The only real form for which the action on $\Sigma$ is not trivial is the real form $\mathfrak{g}_\mathbb{R} = \mathfrak{osp}(2p+1,2q+1 \vert 2m)$
of $\mathfrak{g} = \mathfrak{osp}(2n \vert 2m)$.

In that particular case the action of $\tau$ on $\Sigma$ is given by

\[ \tau(x_n - x_j) = x_n + x_j , \tau(x_n - y_j) = x_n + y_j , \tau(\alpha) = - \alpha, \ \textnormal{else} \]

\noindent This yields the following

\begin{satz}
\label{D21}
Let $\mathfrak{g} = \mathfrak{osp}(2n \vert 2m)$, $\mathfrak{g}_\mathbb{R} = \mathfrak{osp}(2p+1,2q+1 \vert 2m)$ and $D$ as above. Then:

\begin{enumerate}
 \item A $G_\mathbb{R}$-orbit $D$ with open base has maximal odd dimension, if and only if $n \vert d \not\in \delta$ for all $d < m$
 \item A  $G_\mathbb{R}$-flag domain $D$ is strongly measurable if and only if $n \vert m \not\in \delta$ or $n-1 \vert m \in \delta$
 \item A  $G_\mathbb{R}$-flag domain $D$ is weakly measurable if and only if $n \vert m \in \delta$ and its immediate predecessor is $n-d-1 \vert m-d$, $0 \leq d \leq \min \{n-1, m\}$ 
\end{enumerate}
  
\end{satz}

\begin{proof}
 
1. If $\dim_{\bar{1}}D$ is not maximal there must be some $\alpha \in \Phi^c \cap \tau\Phi^c$. By virtue of even symmetry
this can only happen for those roots which satisfy $\tau(\alpha) \neq -\alpha$. As $-x_n - y_j \in \Phi$ for all $1 \leq j \leq m$, 
we may assume $\alpha = x_n - y_j$. Then $\tau(\alpha) = x_n + y_j$, which is always an element of $\Phi^c$ because of even symmetry.
So $D$ does not have maximal odd dimension precisely when $x_n - y_j \in \Phi^c$ for some $1 \leq j \leq m$. This is only the case, if $n \vert d \in \delta$ for some $d < j$.

2. $D_{\bar{0}}$ is measurable if and only if either $n \not\in \delta_{\bar{0}}$ or $n-1 \in \delta_{\bar{0}}$. If $n \vert m \not\in \delta$ or $n-1 \vert m \in \delta$, then $x_n \pm y_j \in \Phi$,
if and only if $- x_n \pm y_j \in \Phi$, so these contributions cancel in the sum of odd roots, yielding strong measurability.

3. Assume $n-d-1 \vert m-d \in \delta$ and $n \vert m$ is its succesor in $\delta$. Then the even roots $\alpha \in \Phi \cap \tau\Phi$
with $-\alpha \not\in \Phi \cap \tau\Phi$ are $-x_j - x_n$ and $x_j - x_n$ for all $n-d-1 \leq j \leq n-1$. On the other hand 
the odd roots $\alpha \in \Phi \cap \tau\Phi$ with $-\alpha \not\in \Phi \cap \tau\Phi$ are $-x_n - y_j$ and $-x_n + y_j$ for all $m - d \leq j \leq m$.
The graded sum of all these roots is zero so $D$ is weakly measurable. The converse stems from the fact
that the odd roots will always be the given ones and the only way to cancel them with even roots is with the given even roots. 

\end{proof}

This completes the characterization of measurability for the orthosymplectic superalgebras.

\subsection{The exceptional Lie superalgebras}

As the existence of non-measurable open orbits requires the existence of a non-trivial automorphism of the Dynkin diagram,
these can only occur for the exceptional superalgebras $E_6$ and $D(2,1,\alpha)$. For the other excpetional Lie superalgebras this implies the following:

\begin{satz}
Let $G = E_7,E_8,F_4,G_2,F(4)$ or $G(3)$, $G_\mathbb{R}$ a real form and $Z = G/P$ a $G$-flag manifold. Then all $G_\mathbb{R}$-orbits $D$ in $Z$ with open base have maximal odd dimension. Moreover they are all strongly measurable.
\end{satz}

Now consider the two remaining cases:

The exceptional family $D(2,1,\alpha)$ does have the same root system as the simple Lie superalgebra $D(2,1)$. Moreover,
if $\mathfrak{g}$ is of type $D(2,1,\alpha)$, then $\mathfrak{g}_{\bar{0}} \cong \mathfrak{sl}_2(\mathbb{C}) \oplus \mathfrak{sl}_2(\mathbb{C}) \oplus \mathfrak{sl}_2(\mathbb{C})$ 
and there are three real forms with respective even parts $\mathfrak{g}_{\bar{0}\mathbb{R}} \cong \mathfrak{sl}_2(\mathbb{R}) \oplus \mathfrak{sl}_2(\mathbb{R}) \oplus \mathfrak{sl}_2(\mathbb{R})$, 
$\mathfrak{g}_{\bar{0}\mathbb{R}} \cong \mathfrak{sl}_2(\mathbb{R}) \oplus \mathfrak{su}(2) \oplus \mathfrak{su}(2)$ or $\mathfrak{g}_{\bar{0}\mathbb{R}} \cong \mathfrak{sl}_2(\mathbb{R}) \oplus \mathfrak{sl}_2(\mathbb{C})$.

For the first two of these, the action of $\tau$ on the root system is $\tau(\alpha) = -\alpha$.
For the third real form, the action of $\tau$ on the root system is given by

 \[\tau(x_1 - x_2) = (-x_1 - x_2), \tau(\pm y \pm x_1) = \mp y \mp x_1, \tau(\pm y \pm x_2) = \mp y \pm x_2\]

\noindent This is the same action as for the real form $\mathfrak{osp}(1,3 \vert 2)$ of $\mathfrak{osp}(4,2)$. 
Moreover, this real form can only occur, if $\alpha = 1, -\frac{1}{2}$ or $-2$ and then $\mathfrak{g}$ and $\mathfrak{g}_\mathbb{R}$ are 
isomorphic to $\mathfrak{osp}(4,2)$ and $\mathfrak{osp}(1,3 \vert 2)$ respectively.

\begin{satz}
Let $G = D(2,1,\alpha)$ and $Z = G/P$ a flag manifold.

\begin{enumerate}
\item If $G_\mathbb{R}$ is one of the real forms acting trivially on the Dynkin diagram, then all $G_\mathbb{R}$-orbits $D$ in $Z$
have maximal odd dimension. Furthermore they are strongly measurable.
\item If $G_\mathbb{R}$ is the real form satisfying $G_{\bar{0}\mathbb{R}} = SL_2(\mathbb{R}) \times SL_2(\mathbb{C})$, then
$G \cong \mathrm{Osp}(4 \vert 2)$, $G_\mathbb{R} \cong \mathrm{Osp}(1,3 \vert 2)$ and the conditions for maximal odd dimension and weak or strong measurability are given by Theorem \ref{D21}.
\end{enumerate}

\end{satz}

The exceptional Lie group $E_6$ is the only exceptional classical simple Lie group that allows non-measurable flag domains.

\begin{satz}
Let $G = E_6$ and $Z = G/P$ be a $G$-flag manifold.

\begin{enumerate}
\item If $G_\mathbb{R}= E_{6,F_4}$ or $E_{6,C_4}$, then an open $G_\mathbb{R}$-orbit $D$ in $Z$ is measurable if $P \cap \tau P$ is a complex reductive group.
\item If $G_\mathbb{R}$ is any other real form, then all open $G_\mathbb{R}$-orbits $D$ in $Z$ are measurable.  
\end{enumerate}

\end{satz}

\subsection{Type P}

Now consider the supervectorspace $V = \mathbb{C}^{n \vert n}$ with a non-degenerate odd super-skewsymmetric bilinear form $\omega$.
Then there is a standard basis $e_1, \ldots, e_n,$  $f_1, \ldots, f_n$ of $V$ such that

\[ \omega(e_i,e_j) = \omega(f_i,f_j) = 0 , \omega(e_i,f_j) = \delta_{ij} \forall 1 \leq i,j \leq n \]

The periplectic Lie superalgebra $\mathfrak{p}(n)$ is the Lie superalgebra of all matrices which are self-adjoint with respect to $\omega$. As in the case of the orthosymplectic Lie superalgebra this restricts the possible flag types:

\begin{satz}
 If $G = P(n)$ then every dimension sequence $\delta$ is \\ odd-symmetric. 
\end{satz}

\begin{proof}
If $\mathfrak{g} = \mathfrak{p}(n)$ stabilizes a subspace $W \subseteq V$, then it also stabilizes $W^\perp$.
As $\dim W^\perp = n - \dim W_{\bar{1}} \vert n - \dim W_{\bar{0}}$, every dimension sequence will be odd-symmetric. 
\end{proof}

The only possible real forms of $\mathfrak{g}$ are $\mathfrak{g}_\mathbb{R} = \mathfrak{p}_\mathbb{R}(n)$
and $\mathfrak{p}_\mathbb{H}(m)$, if $n = 2m$ is even. In both cases the action of $\tau$ on $\Sigma$ is the following:

\[ \tau(\pm x_i \pm x_j) = \pm x_{n-i+1} \pm x_{n-j+1} \]

\noindent In particular the roots $\alpha = \pm (x_i + x_j)$ are fixed by $\tau$. This implies the following:

\begin{satz}
 
Let $G = P(n)$ and $G_\mathbb{R} = P_\mathbb{R}(n)$ or $P_\mathbb{H}(m)$. A $G_\mathbb{R}$-flag domain $D$ in $Z(\delta)$ has maximal odd dimension if and only if $\delta$ is $\Pi$-symmetric.

\end{satz}
 
\begin{proof}
 
For $\dim_{\bar{1}}D$ to be maximal all roots fixed by $\tau$ must be in $\Phi$. These are presiely the roots $\alpha = \pm (x_i + x_{n-i+1})$.
Now $ - x_i - x_{n-i+1} \in \Phi$ if and only if there is some $X \in \mathfrak{p}$ such that $X(f_i) = e_{n-i+1}$. For this to hold for all 
$1 \leq i \leq n$ one needs $d_{\bar{0}} \geq d_{\bar{1}}$ for all $d_{\bar{0}} \vert d_{\bar{1}} \in \delta$. On the other hand, $- x_i - x_{n-i+1} \in \Phi$ requires $d_{\bar{1}} \geq d_{\bar{0}}$
for all $d_{\bar{0}} \vert d_{\bar{1}} \in \delta$. As both these conditions must be satisfied, $\delta$ is $\Pi$-symmetric. 

Conversely suppose there is an odd root $\alpha \in \Phi^c \cap \tau\Phi^c$, without loss of generality $\alpha = - x_i - x_j$. 
Then $n - j + 1 > \overline{d_{\bar{0}}}$, where $\overline{d} = \min \{ d_{\bar{0}} \vert d_{\bar{1}} \in \delta : d_{\bar{1}} \geq i \}$. Moreover, $\tau(\alpha) = - x_{n-i+1} - x_{n-j+1} \not\in \Phi$,
so $ j > \widetilde{d_{\bar{0}}} $, where $\tilde{d} = \min \{ d_{\bar{0}} \vert d_{\bar{1}} \in \delta : d_{\bar{1}} \geq n - i + 1 \}$. This yields:

\[ \widetilde{d_{\bar{0}}} < j \leq n - \overline{d_{\bar{0}}}  , n - \widetilde{d_{\bar{1}}} < i \leq \overline{d_{\bar{1}}}  \]

\noindent Now suppose $\delta$ is $\Pi$-symmetric. Then $\overline{d_{\bar{0}}} = \overline{d_{\bar{1}}}$ and $\widetilde{d_{\bar{0}}} = \widetilde{d_{\bar{1}}}$
and consequently $n = \widetilde{d_{\bar{0}}} + n - \widetilde{d_{\bar{0}}} < n - \overline{d_{\bar{0}}} + \overline{d_{\bar{0}}} = n$, which is a contradiction.   

\end{proof}

Note that given $\Pi$-symmetry the even and odd symmetry conditions coincide. Consequently if $D$ has maximal odd dimension
then $D_{\bar{0}}$ is automatically measurable. Also for all odd roots $\alpha$ one has $\alpha \in \Phi$ if and only if $-\alpha \in \Phi$.
So for $D$ to be strongly measurable one needs to consider the roots $2x_i$ whose negatives are no roots.

\begin{satz}
 
A $P_\mathbb{R}(n)$- or $P_\mathbb{H}(m)$-flag domain $D$ in $Z(\delta)$ is strongly measurable if and only if $n = 2m$ is even and $m \vert m \in \delta$.

\end{satz}
 
\begin{proof}

If $n = 2m+1$ is odd then $\alpha = 2x_{m + 1} \in \Phi \cap \tau\Phi$ in any case(by virtue of the condition for maximal odd dimension), 
so $D$ cannot be measurable.

Now let $n = 2m$ be even. $2x_i$ is an element of $\Phi$, if and only if there is some $X \in \mathfrak{p}$ such that $X(e_i) = f_i$. This requires
$n - i + 1 \leq \overline{d}$, where $\overline{d} \vert \overline{d} = \min \{ d \vert d \in \delta: i \leq d \}$. 
This implies $n + 1 \leq 2\overline{d}$. Analogously $2x_i \in \tau\Phi$ if and only if $i \leq \tilde{d}$, 
where $\tilde{d} \vert \tilde{d} = \min \{ d \vert d \in \delta: n - i + 1 \leq d \}$. By the odd symmetry condition $\tilde{d} = n - \underline{d} + 1$,
where $\underline{d} \vert \underline{d} = \max \{ d \vert d \in \delta : d < i \}$. Summarizing this, one has $2x_i \in \Phi \cap \tau\Phi$ if and only if

\[ 2 (\underline{d} + 1) \leq n + 1 \leq 2 \overline{d} \]

\noindent Now if $m \vert m \not\in \delta$, then $2x_m \in \Phi \cap \tau\Phi$, as in that case $\underline{d} < m$ and $\overline{d} > m$, 
so the above condition is satisfied. Conversely if $m \vert m \in \delta$, then $\underline{d} + 1 $ and $\overline{d}$ are always
either both less or equal to $m$ or both greater than $m$. Consequently $2x_i \not\in \Phi \cap \tau\Phi$ for all $1 \leq i \leq n$ and
therefore $D$ is strongly measurable.  
 
\end{proof}

\subsection{Type Q}

The Lie superalgebras of type $Q$ have the important trait that all root spaces are $1 \vert 1$-dimensional and therefore 
for any parabolic subalgebra $\mathfrak{p} \subseteq \mathfrak{g}$ one always has $\dim_{\bar{1}}(\mathfrak{p}) = \dim_{\bar{0}}(\mathfrak{p})$.
Therefore if $D_{\bar{0}}$ is open, then $D$ will automatically have maximal odd dimension. Moreover as the root spaces are $1 \vert 1$- dimensional,
every root is even and odd and a graded sum of roots is therefore always zero. This yields the following classification result:

\begin{satz}
 
Let $G = PQ(n)$, $G_\mathbb{R}$ any real form and $Z(\delta)$ a $G$-flag supermanifold. A $G_\mathbb{R}$-flag domain $D$ in $Z(\delta)$ always has maximal odd dimension and is always weakly measurable. It is strongly measurable if and only if $D_{\bar{0}}$ is measurable.

\end{satz}

\section{Even real forms and measurable flag domain}
\subsection{Characterization of strong measurability for even real forms}
As the Berezinian module of a flag domain $D = G_\mathbb{R}/L_\mathbb{R}$ is $1$-dimensional the action of $\mathfrak{l}_{\bar{1}\mathbb{R}}$
on it is necessarily trivial. It is therefore reasonable to consider instead real forms of the even group $G_{\bar{0}}$ 
rather than real forms of the whole supergroup $G$. However, as the action of the $\mathbb{C}$-antilinear involution $\tau$
is central to the classification, one should actually consider an even real form $G_\mathbb{R}$ of $G$.

The even real forms of the basic classical Lie superalgebras were classified in \cite{Ser}. The table for the non-exceptional
ones is as follows:

\noindent \begin{tabular}{p{1.6cm} | p{3.4cm}  | p{5.9cm}}
 & even real form  & defining involution \\ \hline 
$\mathfrak{sl}_{n \vert m}(\mathbb{C})$ & $\mathfrak{sl}_{n}(\mathbb{R}) \oplus \mathfrak{sl}_{k}(\mathbb{H}) [\oplus \mathbb{R}]$ & $\tau(X) = \mathrm{Ad}(\mathrm{d}(I_m,J_k))(\bar{X})$ \\ \hline
       & $\mathfrak{su}(p,q) \oplus \mathfrak{su}(r,s) [\oplus i\mathbb{R}]$ & $\tau(X) = - I_{p,q \vert r,s} \bar{X}^{st} I_{p,q \vert r,s}$ \\ \hline
$\mathfrak{psl}_{n \vert n}(\mathbb{C})$ & $\mathfrak{sl}_n(\mathbb{C})$ & $\tau(X) = \Pi(\delta_i(\bar{X}))$  \\ \hline
$\mathfrak{osp}(m \vert 2n)$ & $\mathfrak{so}(p,q) \oplus \mathfrak{sp}(2r,2s)$ & $\tau(X) = \mathrm{Ad}(\mathrm{d}(I_{p,q}, d(I_{r,s},I_{r,s}) \cdot J_n))(\bar{X})$ \\ \hline
$\mathfrak{osp}(2m \vert 2n)$ & $\mathfrak{so}^*(2m) \oplus \mathfrak{sp}_{2n}(\mathbb{R})$ & $\tau(X) = \mathrm{Ad}(\mathrm{d}(J_m,I_{2n}))(\bar{X})$ \\  

 \end{tabular}

As before there is a distinction between strong and weak measurability and a characterization theorem for strong measurability:

\begin{satz}
 
Let $Z = G/P$ be a flag supermanifold, $G_\mathbb{R} = \mathrm{Fix}(\tau)^\circ$ an even real form and $D \subseteq Z$ a flag domain.
Then the following are equivalent:

\begin{enumerate}
 \item $D$ is strongly measurable
 \item $\mathfrak{p} \cap \tau\mathfrak{p}$ is reductive
 \item $\mathfrak{p} \cap \tau\mathfrak{p}$ is the reductive part of $\mathfrak{p}$
 \item $\tau \Phi^r = \Phi^r$ and $\tau \Phi^n = \Phi^c$
 \item $\tau\mathfrak{p}$ is $G_0$-conjugate to $\mathfrak{p}^{op}$.
\end{enumerate}

\end{satz}

\begin{proof}
 
Let $z_0 \in D_{\bar{0}}$ be a base point and $\mathfrak{p} = \mathrm{Stab}_\mathfrak{g}(z_0)$. As $D_{\bar{0}}$ is an open $G_\mathbb{R}$-orbit
$\mathfrak{p} \cap \mathfrak{g}_\mathbb{R}$ contains a maximal compact Cartan subalgebra $\mathfrak{h}_\mathbb{R}$ of $\mathfrak{g}_\mathbb{R}$.
The adjoint representation of $G$ induces a representation $\tilde{\mathrm{ad}}$ of $\mathfrak{h}_\mathbb{R}$ on $(\mathfrak{g}/\mathfrak{p})^\mathbb{R}$
Let $H \in \mathfrak{h}_\mathbb{R}$. Then

\[ \mathrm{str}(\tilde{\mathrm{ad}}(H)) = 2 \sum_{\alpha \in \Phi_{\bar{0}}^c} \mathrm{Re} \ \alpha(H) - 2 \sum_{\alpha \in \Phi_{\bar{1}}^c} \mathrm{Re} \ \alpha(H)\]
\[ = \sum_{\alpha \in \Phi_{\bar{0}}^c} (\alpha(H) + \overline{\alpha(H)}) - \sum_{\alpha \in \Phi_{\bar{1}}^c} (\alpha(H) + \overline{\alpha(H)})\]
\[ = \sum_{\alpha \in \Phi_{\bar{0}}^c} (\alpha(H) + \tau \cdot \alpha(H)) - \sum_{\alpha \in \Phi_{\bar{1}}^c} (\alpha(H) + \tau \cdot \alpha(H))\]
\[ = \sum_{\alpha \in \Phi_{\bar{0}}^c} \alpha(H) + \sum_{\alpha \in \tau\Phi_{\bar{0}}^c} \alpha(H) - \sum_{\alpha \in \Phi_{\bar{1}}^c} \alpha(H) - \sum_{\alpha \in \tau\Phi_{\bar{1}}^c} \alpha(H)\]

Now assume $D_{\bar{0}}$ is measurable. Then the sums over the even roots are zero and one obtains

\[ \mathrm{str}(\tilde{\mathrm{ad}}(H)) = - \sum_{\alpha \in \Phi_{\bar{1}}^c} \alpha(H) - \sum_{\alpha \in \tau\Phi_{\bar{1}}^c} \alpha(H) = 0 \]
if and only if $\tau\Phi_{\bar{1}}^c = -\Phi_{\bar{1}}^c$. So 1. and 4. are equivalent and equivalence of 2. - 5. follows again from the decomposition
of $\mathfrak{p}$ into its reductive and nilpotent parts. 

\end{proof}

\subsection{Classification of measurable flag domains of even real forms}

Again the classification is done case by case.

Let $G = SL_{n \vert m}(\mathbb{C})$. Then there are the following even real forms:

\begin{enumerate}
 \item The unitary groups $G_\mathbb{R} = SU(p,n-p) \times SU(q, m-q) \times U(1)$.  
In this case the action on $\Sigma$ is given by $\tau(\alpha) = - \alpha$
for all $\alpha \in \Sigma$.
 \item If $m = 2k$ is even there is an even real form with $G_\mathbb{R} = SL_n(\mathbb{R}) \times SL_k(\mathbb{H}) \times \mathbb{R}^{>0}$.
The action on the root system is given by
$\tau(x_j - x_i) = x_{n-j+1} - x_{n-i+1}, \tau(y_j - y_i) = y_{m-j+1} - y_{m-i+1}, \tau(y_j - x_i) = y_{m-j+1} - x_{n-i+1}$.
 \item If $n = m$ there is an even real form with $G_\mathbb{R} = SL_n(\mathbb{C})$.
In that case the action on the root system is $\tau(x_j - x_i) = y_j - y_i, \tau(y_j-x_i) = x_j - y_i$. 
\end{enumerate}

Note that all given actions on the root system were already realized by real forms and that the action of $US\Pi(n)$ on the root system
is not realized by an even real form.

Now suppose $G = \mathrm{Osp}(n \vert 2m)$. Then there are the even real forms $G_\mathbb{R} = SO(p,q) \times Sp(2r,2s)$.
Unless $p$ and $q$ are both odd, the action on the root system is trivial and if $p$ and $q$ are both odd, the action on the root system
is identical to that of the real form $\mathrm{Osp}(p,q \vert 2m)$. If $n = 2k$ is even then there is another even real form
$G_\mathbb{R} = SO^*(2k) \times Sp_\mathbb{R}(2m)$ and its action on the root system is trivial. Here again, the results for the even
real forms of $D(2,1,\alpha)$ are identical to those for $D(2,1)$. Contrary to the case of ordinary real forms, the simple Lie superalgebras
of type $P$ and $Q$ are excluded from this list as they do not allow even real forms. 

\begin{satz}

Let $G$ be a basic classical complex reductive Lie supergroup, $Z = G/P$ a $G$-flag supermanifold, $G_\mathbb{R}$ an even real form and $\tilde{G}_\mathbb{R}$ the unique real form such that $G_\mathbb{R}$ and $\tilde{G}_\mathbb{R}$ act identically on the root system. Then a $G_\mathbb{R}$-flag domain $D$ in $Z$ is weakly or strongly measurable if and only if the corresponding $\tilde{G}_\mathbb{R}$-orbit $\tilde{D}$ in $Z$ has maximal odd dimension and is itself weakly or strongly measurable. 

\end{satz}

\begin{proof}
The condition for $D$ or $\tilde{D}$ to be weakly or stronlgy measurable depend on the action of $\tau$ on the root system and on the space $\mathfrak{p} \cap \tau \mathfrak{p}$. As these are identical for $G_\mathbb{R}$ and $\tilde{G}_\mathbb{R}$, the conditions for weak or strong measurability for $D$ and $\tilde{D}$ conincide.
\end{proof}

The correspondence for real forms and even real forms of the simple basic classical Lie supergroups  is given in the following table:

\noindent \begin{tabular}{ p{4cm} | p{6cm} }
 
 real form & even-real form  \\ \hline
  $SU(p,n-p \vert q,m-q)$ & $SU(p,n-p) \times SU(q,m-q) \times U(1)$ \\ \hline
  $SL_{n \vert m}(\mathbb{R}),SL_{k \vert l}(\mathbb{H})$ & $SL_{n}(\mathbb{R}) \times SL_l(\mathbb{H}) \times \mathbb{R}^{>0}$ \\ \hline
  $ ^0PQ(n)$ & $SL_{n}(\mathbb{C})$ \\ \hline
  $US\Pi(n)$  & none \\ \hline
  $\mathrm{Osp}(2p+1,2q+1 \vert 2m)$  & $SO(2p+1,2q+1) \times Sp(2m)$ \\ \hline

\end{tabular}

\section{Tables of measurable open orbits}

The three symmetry conditions are denoted by $\ev$, $\odd$ and three $\Pi$ respectively. For the even and odd symmetry condition,
the respective symmetrizability conditions are denoted by $\ev^*$ and $\odd^*$.

\begin{rem}
Note that the table below includes all information on the classical case as well:
The simple Lie algebras $A_{n-1}$ are realized by $\mathfrak{sl}_n(\mathbb{C})$ and the criteria for
measurability are analogous to the ones for the real forms of $A(n-1,m-1)$ given below. Moreover, there are
 identifications $B_n \cong B(n,0), D_n \cong D(n,0)$ and $C_m \cong D(0,m)$ and the exceptional simple
Lie algebras are explicitly contained in the table. 
\end{rem}

\medskip

\noindent \begin{tabular}{p{1.3cm} | p{3.7cm} | p{2.3cm} | p{2.5cm} | p{2cm}}
 
Type & real form & maximal odd dimension & weak measurability  & strong measurability \\ \hline
A(n,m) & $\mathfrak{sl}_{n+1 \vert m+1}(\mathbb{R}), \mathfrak{sl}_{k \vert l}(\mathbb{H})$ & $\ev^*$ & - & $\ev$ \\ \cline{2-5}
       & $\mathfrak{su}(p, n-p \vert q, m-q)$ & always & - & always \\ \hline
A(n,n) & $\mathfrak{psl}_{n+1 \vert n+1}(\mathbb{R}), \mathfrak{psl}_{k \vert k}(\mathbb{H})$ & $\ev^*$ & $\Pi$ or $\mathbb{P}(\mathbb{C}^{n \vert n})$ & $\ev$ \\ \cline{2-5}
       & $^0\mathfrak{pq}(n+1)$ & $\odd^*$ & $\Pi$ & $\odd$ \\ \cline{2-5}
       & $\mathfrak{us\pi}(n+1)$ & $\Pi$ & - & $\Pi$ \\ \hline
B(n,m) & $\mathfrak{osp}(p,q \vert 2m)$ & always & - & always \\ \hline
C(m)   & $\mathfrak{osp}(2 \vert 2r, 2s)$ & always & - & always \\ \hline
D(n,m) & $\mathfrak{osp}(2p,2q \vert 2m)$ & always & - & always \\ \cline{2-5}
       & $\mathfrak{osp}^*(2n \vert 2r, 2s)$ & always & - & always \\ \cline{2-5}
       & $\mathfrak{osp}(2p+1,2q+1 \vert 2m)$ & $n \vert d \not\in \delta, d < m$ & $ \exists  1 \leq d \leq  \min \{ n - 1, m \}: (n - d - 1 \vert m - d < n \vert m) \subseteq \delta $ & $n \vert m \not\in \delta$ or \newline $n-1 \vert m \in \delta$ \\ \hline 
D(2,1,$\alpha$) & any & as for D(2,1) \\ \hline
$E_6$ & $E_{6,C_4},E_{6,F_4}$ & - & - & $\sigma(J) = J$ \\ \cline{2-5}
 & others & - & - & always \\ \hline
$E_8,E_7$ & any & - & - & always \\ \hline
$F_4,F(4)$ & any & always & - & always \\ \hline
$G_2,G(3)$ & any & always & - & always \\ \hline
P(n)   & $\mathfrak{s\pi}_\mathbb{R}(n), \mathfrak{s\pi}^*(n)$ & $\Pi$ & - & $n = 2k$ and \newline $k \vert k \in \delta$ \\ \hline
Q(n)   & $\mathfrak{pq}_\mathbb{R}(n), \mathfrak{pq}_\mathbb{H}(k), \mathfrak{upq}(p,q)$ & always & always & as for $A_n$\\ \hline

\end{tabular}

\newpage
The respective table for the even real forms is the following:

 \label{Table3}
\noindent \begin{tabular}{p{1.3cm} | p{4cm}  | p{4cm} | p{4cm}}
 
Type & even real form & weak measurability  & strong measurability \\ \hline
A(n,m) & $\mathfrak{sl}_{n+1}(\mathbb{R}) \oplus \mathfrak{sl}_k(\mathbb{H}) \oplus \mathbb{R} $ & - & $\ev$ \\ \cline{2-4}
       & $\mathfrak{su}(p, n-p) \oplus \mathfrak{su}(q, m-q) \oplus \mathfrak{u}(1)$ & - & always \\ \hline
A(n,n) & $\mathfrak{sl}_{n+1}(\mathbb{R}) \oplus \mathfrak{sl}_{k}(\mathbb{H})$ & $\Pi$ or $\mathbb{P}(\mathbb{C}^{n \vert n})$ & $\ev$ \\ \cline{2-4}
       & $\mathfrak{sl}_{n+1}(\mathbb{C})$ & $\Pi$ & $\odd$ \\ \hline
B(n,m) & $\mathfrak{so}(p,q) \oplus \mathfrak{sp}(2r,2s)$ & - & always \\ \hline
C(m)   & $\mathfrak{so}(2) \oplus \mathfrak{sp}(2r, 2s)$ & - & always \\ \hline
D(n,m) & $\mathfrak{so}(2p,2q) \oplus \mathfrak{sp}(2r,2s)$ & - & always \\ \cline{2-4}
       & $\mathfrak{so}^*(2n) \times \mathfrak{sp}_\mathbb{R}(2m)$ & - & always \\ \cline{2-4}
       & $\mathfrak{so}(2p+1,2q+1) \oplus \mathfrak{sp}(2r,2s)$ & $ \exists  1 \leq d \leq  \min \{ n - 1, m \}: (n - d - 1 \vert m - d < n \vert m) \subseteq \delta $ & $n \vert m \not\in \delta$ or \newline $n-1 \vert m \in \delta$ \\ \hline 
D(2,1,$\alpha$) & any & as for D(2,1) 

\end{tabular}

\chapter{Flag domains II: Global holomorphic superfunctions}

In this chapter the classification of global holomorphic functions on flag manifolds due to Vishnyakova is extended to the case of flag domains. 

First the distinction between the possible types of bases is discussed: For cycle-connected bases, the theory developed in \cite{V} extends directly, whereas in the cases where hermitian holomorphic factors are involved within the bases, the projection onto the bounded symmetric domain is used as a complementary tool to compute the global holomorphic superfunctions. 

After this discussion the main results from \cite{V} are reviewed and their generalizations to cycle-connected flag domains are presented. These results are then used to classify the global holomorphic superfunctions on cycle-connected flag domains of real forms. 

Following this computation the projection onto the hermitian symmetric domain and a suitable notion of holomorphic reduction for split flag domains are used to execute an analogous computation in the case that the base of the flag domain contains hermitian holomorphic factors. 

In the end these results are extended to flag domain of even real forms. 
All classification results are summarized in tables at the end of the chapter.

\section{The two types of classical flag domains}

For a real form $G_{\bar{0}\mathbb{R}}$ of a complex simple Lie group $G$ there are exactly two possible types of flag domains:

\begin{description}
\item[The hermitian holomorphic case:] $D$ projects onto a bounded symmetric domain. This fibration is a trivial fibre bundle and its typical fibre is the base cycle. In this case $G_{\bar{0}\mathbb{R}}$ is always a hermitian real form.
\item[The cycle-connected case:] Every holomorphic function on $D$ is constant. 
\end{description}

In the supersymmetric setting, the base $D_{\bar{0}}$ will usually be a product of two classical flag domains and all possible combinations of the two above classical cases actually arise as bases of flag domains. As it turns out the classification of global odd functions on flag supermanifolds due to Vishnyakova extends directly to flag domains $D$ with cycle-connected bases. 
   
 To classify the global holomorphic superfunctions on flag domains whose bases are not cycle-connected, the projection onto the symmetric domain plays an important role. Classically, it is in fact the holomorphic reduction of $D_{\bar{0}}$. In the supersymmetric case it is not all clear how a holomorphic reduction should be defined. However in the case of a split flag domain there is a natural construction that is also useful for the characterization of global holomorphic superfunctions:

Let $D = G_\mathbb{R}/L_\mathbb{R}$ be a split flag domain. Then $D_{\bar{0}} = G_{\bar{0}\mathbb{R}}/L_{\bar{0}\mathbb{R}}$ and the holomorphic reduction $Y_{\bar{0}} = G_{\bar{0}\mathbb{R}}/J_{\bar{0}}$ is again a homogeneous space. 
Suppose $\mathfrak{g}$ is a split Lie superalgebra, i.e. $[\cdot,\cdot]_\mathfrak{g} \vert_{\mathfrak{g}_{\bar{1}} \times \mathfrak{g}_{\bar{1}}} = 0$.
Let $\mathfrak{j}_{\bar{1}} = \mathrm{span}_{J_{\bar{0}}}(\mathfrak{l}_{\bar{1}\mathbb{R}})$ and $\mathfrak{j} = \mathfrak{j}_{\bar{0}} \oplus \mathfrak{j}_{\bar{1}}$.
Then $(J_{\bar{0}},\mathfrak{j})$ is a sub-SHCP of $(G_{\bar{0}\mathbb{R}},\mathfrak{g}_\mathbb{R})$ and the holomorphic reduction of $D$ is $Y = G/J$, where
$J$ is the Lie supergroup defined by the SHCP $(J_{\bar{0}},\mathfrak{j})$.

Given a not necessarily split flag domain $D$ the holomorphic reduction can be used to compute the space of holomorphic functions on $\mathrm{gr} D$ which provides an upper bound for the space of holomorphic functions on $D$.

\begin{rem}

The holomorphic reduction of a complex manifold $X_{\bar{0}}$ is classically defined to be $Y_{\bar{0}} = X_{\bar{0}}/\tild$, where $x \tild y$ if and only if $f(x) = f(y)$ for all $f \in \mathcal{F}(X_{\bar{0}})$. If the analogous definition is used in the supercase, i.e.  $x \tild y$ if and only if $f(x) = f(y)$ for all $f \in \mathcal{O}(X_{\bar{0}})$, it only produces the classical holomorphic reduction of the base, as the numerical value of every odd function is zero. One possible improvement is the following definition:

Let $X$ be a complex supermanifold, $Y_{\bar{0}}$ the holomorphic reduction of $X_{\bar{0}}$ and $\varphi: X_{\bar{0}} \rightarrow Y_{\bar{0}}$ the canonical projection. Then the holomorphic reduction of $X$ is $Y = (Y_{\bar{0}}, \varphi_* \mathcal{O}_X)$.

The problem with this definition is that in general the sheaf $\varphi_* \mathcal{O}_X$ will have unfavorable properties, e.g. it may not be coherent.

Another idea is to change the equivalence relation slightly and set $x \tild y$ if and only if $Df(x) = Df(y)$ for all $f \in \mathcal{O}_X(X_{\bar{0}}), D \in \mathcal{U}(\mathfrak{g})$.
Especially this latter idea might be a good starting point for further investigation.

\end{rem}

\section{Reformulation of Vishnyakova's main results}

Let $G_\mathbb{R}$ be a Lie supergroup, $G$ its complexification and $P \subseteq G$ a subgroup
such that $Z = G/P$ is compact. Suppose $D = G_\mathbb{R}/L_\mathbb{R}$ is an open submanifold of $Z$.
In that case $\mathfrak{p}$ is a parabolic subalgebra of $\mathfrak{g}$ and $\mathfrak{l} = 
(\mathfrak{l}_\mathbb{R})^\mathbb{C}$ is the Levi component of $\mathfrak{p}$. 

It is shown in \cite{V} that if $H^0(G_{\bar{0}}/P_{\bar{0}}, \mathrm{gr} \mathcal{O}) = \mathbb{C}$ then $H^0(G_{\bar{0}}/P_{\bar{0}}, \mathcal{O}) = \mathbb{C}$.
The proof can be extended to the following statement on open real orbits:

\begin{lemma}\label{Lemma6V}

Let $D = G_\mathbb{R}/L_\mathbb{R}$ be a flag domain with $H^0(D_{\bar{0}},\mathrm{gr}\mathcal{O}_D) = H^0(D_{\bar{0}},\mathcal{F})$.
Then $H^0(D,\mathcal{O}_D) = H^0(D_{\bar{0}},\mathcal{F})$.  

\end{lemma}

\begin{proof}

As in the original proof in \cite{V} one considers the short exact sequences

\[ 0 \rightarrow \mathcal{J}^{p+1} \rightarrow \mathcal{J}^p \rightarrow (\mathrm{gr}\mathcal{O}_D)_p \rightarrow 0 \]

By assumption $H^0(D_{\bar{0}},(\mathrm{gr}\mathcal{O}_D)_p) = 0$ for $p > 0$. This implies $H^0(D_{\bar{0}},\mathcal{J}^p) = 0$ for all
$p > 0$ and $H^0(D_{\bar{0}}, \mathcal{O}_D) = H^0(D_{\bar{0}},(\mathrm{gr}\mathcal{O}_D)_0) = H^0(D_{\bar{0}},\mathcal{F})$. 

\end{proof}

Suppose for now that $Z$ is a split supermanifold. In that case the linear odd functions are given by sections of a homogeneous vector bundle. For the characterization of global holomorphic superfunctions a useful fact is that $\mathbb{E}$ is always embedded in a trivial vector bundle $\mathbb{V}$ (Proposition 1 in \cite{V}). By restriction this also holds for flag domains: 

\begin{lemma}
 
Let $\tilde{\mathcal{E}}$ be a $G_{\bar{0}}$-homogeneous vector bundle on $Z_{\bar{0}}$ such that $\bigwedge \tilde{\mathcal{E}}$ is the structure sheaf
of a split homogeneous superspace and let $\mathcal{E} = \tilde{\mathcal{E}} \vert_{D_{\bar{0}}}$. 
Then $\mathcal{E}$ is a subbundle of a trivial $G_{\bar{0}\mathbb{R}}$-bundle $\mathcal{V}$.
Moreover every split $G_\mathbb{R}$-flag domain arises in this way.

\end{lemma}

\begin{proof}
 
According to \cite{V}, $\tilde{\mathcal{E}}$ is a subbundle of a trivial $G_{\bar{0}}$-bundle $\tilde{\mathcal{V}}$. Denote its fibre by $V$. Then $V$ is also a $G_{\bar{0}\mathbb{R}}$-module and $\mathcal{E}$ is naturally a subbundle
of $\mathcal{V} = \mathcal{O}(G_{\bar{0}\mathbb{R}} \times_{L_{\bar{0}\mathbb{R}}} V)$. The proof of the second claim is immediate as every split $G$-flag manifold arises in this way according to \cite{V}.  

\end{proof}

The main ingredient of Vishnyakova's characterization theorem is the following fact on sections of subbundles of trivial vector bundles.

\begin{lemma}[Lemma 3 in \cite{V}]

Let $G_{\bar{0}}$ be a complex Lie group, $H_{\bar{0}}$ a closed complex Lie subgroup, $V$ a $G_{\bar{0}}$-module and $E \subseteq V$ an $H_{\bar{0}}$-submodule. Assume that $G_{\bar{0}}/H_{\bar{0}}$ is compact. Let $\mathbb{E} = G_{\bar{0}} \times_{H_{\bar{0}}} E$.  Then the following are equivalent:

\begin{enumerate}
 \item Non-trivial $G_{\bar{0}}$-submodules $W$ of $E$ do not exist.
 \item $\mathcal{O}(G_{\bar{0}}/H_{\bar{0}}, \mathbb{E}) = 0$
\end{enumerate}
 
\end{lemma}

Now assume that $G_{\bar{0}\mathbb{R}}/L_{\bar{0}\mathbb{R}}$ is an open submanifold of $G_{\bar{0}}/H_{\bar{0}}$ such that $H^0(G_{\bar{0}\mathbb{R}}/L_{\bar{0}\mathbb{R}}, \mathcal{F}) = \mathbb{C}$. 
In that case Vishynakova's lemma extends directly:

\begin{lemma}\label{Lemma3V}

Let $ G_{\bar{0}}, H_{\bar{0}}$ as before. Let $D_{\bar{0}} = G_{\bar{0}\mathbb{R}}/L_{\bar{0}\mathbb{R}} \subseteq G_{\bar{0}}/H_{\bar{0}}$ be a flag domain such that $H^0(D_{\bar{0}},\mathcal{F}) = \mathbb{C}$. Let $E$ be the typical fibre of the homogeneous bundle $\mathcal{E} = \mathcal{O}( G_{\bar{0}\mathbb{R}} \times_{L_{\bar{0}\mathbb{R}}} E)$ and let $\mathcal{V} = \mathcal{O}(\mathbb{V})$ be a trivial vector bundle on $D_{\bar{0}}$ such that $\mathbb{E} \subseteq \mathbb{V}$. Then the following are equivalent:

\begin{enumerate}
 \item Non-trivial $G_{\bar{0}\mathbb{R}}$-submodules $W$ of $E$ do not exist.
 \item $\mathcal{O}(D, \mathbb{E}) = 0$
\end{enumerate}
 
\end{lemma}

\begin{proof}

Sections of $\mathbb{E}$ can be identified with $L_\mathbb{R}$-equivariant maps $s: G_{\bar{0}\mathbb{R}} \rightarrow E$. Such a map $s$ is holomorphic, if $Xs = 0$ for all $X \in \mathfrak{n}$ where
$\mathfrak{n}$ denotes the nilpotent radical of $\mathfrak{p}$.

Assume that $\mathcal{O}(D,\mathbb{E}) \neq 0$ and let $\mathcal{V} = \mathcal{O}(\mathbb{V})$ be the trivial bundle containing $\mathcal{E}$ as a subbundle and let $V$ be its typical fibre. As $H^0(D_{\bar{0}}, \mathcal{F}) = \mathbb{C}$, there is an isomorphism of $G_{\bar{0}\mathbb{R}}$-modules $\mathcal{O}(D, \mathbb{V}) \cong V$. As $\mathcal{E}$ is a subbundle of $\mathcal{V}$, $\mathcal{O}(D,\mathbb{E})$ is a $G_{\bar{0}\mathbb{R}}$-submodule of $\mathcal{O}(D,\mathbb{V})$, so its image $W$ in $V$ is a non-trivial $G_{\bar{0}\mathbb{R}}$-submodule of $E$.

Conversely assume there is a non-trivial $G_{\bar{0}\mathbb{R}}$-submodule $W \subseteq E$. It gives rise to a trivial subbundle $\mathcal{W}$ of $\mathcal{E}$ and therefore $\mathcal{O}(D,\mathbb{E}) \neq 0$.  

\end{proof}

The last necessary ingredient for the proof of the characterization theorem is the following lemma.

\begin{lemma}[Lemma 5 in \cite{V}]

Let $X$ be a split homogeneous supermanifold and $\mathcal{O}_X \cong \bigwedge \mathcal{E}$. If $H^0(X_{\bar{0}}, \mathcal{E}) = 0$, then  $H^0(X_{\bar{0}}, \bigwedge^p \mathcal{E}) = 0$ for all $p \in \mathbb{N}$.

\end{lemma}

The main characterization results are the following two theorems:

\begin{satz}[Theorem 3 in \cite{V}]\label{SplitSec}
 
Let $M = G/H$ be a $G$-homogeneous supermanifold, $M_{\bar{0}}$ a compact connected manifold, $\mathfrak{g} = \mathrm{Lie}(G), \mathfrak{h} = \mathrm{Lie}(H)$. Consider the exact sequence of $H_{\bar{0}}$-modules:

\[ \xymatrix{ 0 \ar[r] & \mathfrak{h}_{\bar{1}} \ar[r]^\delta & \mathfrak{g}_{\bar{1}} \ar[r]^\gamma & \mathfrak{g}_{\bar{1}}/\mathfrak{h}_{\bar{1}} \ar[r] & 0 } \]

If there do not exist non-trivial $G_{\bar{0}}$-submodules $W \subseteq \mathfrak{g}_{\bar{1}}^*$ such that $W \subseteq \mathrm{Im} \gamma^*$, then $H^0(M, \mathcal{O}) = \mathbb{C}$. If $M$ is split, the converse is also true.

\end{satz}

In particular if $\mathfrak{g}_{\bar{1}}$ is an irreducible $G_{\bar{0}}$-module, then nontrivial global odd functions on $M = G/H$
can only exist if $H$ is a purely even subgroup. In that case $H^0(M,\mathcal{O}) \cong \bigwedge \mathfrak{g}_{\bar{1}}^*$.
If $\mathfrak{g}_{\bar{1}}$ is not an irreducible $G_{\bar{0}}$-module, then the global odd functions are characterized by the following theorem:

\begin{satz}[Proposition 2 in \cite{V}]
 
Let $M = G/H$ be a $G$-homogeneous supermanifold, $M_{\bar{0}}$ a compact connected manifold, $\mathfrak{g} = \mathrm{Lie}(G), 
\mathfrak{h} = \mathrm{Lie}(H)$. Assume that $\mathfrak{g}_{\bar{1}}$ is a
completely reducible $G_{\bar{0}}$-module. Consider the exact sequence of $H_{\bar{0}}$-modules:

\[ \xymatrix{ 0 \ar[r] & \mathfrak{h}_{\bar{1}} \ar[r]^\delta & \mathfrak{g}_{\bar{1}} \ar[r]^\gamma & \mathfrak{g}_{\bar{1}}/\mathfrak{h}_{\bar{1}} \ar[r] & 0 } \]

Let $W \subseteq \mathrm{Im} \gamma^*$ be the maximal $G_{\bar{0}}$-module and let $Y = \{y \in \mathfrak{g}_1 \vert W(y) = 0\}$.
If $\delta(\mathfrak{h}_{\bar{1}}) \subseteq Y$, then $H^0(M, \mathcal{O}) \cong \bigwedge W$. If in addition $M$ is split,
then $ M \cong N \times (\mathrm{pt}, \bigwedge W )$, where $N$ is a homogeneous
supermanifold such that $H^0(N, \mathcal{O}) = \mathbb{C}$.

\end{satz}

It turns out that the proof of the above characterization results extends to the case of a cycle-connected flag domain, i.e. an open $G_\mathbb{R}$-orbit $D \subseteq G/P$ such that $H^0(D, \mathcal{F}) = \mathbb{C}$:

\begin{satz}
 
Let $Z = G/P$ be a compact $G$-homogeneous supermanifold, $G_\mathbb{R}$ a real form of $G$ and $D = G_\mathbb{R}/L_\mathbb{R} \subseteq Z$ a flag domain. Let $\mathfrak{g}_\mathbb{R} = \mathrm{Lie}(G_\mathbb{R})$ and $ \mathfrak{l}_\mathbb{R} = \mathrm{Lie}(L_\mathbb{R})$. Consider the following exact sequence of $L_{\bar{0}\mathbb{R}}$-modules:

\[ \xymatrix{ 0 \ar[r] & \mathfrak{l}_{\bar{1}\mathbb{R}} \ar[r]^\delta & \mathfrak{g}_{\bar{1}\mathbb{R}} \ar[r]^\gamma & \mathfrak{g}_{\bar{1}\mathbb{R}}/\mathfrak{l}_{\bar{1}\mathbb{R}} \ar[r] & 0 } \]

If there do not exist non-trivial $G_{\bar{0}\mathbb{R}}$-submodules $W \subseteq \mathfrak{g}_{\bar{1}\mathbb{R}}^*$ such that $W \subseteq \mathrm{Im} \gamma^*$, then $H^0(D, \mathcal{O}_D) = \mathbb{C}$. If $D$ is split, the converse is also true.

\end{satz}

\begin{proof}

According to Lemma \ref{OnFunctions} $\mathrm{gr} \mathcal{O}_D $ is isomorphic to \\ $ \bigwedge \mathcal{F}(G_{\bar{0}\mathbb{R}} \times_{L_{\bar{0}\mathbb{R}}} (\mathfrak{g}_{\bar{1}\mathbb{R}}/\mathfrak{l}_{\bar{1}\mathbb{R}})^*)$. The theorem therefore follows from Lemma \ref{Lemma3V} applied to $E = \mathfrak{g}_{\bar{1}\mathbb{R}}/\mathfrak{l}_{\bar{1}\mathbb{R}})^*$ and Lemma \ref{Lemma6V}.

\end{proof}

\begin{satz}

Let $G_\mathbb{R}/L_\mathbb{R} = D \subseteq Z = G/P$ be a flag domain such that $H^0(D,\mathcal{F}) = \mathbb{C}$.
Assume that $\mathfrak{g}_{\bar{1}\mathbb{R}}$ is a completely reducible $G_{\bar{0}\mathbb{R}}$-module. Consider the following short exact sequence of $L_{\bar{0}\mathbb{R}}$-modules:

\[ \xymatrix{ 0 \ar[r] & \mathfrak{l}_{\bar{1}\mathbb{R}} \ar[r]^\delta & \mathfrak{g}_{\bar{1}\mathbb{R}} \ar[r]^\gamma &
 \mathfrak{g}_{\bar{1}\mathbb{R}}/\mathfrak{l}_{\bar{1}\mathbb{R}} \ar[r] & 0 } \]
 
Let $W \subseteq \mathrm{Im} \gamma^*$ be the maximal $G_{\bar{0}\mathbb{R}}$-module and let $Y = \{y \in \mathfrak{g}_{\bar{1}\mathbb{R}} \vert W(y) = 0\}$.
If $\delta(\mathfrak{l}_{\bar{1}\mathbb{R}}) \subseteq Y$, then $H^0(D, \mathcal{O}) \cong \bigwedge W$.

\end{satz}

\begin{proof}

Consider the SHCP $(G_{\bar{0}\mathbb{R}}, \mathfrak{g}_{\bar{0}\mathbb{R}} \oplus Y)$ and
denote by $G_1$ the real Lie supergroup defined by this SHCP. Then $G_1$ is a Lie subsupergroup
of $G_\mathbb{R}$, containing $L_\mathbb{R}$ by assumption ,and the projection $D = G_\mathbb{R}/L_\mathbb{R} \rightarrow G_\mathbb{R}/G_1$ induces an injective homomorphism $ \bigwedge W = H^0(G_\mathbb{R}/G_1, \mathcal{O}) \rightarrow H^0(D, \mathcal{O})$. 

Moreover by virtue of Theorem \ref{SplitSec} $H^0(D, \mathrm{gr} \mathcal{O}) \cong \bigwedge W$. As this is an upper bound for $H^0(D, \mathcal{O})$, that space is isomorphic to $\bigwedge W$. 

\end{proof}

The full classification of global odd functions on flag supermanifolds is given by the following theorem:

\begin{satz}[Theorems 5-8 in \cite{V}]

Let $G$ be a classical simple complex Lie superalgebra and $Z(\delta)$ a $G$-flag manifold.

\begin{enumerate}
 \item If $G = SL_{n \vert m}(\mathbb{C})$ then $H^0(Z(\delta), \mathcal{O}) = \mathbb{C}$, if
       $n \vert 0, 0 \vert m \not\in \delta$ \\ and $H^0(Z(\delta), \mathcal{O}) \cong \bigwedge \mathbb{C}^{nm}$ otherwise.
 \item If $G = \mathrm{Osp}(n \vert 2m)$ then $H^0(Z(\delta), \mathcal{O}) = \mathbb{C}$, unless $n = 2$ and $1 \vert 0 \in \delta$.
       In that case $H^0(Z(\delta), \mathcal{O}) \cong \bigwedge \mathbb{C}^{2m}$
 \item If $G = P(n)$ and $n \vert 0 \in \delta$ then $H^0(Z(\delta), \mathcal{O}) \cong \bigwedge \mathbb{C}^{\frac{1}{2}n(n+1)}$.
       Moreover if $0 \vert n \in \delta$ or $0 \vert n-1 \in \delta$ then $H^0(Z(\delta), \mathcal{O}) \cong \bigwedge \mathbb{C}^{\frac{1}{2}n(n-1)}$
       In all other cases $H^0(Z(\delta), \mathcal{O}) = \mathbb{C}$.
 \item If $G = Q(n)$ then $H^0(Z(\delta), \mathcal{O}) = \mathbb{C}$ always.
\end{enumerate}

\end{satz} 

In the following sections this classification theorem is extended to flag domains.

\section{Holomorphic superfunctions on flag domains: The cycle-connected Case}

This section uses the classification of complex simple Lie superalgebras and their real forms given in Chapters 2 and 3 respectively.
Before analyzing the complex simple Lie supergroups and their real forms case by case some general observations can be made:

\begin{satz}

Let $G$ be a complex simple Lie supergroup, $G_\mathbb{R}$ a real form, $Z = G/P$ and $D = G_\mathbb{R}/L_\mathbb{R}$ an open
orbit in $Z$. 

\begin{enumerate}
 \item If there are global holomorphic functions on $Z$, then these restrict to $D$. 
 \item If $\mathfrak{p}_{\bar{1}} = \mathfrak{b}_{\bar{1}}$ for a Borel subalgebra $\mathfrak{b} \subseteq \mathfrak{g}$
       then $\mathfrak{l}_{\bar{1}} = 0$ and therefore\newline $H^0(D,\mathcal{O}) \supseteq \bigwedge \mathfrak{g}_{\bar{1}\mathbb{R}}^*$
 \item If the action of the defining involution $\tau$ of $G_\mathbb{R}$ on the simple roots is trivial then
       $\mathfrak{g}_{\bar{1}\mathbb{R}}^*$ is an irreducible $G_{\bar{0}\mathbb{R}}$-module. Consequently in this case
       global odd functions can only exist if $\mathfrak{l}_{\bar{1}} = 0$ or if there are non-constant holomorphic functions on $D_{\bar{0}}$.
\end{enumerate}

\end{satz}

\begin{proof}

The first assertion is trivial.

If $\mathfrak{l}_{\bar{1}} = 0$ then $(\mathfrak{g}_{\bar{1}\mathbb{R}}/\mathfrak{l}_{\bar{1}\mathbb{R}})^* = \mathfrak{g}_{\bar{1}\mathbb{R}}^*$, which is itself a $G_{\bar{0}\mathbb{R}}$-module.
Consequently $H^0(D_{\bar{0}}, \mathcal{O}_D) \supseteq \bigwedge \mathfrak{g}_{1\mathbb{R}}^*$

Let $\mathfrak{g}$ be a simple classical Lie superalgebra of type I, i.e. $\mathfrak{g}_{\bar{1}}$ is a completely reducible $G_{\bar{0}}$-module with two
invariant submodules $\mathfrak{g}_1$ and $\mathfrak{g}_{-1}$. For a suitable choice of simple roots $\Pi$, these invariant subspaces are the direct sums of all positive and all negative root spaces respectively. If $\tau$ acts trivially on $\Pi$, then it interchanges $\mathfrak{g}_1$ and $\mathfrak{g}_{-1}$ and therefore $\mathfrak{g}_{\bar{1}\mathbb{R}}$ is irreducible. Thus
non-constant global odd functions on $D$ can only exist if $\mathfrak{l}_{\bar{1}} = 0$ or if there are non-constant holomorphic functions on $D_{\bar{0}}$. 

\end{proof}

In this section only the cycle-connected case is analyzed, i.e. only flag domains $D$ satisfying $H^0(D,\mathcal{F}) = \mathbb{C}$
are considered.

The cases in which there are non-constant global holomorphic superfunctions on $Z$ were classified in \cite{V}, and we assume from now on that $H^0(Z,\mathcal{O}) = \mathbb{C}$. 

\subsection{Type A}

The simple Lie supergroups of type $A(n,m)$ are $G = SL_{n \vert m}(\mathbb{C})$ if $n \neq m$ and $G = PSL_{n \vert n}(\mathbb{C})$ if $n = m$. These complex Lie suprgroups have four classes of real forms each yielding a different action of $\tau$ on the simple roots:

\begin{enumerate}
 \item The super-unitary groups $G_\mathbb{R} = SU(p,q \vert r,s)$. In that case the action of $\tau$ on $\Sigma$ is trivial.
 \item The real and quaternionic groups $G_\mathbb{R} = SL_{n \vert m}(\mathbb{R})$ and $G_\mathbb{R} = SL_{k \vert l}(\mathbb{H})$.
       In that case the action is given by $\tau(x_i) = x_{n-i+1}, \tau(y_j) = y_{m-j+1}$.
 \item If $n = m$ there is the real form $G_\mathbb{R} = \mbox{}^0PQ(n)$ with its involution acting by $\tau(x_i) = y_i$
 \item In case $n = m$ there is also the real form $G_\mathbb{R} = US\Pi(n)$ with the action $\tau(x_i) = -y_i$
\end{enumerate}

In the first and in the third case $\mathfrak{g}_{\bar{1}\mathbb{R}}$ is an irreducible $G_{\bar{0}\mathbb{R}}$-module, whereas in the second and fourth case it decomposes into two irreducible $G_{\bar{0}\mathbb{R}}$-modules, which are precisely the intersections of $\mathfrak{g}_{1\mathbb{R}}$ with the two irreducible components $\mathfrak{g}_1$ and $\mathfrak{g}_{-1}$ of $\mathfrak{g}_{\bar{1}}$. The first two results are therefore immediate:

\begin{satz}

Let $G = SL_{n \vert m}(\mathbb{C}), G_\mathbb{R} = SU(p,q \vert r,s)$ and $G_\mathbb{R}/L_\mathbb{R} = D \subseteq Z = G/P$ 
an open orbit. Suppose that $\mathfrak{l}_{\bar{1}} \neq 0$ and $D_{\bar{0}}$ is cycle-connected.\newline  Then $H^0(D, \mathcal{O}) = \mathbb{C}$.  

\end{satz}

\begin{satz}

Let $G = SL_{n \vert m}(\mathbb{C}), G_\mathbb{R} = \mbox{}^0PQ(n)$ and $G_\mathbb{R}/L_\mathbb{R} = D \subseteq Z = G/P$ 
an open orbit. Then $H^0(D, \mathcal{O}) = \mathbb{C}$ unless $\mathfrak{l}_{\bar{1}} = 0$.

\end{satz}

\begin{proof}
 
Note that $G_{\bar{0}\mathbb{R}} = SL_n(\mathbb{C})$ viewed as a real form of $G_{\bar{0}} = SL_n(\mathbb{C}) \times SL_n(\mathbb{C})$.
According to the classical theory all open orbits of this real form are cycle-connected which is all that needed to be shown. 

\end{proof}

In the case $G_\mathbb{R} = US\Pi(n)$ once again $G_{\bar{0}\mathbb{R}} = SL_n(\mathbb{C})$ so all underlying open orbits are cycle-connected.
Furthermore in that case the odd roots $\alpha_i = x_i - y_i$ and $\beta_i = y_i - x_i$ are fixed points of the involution $\tau$.
Therefore, in order for $D$ to have maximal odd dimension, one needs to have $\mathfrak{g}^{\alpha_i}, \mathfrak{g}^{\beta_i} 
\subseteq \mathfrak{p}$ for all $1 \leq i \leq n$. As $\mathfrak{g}^{\alpha_i} \subseteq \mathfrak{g}_{-1}$ and $\mathfrak{g}^{\beta_i} \subseteq \mathfrak{g}_{1}$   
for all $1 \leq i \leq n$, $\mathfrak{l}$ never has zero intersection with any of the two irreducible components of 
$\mathfrak{g}_{\bar{1}\mathbb{R}}$. Therefore

\begin{satz}

Let $G = SL_{n \vert m}(\mathbb{C}), G_\mathbb{R} = US\Pi(n)$ and $G_\mathbb{R}/L_\mathbb{R} = D \subseteq Z = G/P$ 
an open orbit. Then $H^0(D, \mathcal{O}) = \mathbb{C}$. 

\end{satz}

\subsubsection{The real and quaternionic linear supergroups}

Recall from Chapter 3 the condition for a $G_\mathbb{R}$-orbit with open base $D_{\bar{0}}$ to have maximal odd dimension:
$D \subseteq Z(\delta)$ is open if and only if $\delta$ is even-symmetrizable.

\noindent The global odd functions on $D = G_\mathbb{R}/L_{\mathbb{R}}$ are given by the maximal $G_{\bar{0}\mathbb{R}}$-submodule of
$(\mathfrak{g}_{{\bar{1}}\mathbb{R}}/\mathfrak{l}_{\bar{1}\mathbb{R}})^*$. The odd part $\mathfrak{g}_{\bar{1}\mathbb{R}}$ is a completely reducible $G_{\bar{0}\mathbb{R}}$-module and that the two irreducible submodules are given by $\mathfrak{g}_{-1} = \mathrm{Hom}(V_0,V_1)$ and  $\mathfrak{g}_1 = \mathrm{Hom}(V_1,V_0)$. The respective decompositions into weight spaces are

\[  \mathrm{Hom}(V_0,V_1) = \bigoplus_{1 \leq i \leq n, 1 \leq j \leq m} \mathfrak{g}^{x_i - y_j} \]
\[  \mathrm{Hom}(V_1,V_0) = \bigoplus_{1 \leq i \leq n, 1 \leq j \leq m} \mathfrak{g}^{y_j - x_i} \]

\noindent and the action of $\tau$ on the weights is $\tau(\pm (x_i - y_j)) = \pm (x_{n-i+1} - y_{m-j+1})$.

$(\mathfrak{g}_{\bar{1}\mathbb{R}}/\mathfrak{l}_{\bar{1}\mathbb{R}})^*$ contains a non-trivial  $G_{\bar{0}\mathbb{R}}$-submodule if and
only if $\mathfrak{l}_{\bar{1}} \cap \mathfrak{g}_1 = 0$ or $\mathfrak{l}_{\bar{1}} \cap \mathfrak{g}_{-1} = 0$. 

Let $i \in \{1,\ldots,n\}$, $j \in \{1,\ldots,m\}$ and define $\overline{d_{\bar{0}}}(j) = \min \{ 1 \leq d_{\bar{0}} \leq n : d_{\bar{0}} \vert d_{\bar{1}} \in \delta \ \textnormal{and} \ j \leq d_{\bar{1}}  \}$ and $\overline{d_{\bar{1}}}(i) = \min \{ 1 \leq d_{\bar{1}} \leq m : d_{\bar{0}} \vert d_{\bar{1}} \in \delta \ \textnormal{and} \ i \leq d_{\bar{0}}  \}$.

\begin{lemma}

The conditions for the existence of global holomorphic functions on $D$ are characterized as follows:

\begin{enumerate}
\item $\mathfrak{l}_{\bar{1}\mathbb{R}} \cap \mathfrak{g}_{-1} = 0$ if and only if $\overline{d_{\bar{0}}}(j) \leq m - \overline{d_{\bar{0}}}(m-j+1)$ for all $1 \leq j \leq m$(condition I)
\item $\mathfrak{l}_{\bar{1}\mathbb{R}} \cap \mathfrak{g}_{1} = 0$ if and only if $\overline{d_{\bar{1}}}(i) \leq n - \overline{d_{\bar{1}}}(n-i+1)$ for all $1 \leq i \leq n$(condition II)
\end{enumerate}

\end{lemma}

\begin{proof}

By virtue of the above decomposition $(\mathfrak{p} \cap \tau\mathfrak{p}) \cap \mathfrak{g}_{-1} = 0$ if and only if $x_i - y_j \not\in \Phi \cap \tau\Phi \ \forall 1 \leq i \leq n, 1 \leq j \leq m$. For
$G_\mathbb{R} = SL_{n \vert m}(\mathbb{R})$ one has $\tau(x_i - y_j) = x_{n-i+1} - y_{m-j+1}$.

Let $e_1, \ldots, e_n,f_1, \ldots, f_m$ be a basis of $\mathbb{C}^{n \vert m}$  such that  the subspaces occuring in a flag fixed by $\mathfrak{p}$ are spanned by this basis. Then $x_i - y_j \in \Phi$
if and only if there is some $X \in \mathfrak{p}$ such that $X(f_j) = e_i$. Analogously $x_{n-i+1} - y_{m-j+1} \in \Phi$ if and only if there is a $Y \in \mathfrak{p}$ such that $Y(f_{m-j+1}) = e_{n-i+1}$.

Then $x_i - y_j \in \Phi$ if and only if $d_{\bar{0}}(j) \geq i$ and analogously $x_{n-i+1} - y_{m-j+1} \in \Phi$ if and only if $d_{\bar{0}}(m-j+1) \geq n-i+1$.

The lemma claims that this is equivalent to $\overline{d_{\bar{0}}}(j) < n - \overline{d_{\bar{0}}} + 1$ for all $1 \leq j \leq m$.

Suppose this is the case. Then $\overline{d_{\bar{0}}}(j) + \overline{d_{\bar{0}}}(m-j+1) < n+1$.
Assume furthermore that $x_i - y_j \in \Phi \cap \tau\Phi$ for some $1 \leq i \leq n$ and $1 \leq j \leq m$.
Then $\overline{d_{\bar{0}}}(j) \geq i$ and $\overline{d_{\bar{0}}}(m-j+1) \geq n-i+1$ and therefore $\overline{d_{\bar{0}}}(j) + \overline{d_{\bar{0}}}(m-j+1) \geq n+1$, which is a contradiction.

Conversely suppose that there is $1 \leq j \leq m$ such that $\overline{d_{\bar{0}}}(j) \geq n - \overline{d_{\bar{0}}}(m-j+1) + 1$. Choose $i = \overline{d_{\bar{0}}}(j)$. Then $x_i-y_j \in \Phi$ by construction. Moreover $\tau(x_i - y_j) = x_{n-i+1} - y_{m-j+1}$ and $n-i+1 = n - \overline{d_{\bar{0}}}(j) +1 \leq \overline{d_{\bar{0}}}(m-j+1)$. This implies
$x_{n-i+1} - y_{m-j+1} \in \Phi$ and therefore $x_i - y_j \in \Phi \cap \tau\Phi$.

The proof of the second part is completely analogous. 

\end{proof} 

If conditions I and II are both fulfilled then $D$ is automatically measurable, however the converse need not always be true. 
Furthermore both condition I and condition II require either $m$, or $n$ to be even.

The final result is therefore the following:

\begin{satz}
Let $G = SL_{n \vert m}(\mathbb{C}), G_\mathbb{R} = SL_{n \vert m}(\mathbb{R})$ or $G_\mathbb{R} = SL_{k \vert l}(\mathbb{H})$ and $Z(\delta)$ a $G$-flag manifold. then the following is true:
\begin{enumerate}
\item If neither condition I nor condition II is fulfilled, then $H^0(D, \mathcal{O}) \cong \mathbb{C}$
\item If condition I is satisfied then $H^0(D, \mathcal{O}) \cong \bigwedge \mathbb{C}^{\frac{1}{2}nm}$
\item If condition II is satisfied then $H^0(D, \mathcal{O}) \cong \bigwedge \mathbb{C}^{\frac{1}{2}nm}$
\item If conditions I and II are both satisfied then $H^0(D, \mathcal{O}) \cong \bigwedge \mathbb{C}^{nm}$
\end{enumerate} 
\end{satz}

\subsubsection{Final result for Type A}

Summing up these results the global holomorphic superfunctions on \\ cycle-connected flag domains of Type $A$ are given as follows:

\begin{satz}

Let $G = SL_{n \vert m}(\mathbb{C}), G_\mathbb{R}$ a real form, $Z = Z(\delta) = G/P$ a flag supermanifold and $D \subseteq Z$ a $G_\mathbb{R}$-flag domain with $H^0(D_{\bar{0}}, \mathcal{F}) = \mathbb{C}$.

\begin{enumerate}
\item If $G_\mathbb{R}$ is any real form and $\dim L_{\bar{1}} = \dim (\mathfrak{p} \cap \tau\mathfrak{p})_{\bar{1}} = 0$, then $H^0(D, \mathcal{O}) \cong \bigwedge \mathbb{C}^{nm}$
\item If $G_\mathbb{R} = SU(p,q \vert r,s)$ or $m = n$ and $G_\mathbb{R} = \ ^0PQ(n)$, and $n \vert 0 \in \delta$ or $0 \vert m \in \delta$, then $H^0(D,\mathcal{O}) \cong \mathbb{C}^{nm}$
and these global odd functions are restrictions of the global odd functions on $Z$.
\item If $G_\mathbb{R} = SL_{n \vert m}(\mathbb{R})$ or $SL_{k \vert l}(\mathbb{H})$, then:
\begin{enumerate}
\item If either condition I, or condition II is satisfied, then $H^0(D,\mathcal{O}) \cong \bigwedge \mathbb{C}^{\frac{1}{2}nm}$.
\item If both condition I and condition II are satisfied, then $H^0(D, \mathcal{O}) \cong \bigwedge \mathbb{C}^{nm}$. 
\end{enumerate}
\item In all other cases, $H^0(D, \mathcal{O}) = \mathbb{C}$.
\end{enumerate} 

\end{satz}

\subsection{Types B, C and D}

In this section $G = \mathrm{Osp}(n \vert 2m)$. It has two possible classes of real forms, i.e. $G_\mathbb{R} = \mathrm{Osp}_\mathbb{R}(p,q \vert 2m)$ or $G_\mathbb{R} = \mathrm{Osp}^*(n \vert 2r,2s)$. The latter is only possible if $n$ is even. Here $D_{\bar{0}} = D_1 \times D_2$,
where $D_1$ is a $SO(p,q)$- or $SO^*(2k)$-flag domains, $D_2$ is a $\mathrm{Sp}_\mathbb{R}(2m)$- or $\mathrm{Sp}(2r,2s)$-flag domain, and at most one of the factors is hermitian holomorphic. Moreover the action of the antiholomorphic involution $\tau$ on the simple roots is trivial
unless $G_\mathbb{R} = \mathrm{Osp}_\mathbb{R}(2p+1,2q+1 \vert 2m)$. 

Furthermore $\mathfrak{g}_{\bar{1}}$ is an irreducible $G_{\bar{0}}$-module unless $n = 2$. In that case $\mathfrak{g}_{\bar{1}\mathbb{R}}$ will still be an irreducible $G_{\bar{0}\mathbb{R}}$-module unless $G_\mathbb{R} = \mathrm{Osp}_\mathbb{R}(1,1 \vert 2m)$.

\begin{satz}
 
Let $G = \mathrm{Osp}(n \vert 2m)$ with $n \neq 2$ and $D \subseteq Z(\delta) = G/P$ an open real supergroup orbit. Assume $H^0(D,\mathcal{F}) = \mathbb{C}$ and $\mathfrak{l}_1 \neq 0$. Then $H^0(D,\mathcal{O}) = \mathbb{C}$. 

\end{satz}

\begin{proof}
 In the given case $H^0(Z,\mathcal{O}) = \mathbb{C}$ as was proven in \cite{V}. Moreover ,as $n \neq 2$, $\mathfrak{g}_{\bar{1}\mathbb{R}}$
is always an irreucible $G_{\bar{0}\mathbb{R}}$-module and therefore nontrivial $G_{\bar{0}\mathbb{R}}$-modules in 
$(\mathfrak{g}_{\bar{1}\mathbb{R}}/\mathfrak{l}_{\bar{1}\mathbb{R}})^*$ can only exist if $\mathfrak{l} = 0$. 
\end{proof}

Now let $G = \mathrm{Osp}(2 \vert 2m)$. Then $Z(\delta) = G/P$ has global odd functions if and only if $\delta = 1 \vert 0 < 1 \vert d_1 < \ldots < 1 \vert d_k < 1 \vert m$. Then due to Theorem \ref{MaxOddOsp} the real supergroup orbits in $Z(\delta)$ with open base have maximal odd dimension  for all real forms except for $G_\mathbb{R} = \mathrm{Osp}_\mathbb{R}(1,1 \vert 2m)$.

Now let $G_\mathbb{R} = \mathrm{Osp}_\mathbb{R}(1,1\vert 2m)$. Then $D \subseteq Z(\delta)$ is open if and only if $1 \vert d \not\in \delta$ for all $d < m$(see Theorem \ref{D21}). Moreover $D$ is strongly measurable if and only if $1 \vert m \not\in \delta$
or $0 \vert m \in \delta$. The $G_{\bar{0}\mathbb{R}}$-module $\mathfrak{g}_{\bar{1}\mathbb{R}}$ decomposes into two irreducible components
$\mathfrak{s}_1 = \mathfrak{g}_{\bar{1}\mathbb{R}} \cap \mathfrak{g}_{-1}$ and $\mathfrak{s}_2 = \mathfrak{g}_{\bar{1}\mathbb{R}} \cap \mathfrak{g}_1$ according to Example \ref{TypeIcomp}. This yields the following result for the global odd functions:

\begin{satz}
 
Let $G_\mathbb{R} = \mathrm{Osp}_\mathbb{R}(1,1\vert 2m)$ and $D \subseteq Z(\delta)$ an open $G_\mathbb{R}$-orbit. 
Assume $H^0(D,\mathcal{F}) = \mathbb{C}$.

\begin{enumerate}
 \item If $1 \vert m \in \delta$ and $0 \vert m \not\in \delta$ then $H^0(D,\mathcal{O}) \cong \bigwedge \mathbb{C}^m$.
 \item If $0 \vert m \in \delta$ then $H^0(D,\mathcal{O}) \cong \bigwedge \mathbb{C}^{2m}$.
 \item Else $H^0(D,\mathcal{O}) = \mathbb{C}$.
\end{enumerate}
 
\end{satz}

\begin{proof}

Non-constant global odd functions on $D$ exist if $\mathfrak{l}_{\bar{1}\mathbb{R}} \cap \mathfrak{s}_i = 0$ for $i = 1$ or $2$. Now recall from example \ref{TypeIcomp} the block form of an element of $\mathfrak{g}$:

\[ M = \begin{pmatrix} a & 0 & u_1 & u_2 \\ 0 & -a & v_1 & v_2 \\ v_2^T & u_2^T & X & Y \\ -v_1^T & -v_2^T & Z & -X^T \end{pmatrix}, \begin{matrix} a \in \mathbb{C} \\ u_i, v_i \in \mathbb{C}^m \\ X,Y,Z \in \mathbb{C}^{m \times m} \\ Y = Y^T, Z = Z^T \end{matrix} \]

Suppose $M \in \mathfrak{p}$. If $1 \vert m \in \delta$, then $v_1 = 0$ and therefore $(\mathfrak{p} \cap \tau \mathfrak{p}) \cap \mathfrak{g}_{-1} = 0$ which is equivalent to $\mathfrak{l}_{\bar{1}\mathbb{R}} \cap \mathfrak{s}_1 = 0$. Conversely, if $1 \vert m \not\in\delta$, then there will be an element $M_0$ of $\mathfrak{p}$ such that $v_1 \neq 0$
and therefore $\mathfrak{l}_{\bar{1}\mathbb{R}} \cap \mathfrak{s}_1 \neq 0$. 

Suppose now that $0 \vert m \in \delta$ and $M$ is again an arbitrary element of $\mathfrak{p}$.
Then $u_1 = v_1 = 0$ and therefore $(\mathfrak{p} \cap \tau\mathfrak{p}) \cap (\mathfrak{s}_1 + \mathfrak{s}_2) = 0$. Conversely, if $0 \vert m \not\in \delta$, then there is some $M_0$ in $\mathfrak{p}$ such that $u_1 \neq 0$ and there $\mathfrak{l}_{\bar{1}\mathbb{R}} \cap \mathfrak{s}_2 \neq 0$.

\end{proof}

Putting these results together the global holomorphic superfunctions on cycle-connected flag domains of types $B,C$ and $D$ are classified as follows:

\begin{satz}

Let $G = \mathrm{Osp}(n \vert 2m), G_\mathbb{R}$ a real form, $Z = Z(\delta) = G/P$ a flag supermanifold and $D \subseteq Z$ a $G_\mathbb{R}$-flag domain with $H^0(D_{\bar{0}}, \mathcal{F}) = \mathbb{C}$.

\begin{enumerate}
\item If $n \neq 2$, then in all cases $H^0(D, \mathcal{O}) = \mathbb{C}$.
\item If $n = 2$, $G_\mathbb{R} \neq \mathrm{Osp}(1,1 \vert 2m)$ and $1 \vert 0 \in \delta$,
then $H^0(D,\mathcal{O}) \cong \bigwedge \mathbb{C}^{2m}$ and these global odd functions are restrictions of the global odd functions on $Z$.
\item If $n = 2$, $G_\mathbb{R} = \mathrm{Osp}(1,1 \vert 2m)$ and $1 \vert m \in \delta, 0 \vert m \not\in \delta$, then $H^0(D,\mathcal{O}) \cong \bigwedge \mathbb{C}^m$.
\item If $n = 2$, $G_\mathbb{R} = \mathrm{Osp}(1,1 \vert 2m)$ and $0 \vert m \in \delta$, then $H^0(D,\mathcal{O}) \cong \bigwedge \mathbb{C}^{2m}$.
\item In all other cases for $n = 2$, $H^0(D, \mathcal{O}) = \mathbb{C}$.
\end{enumerate} 

\end{satz}

\section{Holomorphic superfunctions on flag domains: The non-cycle-connected Case}\label{HermCase}

\noindent In this section $D$ is assumed to have global holomorphic functions, \\ i.e. $H^0(D_{\bar{0}},\mathcal{F}) \neq \mathbb{C}$.
If $D$ is connected this implies that the bounded symmetric domain subordinate to $D_{\bar{0}}$ is not reduced to a point.
If furthermore $\mathfrak{g} = \mathrm{Lie}(G)$ is simple and basic classical then $D_{\bar{0}} = D_1 \times D_2$ where one of the $D_i$ is a classical flag domain of hermitian holomorphic type. There are the following possible cases:

\begin{enumerate}
 \item $G = SL_{n \vert m}(\mathbb{C})$, $G_\mathbb{R} = SU(p,q \vert r,s)$, $D_1$ hermitian holomorphic 
       and $D_2$ compact or hermitian holomorphic.
 \item $G,G_\mathbb{R}$ and $D_1$ as in the first case, but $D_2$ cycle-connected.
 \item $G = \mathrm{Osp}(n \vert 2m)$, $G_\mathbb{R}$ is any real form and either $D_1$ is a hermitian holomorphic 
       $SO^*(n)$-orbit or $D_2$ is a hermitian holomorphic $Sp_{2m}(\mathbb{R})$-orbit.
 \item $G = \mathrm{Osp}(n + 2 \vert 2m), G_\mathbb{R} = \mathrm{Osp}_\mathbb{R}(n,2 \vert 2m)$ and 
       $D_1$ and $D_2$ hermitian holomorphic.  
 \item $G$ and $G_\mathbb{R}$ as in the fourth case, but only one $D_i$ hermitian holomorphic.
\end{enumerate}

In the first and fourth case $D$ projects onto a hermitian symmetric superspace. This allows
the following characterization of global odd functions:

\begin{satz}

Let $G = SL_{n \vert m}(\mathbb{C})$, $G_\mathbb{R} = SU(p,q \vert r,s)$, $D_1$ hermitian holomorphic 
and $D_2$ compact or hermitian holomorphic or $G = \mathrm{Osp}(n + 2 \vert 2m), G_\mathbb{R} = \mathrm{Osp}_\mathbb{R}(n,2 \vert 2m)$ and $D_1$ and $D_2$ hermitian holomorphic. Then there exists a fibration $\varphi: D \rightarrow \mathcal{B} = G_\mathbb{R}/K_\mathbb{R}$ with compact fibre $F = K_\mathbb{R}/L_\mathbb{R}$ and Stein base $\mathcal{B}$ and the Leray spectral sequence yields $H^0(D, \mathcal{O}_D) \cong H^0(\mathcal{B},\mathcal{O}_\mathcal{B}) \otimes H^0(F, \mathcal{O}_F)$.

\end{satz}

\noindent In the other three cases $D$ does not necessarily project onto a $G_\mathbb{R}$-homogeneous superspace,
but the holomorphic reduction of $\mathrm{gr} D$ is always \\ a $(\mathrm{gr} G_\mathbb{R})$-homogeneous space. Its base
is the bounded symmetric domain subordinate to the hermitian holomorphic component of $D_{\bar{0}}$, in particular it is 
a Stein space.

First let $G = SL_{n \vert m}(\mathbb{C}), G_\mathbb{R} = SU(p,q \vert r,s)$ and $ D \subseteq Z = G/P$ open
such that $D_1$ is hermitian holomorphic and $D_2$ is cycle-connected. Then the holomorphic reduction of
$D_{\bar{0}} = D_1 \times D_2$ is $\mathcal{B} = S(U(p,q) \times U(r,s))/S(U(p) \times U(q) \times U(r,s)) = 
G_{\bar{0}\mathbb{R}}/(K_{1,\mathbb{R}} \times G_{2,\mathbb{R}})$ and $\mathfrak{g}_{\bar{1}\mathbb{R}}$ decomposes
into two irreducible $(K_{1,\mathbb{R}} \times G_{2,\mathbb{R}})$-modules $\mathfrak{s}_-,\mathfrak{s}_+$ which are isomorphic to
$\mathrm{Hom}(\mathbb{C}^p, \Pi\mathbb{C}^m)$ and $\mathrm{Hom}(\mathbb{C}^q, \Pi\mathbb{C}^m)$ respectively.

\begin{lemma}
Let $\mathrm{gr}G_\mathbb{R}/J$ be the holomorphic reduction of $\mathrm{gr} D$.
One has $\mathfrak{l}_{\bar{1}\mathbb{R}} \cap \mathfrak{s}_-= 0$ if and only if $J \subseteq \mathrm{gr}J^- = \mathrm{gr}S(U(p) \times U(0,q \vert r,s))$
and $\mathfrak{l}_{\bar{1}\mathbb{R}} \cap \mathfrak{s}_+ = 0$ if and only if $J \subseteq \mathrm{gr}J^+ = \mathrm{gr}S(U(q) \times U(p,0 \vert r,s))$.
In particular, in both cases $D$ projects onto an actual $G_\mathbb{R}$-homogeneous superspace. 
Moreover the global holomorphic superfunctions are given as follows:

\begin{enumerate}
 \item If $D$ maps onto $G_\mathbb{R}/J^-$, then $H^0(D, \mathcal{O}) \cong H^0(D_1,\mathcal{F}) \otimes \bigwedge \mathfrak{s}_-$ \\ $ \cong H^0(D_1,\mathcal{F}) \otimes \bigwedge \mathbb{C}^{pm}$.
 \item If $D$ maps onto $G_\mathbb{R}/J^+$, then $H^0(D, \mathcal{O}) \cong H^0(D_1,\mathcal{F}) \otimes \bigwedge \mathfrak{s}_+$ \\ $ \cong H^0(D_1,\mathcal{F}) \otimes \bigwedge \mathbb{C}^{qm}$.
 \item If $D$ maps onto $G_\mathbb{R}/(J^- \cap J^+)$, \\ then $H^0(D, \mathcal{O}) \cong H^0(D_1,\mathcal{F}) \otimes \bigwedge (\mathfrak{s}_+ \oplus \mathfrak{s}_-) \cong H^0(D_1,\mathcal{F}) \otimes \bigwedge \mathbb{C}^{mn}$. 
 \item If $D$ does not map onto either of $G_\mathbb{R}/J^+$ or $G_\mathbb{R}/J^-$, then $H^0(D,\mathcal{O}) \cong H^0(D_1, \mathcal{F})$.
\end{enumerate}

\end{lemma}

\begin{proof}
 The first statement follows as $\mathfrak{s}_- = T_{e,\bar{1}}J^+$ and $\mathfrak{s}_+ = T_{e,\bar{1}}J^-$. Let $\varepsilon \in \{+,-\}$.
 
 Suppose now that $D$ projects onto $G_\mathbb{R}/J^\varepsilon$.
 As the fibre is a cycle-connected $J^\varepsilon$-orbit with no non-constant odd functions, 
 $H^0(D,\mathcal{O}) \cong H^0(G_\mathbb{R}/J^\varepsilon,\mathcal{O}) = H^0(D_1,\mathcal{F}) \otimes \bigwedge \mathfrak{s}_\varepsilon$.
 If $D$ projects onto $G_\mathbb{R}/(J^+ \cap J^-)$, then $\mathfrak{p} \cap \tau\mathfrak{p} = 0$ and therefore the global odd functions are given by $\bigwedge \mathfrak{g}_{\bar{1}\mathbb{R}}^* \cong \bigwedge \mathbb{C}^{mn}$.

 If $D$ does not map onto either $G_\mathbb{R}/J^\varepsilon$, then $\mathfrak{l} \cap \mathfrak{s}_\varepsilon \neq 0$ for $\varepsilon = \pm$.
 Therefore $\mathfrak{j}_1 = \mathfrak{g}_{1\mathbb{R}}$ and the holomorphic reduction of $\mathrm{gr}D$ is $(D_1, \mathcal{F})$.
 This then implies $H^0(D,\mathcal{O}) \cong H^0(D_1,\mathcal{F})$.      
\end{proof}

One obtains analogous results for the cases involving the other nonexceptional hermitian symmetric spaces
by noting that the respective real forms are the intersections $U(n,n) \cap SO(2n)$, $U(m,m) \cap Sp(2m)$ and 
$U(n,2) \cap SO(n+2)$.

This leads to the following summarizing theorem for the non-cycle-connected case:

\begin{satz}

Let $G$ be a complex simple basic classical Lie supergroup and $G_\mathbb{R}$ a real form with even part $G_{\bar{0}\mathbb{R}} = G_{1,\mathbb{R}} \times G_{2,\mathbb{R}}$ such that $G_{1,\mathbb{R}}$ is a hermitian real form. Let $Z = G/P$ be a flag supermanifold and $D \subseteq Z$ a $G_\mathbb{R}$-flag domain with $D_{\bar{0}} = D_1 \times D_2$ such that $D_1$ is hermitian holomorphic. Furthermore let $\mathfrak{s}_+$ and $\mathfrak{s}_-$ be the two irreducible $K_{1,\mathbb{R}} \times G_{2,\mathbb{R}}$-submodules of $\mathfrak{g}_{\bar{1}\mathbb{R}}$ and let $J^\pm$ be the Lie supergroups corresponding to the SHCP $(K_{1,\mathbb{R}} \times G_{2,\mathbb{R}}, \mathfrak{k}_{1,\mathbb{R}} \oplus \mathfrak{g}_{2,\mathbb{R}} \oplus \mathfrak{s}_\pm)$.

\begin{enumerate}
\item If $D_2$ is compact or hermitian holomorphic, then $D$ projects onto a hermitian symmetric superspace $G_\mathbb{R}/K_\mathbb{R}$ and this projection induces an isomorphism on superfunctions.
\item If $D_2$ is cycle-connected and $\mathfrak{l}_{\bar{1}\mathbb{R}}$ has zero intersection with either $\mathfrak{s}_+$, or $\mathfrak{s}_-$, then $D$ projects onto $G_\mathbb{R}/J^+$
or $G_\mathbb{R}/J^-$ and this induces an isomorphism on superfunctions.
\item If $D_2$ is cycle-connected and $\mathfrak{l}_{\bar{1}\mathbb{R}}$ has zero intersection with both $\mathfrak{s}_+$, and $\mathfrak{s}_-$, then $D$ projects onto $G_\mathbb{R}/(J^+ \cap J^-)$ and this induces an isomorphism on superfunctions.
\item If $D_2$ is cycle-connected and  $\mathfrak{l}_{\bar{1}\mathbb{R}}$ has non-zero intersection with both $\mathfrak{s}_+$, and $\mathfrak{s}_-$, then $H^0(D,\mathcal{O}) \cong H^0(D_1,\mathcal{F})$.
\end{enumerate}

\end{satz}

\section{The Flag domains of types P and Q}

If $G$ is a simple Lie supergroup of type $P$ or $Q$ then $G_{\bar{0}} = SL_n(\mathbb{C})$
and $G_{\bar{0}\mathbb{R}}$ is one of its real forms. Note however that if $G = P(n)$ then there
are no real forms with $G_{0\mathbb{R}} = SU(p,q)$ for any choice of $p$ and $q$.

First suppose $G$ is of type $Q$. In that case $\mathfrak{p} = \mathfrak{p}_{\bar{0}} \oplus \Pi \mathfrak{p}_{\bar{0}}$
which implies the following:

\begin{satz}

Let $G = Q(n)$, $G_\mathbb{R}$ one of its real forms and $D \subseteq Z = G/P$ an open orbit.
If the bounded symmetric domain subordinate to $D_{\bar{0}}$ is $\mathcal{B} = G_{\bar{0}\mathbb{R}}/K_{\bar{0}\mathbb{R}}$
then $H^0(D,\mathcal{O}) \cong H^0(\mathcal{B},\mathcal{F}) \otimes \bigwedge \Pi(\mathfrak{g}_{\bar{0}\mathbb{R}}/\mathfrak{k}_{\bar{0}\mathbb{R}})^*$.
Otherwise $H^0(D,\mathcal{O}) = \mathbb{C}$.

\end{satz}

Now let $G= P(n), n \geq 4$. For odd $n$ its only real form up to isomorphism is $P_\mathbb{R}(n)$, for even $n$ there is a second real form $P_\mathbb{H}(k)$. To avoid confusion the parabolic subalgebra of $G$ will be denoted $Q$ in this case. The action of $\tau$ on the root system is given in both cases by $\tau(\pm x_i \pm x_j) = \pm x_{n-i+1} \pm x_{n-j+1}$. As the root system of $\mathfrak{g}$ contains roots whose negatives are not roots, a necessary condition for a real orbit to have maximal odd dimension is $\mathfrak{g}^{2x_i} \in \mathfrak{q}$ for all $1 \leq i \leq n$. This requires the dimension sequence $\delta$ to be of the form $\delta = 0 \vert 0 < d_1 \vert d_1 < \ldots < d_k \vert d_k < n \vert n$.
As the underlying open orbits are open $SL_n(\mathbb{R})$-orbits or $SL_k(\mathbb{H})$-orbits, which are all cycle-connected, the global odd functions
can be characterized as follows:

\begin{satz}
 
Let $G = P(n), n \geq 4$, $G_\mathbb{R}$ one of its real forms and $D \subseteq Z(\delta) = G/Q$ a $G_\mathbb{R}$-flag domain. Then  $H^0(D, \mathcal{O}) = \mathbb{C}$.

\end{satz}

\begin{proof}

Let $\mathfrak{g}_1$ and $\mathfrak{g}_{-1}$ be the two irreducible components of $\mathfrak{g}_{\bar{1}}$ according to Example \ref{TypeIcomp}. Then $\mathfrak{g}_1$ contains
the root spaces $\mathfrak{g}^{x_i + x_{n-i+1}}, 1 \leq i \leq n$ and $\mathfrak{g}_{-1}$ contains the root spaces $\mathfrak{g}^{-x_i - x_{n-i+1}}, 1 \leq i \leq n$. All these roots are fixed points of $\tau$, therefore for a $G_\mathbb{R}$-orbit $D$ with maximal odd dimension
the intersection $\mathfrak{l}_{\bar{1}\mathbb{R}} \cap \mathfrak{g}_{\pm 1}$ will always be non-zero. Thus there are no non-trivial $g_\mathbb{R}$-submodules in $(\mathfrak{g}_{\bar{1}\mathbb{R}}/\mathfrak{l}_{\bar{1}\mathbb{R}})^*$.

\end{proof}
  
\section{Extension of the theory to even-homogeneous flag domains}

In addition to the flag domains considered so far the flag domains of even real forms also deserve attention. 
In this setting $G_\mathbb{R}$ is a purely even group, but the subalgebra $\mathfrak{l} = \mathfrak{p} \cap \tau\mathfrak{p}$
can still be defined. The real stabiliser $L_\mathbb{R}$ is a real form of $L_{\bar{0}}$.

Let $D$ be a flag domain of even-homogeneous type. Then $\mathrm{gr} D = (D_{\bar{0}}, \bigwedge \mathcal{E})$, where
$\mathbb{E} = G_\mathbb{R} \times_{L_\mathbb{R}} (\mathfrak{g}_{\bar{1}}/\mathfrak{p}_{\bar{1}})^*$. This identification suggests
that the holomorphic reduction is once again a useful tool to compute the global holomorphic functions on $D$. 

The classification of global holomorphic functions is very similar to the case considered before. The main differences are
the following:

\begin{itemize}
 \item Even-real forms only exist for the basic classical Lie superalgebras, not for Lie superalgebras of type $P$ and $Q$
 \item The role of the maximal $G_{\bar{0}\mathbb{R}}$-submodule of $\mathfrak{g}_{\bar{1}\mathbb{R}}$ is now assumed by the maximal
       $\tau$-invariant $G$-submodule of $\mathfrak{g}_{\bar{1}}$.  
\end{itemize}

\noindent The actions of $\tau$ on the root system coincide with those considered before according to the following table:\newline

 \begin{tabular}{ p{4cm} | p{6cm} }
 
 real form & even-real form  \\ \hline
  $SU(p,n-p \vert q,m-q)$ & $SU(p,n-p) \times SU(q,m-q) \times U(1)$ \\ \hline
  $SL_{n \vert m}(\mathbb{R}),SL_{k \vert l}(\mathbb{H})$ & $SL_{n}(\mathbb{R}) \times SL_l(\mathbb{H}) \times \mathbb{R}^{>0}$ \\ \hline
  $ ^0PQ(n)$ & $SL_{n}(\mathbb{C})$ \\ \hline
  $US\Pi(n)$  & none \\ \hline
  $\mathrm{Osp}(2p+1,2q+1 \vert 2m)$  & $SO(2p+1,2q+1) \times Sp(2m)$ \\ \hline

\end{tabular}\newline

\noindent For all even-real forms not listed above the action of $\tau$ on the root system is trivial. 

\begin{satz}
 
Let $G$ be a basic classical Lie supergroup, $Z = G/P$ a $G$-flag manifold, $G_\mathbb{R}$ an even-real form and $\widetilde{G_\mathbb{R}}$ be the associated
real form according to the above table. Let $D$ and $\tilde{D}$ be open orbits of $G_\mathbb{R}$ and $\widetilde{G_\mathbb{R}}$
respectively with the same base point in $Z$. If $H^0(D, \mathcal{F}) = H^0(\tilde{D}, \mathcal{F}) = \mathbb{C}$,
then $H^0(D, \mathcal{O}) = H^0(\tilde{D}, \mathcal{O})$. 

\end{satz}

\begin{proof}

First note that in a suitable basis the defining involutions $\tau$ and $\tilde{\tau}$ only differ by a linear operator
taking different scalar values on components of $\mathfrak{g}_{\bar{1}}$. Consequently $\tau$-invariant $G$-submodules
and $\tilde{\tau}$-invariant $G$-submodules coincide.

Now the invariant $G_\mathbb{R}$-submodules of $(\mathfrak{g}_{\bar{1}}/\mathfrak{p}_{\bar{1}})^*$ are in one-to-one correspondence with
the $\widetilde{G_{\bar{0}\mathbb{R}}}$-submodules of $(\widetilde{\mathfrak{g}_{\bar{1}\mathbb{R}}}/\mathfrak{l}_{\bar{1}\mathbb{R}})^*$.  As the global holomorphic functions
are given by these modules, the respective $H^0$-groups must be isomorphic.

\end{proof}

Now assume $H^0(D,\mathcal{F}) \neq \mathbb{C}$. Then $D_{\bar{0}} = D_1 \times D_2$ where at least one $D_i$ is a flag domain of hermitian holomorphic type. 

First assume that both $D_1$ and $D_2$ are of hermitian holomorphic type. In that case $Z$ maps onto $\hat{Z} = G/KS$,
where $K \subseteq G$ is a complex subsupergroup defined by an involution $\theta: \mathfrak{g} \rightarrow \mathfrak{g}$
which restricts to the Cartan involution on $\mathfrak{g}_\mathbb{R}$, and $S$ is a nilpotent subsupergroup of $G$.
The image $\hat{D}$ of $D$ in $\hat{Z}$ has as its base a hermitian symmetric domain. Moreover the fibration $D \rightarrow \hat{D}$ is a fibre bundle with compact fibre $F$. This leads to the following result:

\begin{satz}

Let $D$ be a flag domain of an even real form with $D_{\bar{0}}$ hermitian holomorphic, $\hat{D} \subseteq G/KS$ such that $\hat{D}_{\bar{0}}$ is the bounded symmetric domain subordinate to $D_{\bar{0}}$ and $F$ the fibre of the canonical projection. Then

\[ H^0(D, \mathcal{O}) \cong H^0(\hat{D},\mathcal{O}) \otimes H^0(F, \mathcal{O}) = H^0(\hat{D}_{\bar{0}}, \mathcal{F}) \otimes \bigwedge \overline{\mathfrak{s}}_{\bar{1}} \otimes H^0(F, \mathcal{O}) \] 

\noindent where $\bar{\mathfrak{s}}$ is the complementary $K_{\bar{0}}$-module to $\mathfrak{k} + \mathfrak{s}$ in $\mathfrak{g}$.

\end{satz}

Now assume that precisely one $D_i$, say $D_1$, is of hermitian holomorphic type. Then the holomorphic reduction of $D_{\bar{0}}$ is $\hat{D_1} = (G_{\mathbb{R},1} \times G_{\mathbb{R},2})/(K_{\mathbb{R},1} \times G_{\mathbb{R},2})$, so $J_\mathbb{R} = K_{\mathbb{R},1} \times G_{\mathbb{R},2}$.
The global holomorphic functions in that case are characterized as follows:

\begin{satz}
 
Let $D \subseteq Z$ open and even-homogeneous and $D_{\bar{0}} = D_1 \times D_2$ with $D_1$ hermitian holomorphic and $D_2$ cycle-connected.
Further let $J_\mathbb{R} = K_{\mathbb{R},1} \times G_{\mathbb{R},2} \subseteq G_\mathbb{R}$ so $G_\mathbb{R}/J_\mathbb{R}$ is the holomorphic reduction of $D_{\bar{0}}$. 
Let $J_{\bar{0}} = P_1 \times G_2$. Then there are two possibilities:

\begin{enumerate}
 \item $(\mathfrak{g}_{\bar{1}}/\mathfrak{p}_{\bar{1}})^*$ contains a maximal non-trivial $J_{\bar{0}}$-module $W$. Let $Y = (\mathfrak{g}_{\bar{1}}^*/W)^*$.
Then $(J_{\bar{0}}, \mathfrak{j}_{\bar{0}} + Y)$ is a SHC subpair of $(G_{\bar{0}},\mathfrak{g})$, $D$ maps onto an open submanifold $D^\prime \subseteq Z^\prime = G/J$ 
 and $H^0(D,\mathcal{O}) \cong H^0(D^\prime,\mathcal{O}) = H^0(D_1, \mathcal{F}) \otimes \bigwedge W$.
 \item $(\mathfrak{g}_{\bar{1}}/\mathfrak{p}_{\bar{1}})^*$ does not contain non-trivial $J_{\bar{0}}$-modules.\newline Then $H^0(D,\mathcal{O}) \cong H^0(D_1,\mathcal{F})$. 
\end{enumerate}

\end{satz}

\begin{proof}

In the first case consider the fibration $D \rightarrow D^\prime$. Its fibre $F$ is a cycle-connected flag domain in
a $J$-flag manifold with no non-constant global odd functions. Therefore $H^0(D, \mathcal{O}) \cong H^0(D^\prime, \mathcal{O}) \otimes H^0(F,\mathcal{O}) =
H^0(D^\prime, \mathcal{O}) \otimes \mathbb{C} = H^0(D^\prime, \mathcal{O}) = H^0(D_1,\mathcal{F}) \otimes \bigwedge W$.

In the second case note that $H^0(\mathrm{gr} D, \mathcal{O}) = H^0(D_1,\mathcal{F})$ so the same is true for $H^0(D, \mathcal{O})$. 

\end{proof}

The maximal $J_{\bar{0}}$-submodules of $(\mathfrak{g}_{\bar{1}}/\mathfrak{p}_{\bar{1}})^*$
have already been determined in the computations in Section \ref{HermCase}. The results again correspond according to the above table.

\newpage
\section{Tables of the results}

For all the tables it is assumed that $H^0(Z,\mathcal{O}) = \mathbb{C}$. 
The table for the pure cycle-connected case:

\noindent \begin{tabular}{p{1.6cm} | p{4.2cm} | p{3cm} | p{1.5cm}}
 
$G$ & $G_\mathbb{R}$ & condition & $H^0(D, \mathcal{O})$  \\ \hline
$SL_{n \vert m}(\mathbb{C})$ & $SU(p,q \vert r,s), \newline SU(p,q) \times SU(r,s) \times U(1)$ & $\dim \mathfrak{l}_1 = 0$ & $ \bigwedge \mathbb{C}^{nm}$ \\ \cline{3-4}
 &  & $\dim \mathfrak{l}_1 \neq 0$ & $\mathbb{C}$ \\ \cline{2-4}
 & $SL_{n \vert m}(\mathbb{R}), SL_{k \vert l}(\mathbb{H}),\newline SL_n(\mathbb{R}) \times SL_l(\mathbb{H}) \times \mathbb{R}^{>0}$ & cond. I & $\bigwedge \mathbb{C}^{km}$ \\ \cline{3-4}
 &  & cond. II & $\bigwedge\mathbb{C}^{nl}$ \\ \cline{3-4}
 &  & cond. I + cond. II & $\bigwedge\mathbb{C}^{nm}$ \\ \cline{3-4}
 &  & otherwise & $\mathbb{C}$ \\ \hline
$PSL_{n \vert n}(\mathbb{C})$ & $^0PQ(n), SL_n(\mathbb{C})$ & $\dim \mathfrak{l}_1 = 0$ & $\bigwedge\mathbb{C}^{nm}$ \\ \cline{3-4}
 &  & $\dim \mathfrak{l}_1 \neq 0$ & $\mathbb{C}$ \\ \cline{2-4}
 & $US\Pi(n)$ & always & $\mathbb{C}$ \\ \hline
$\mathrm{Osp}(n \vert 2m) \newline(n > 2)$ & any & $\dim \mathfrak{l}_1 = 0$ & $\bigwedge\mathbb{C}^{nm}$ \\ \cline{3-4}
 &  &  $\dim \mathfrak{l}_1 \neq 0$ & $\mathbb{C}$ \\ \hline
$\mathrm{Osp}(2 \vert 2m)$ & $\mathrm{Osp}(1,1 \vert 2m)$ & $0 \vert m, 1 \vert m \in \delta$ & $\bigwedge\mathbb{C}^{2m}$ \\ \cline{3-4}
 &  & $0 \vert m \not\in \delta, 1 \vert m \in \delta$ & $\bigwedge\mathbb{C}^m$ \\ \cline{3-4}
 &  & otherwise & $\mathbb{C}$ \\ \cline{2-4}
 & any other & $1 \vert 0 \in \delta$ & $\bigwedge\mathbb{C}^{2m}$ \\ \cline{3-4}
 &  & $1 \vert 0 \not\in \delta$ & $\mathbb{C}$ \\ \hline
$P(n)$ & any & always & $\mathbb{C}$ \\ \hline
$Q(n)$ & any & always & $\mathbb{C}$ 

\end{tabular} \medskip

Next is the table for the pure hermitian holomorphic case. Here it is assumed that $D_{\bar{0}} = D_1 \times D_2$ is a product of two classical hermitian holomorphic flag domains, unless $G = PSQ(n)$. In that case $D_{\bar{0}} = D_1$. Furthermore it is assumed that $\dim (\mathfrak{p} \cap \tau\mathfrak{p})_{\bar{1}} > 0$. 

\medskip

\noindent \begin{tabular}{p{1.6cm} | p{4.2cm} | p{4cm}}
 
$G$ & $G_\mathbb{R}$  & $H^0(D, \mathcal{O})$  \\ \hline
$SL_{n \vert m}(\mathbb{C})$ & $SU(p,q \vert r,s) \newline SU(p,q) \times SU(r,s) \times U(1)$ & $H^0(D_{\bar{0}}, \mathcal{F}) \otimes \bigwedge \mathbb{C}^{qr+ps}$ \\ \hline  
$\mathrm{OSp}(n \vert 2m)$ & $\mathrm{Osp}(2, n-2 \vert 2m) $ & $H^0(D_{\bar{0}}, \mathcal{F}) \otimes \bigwedge \mathbb{C}^{2m}$ \\ \hline
$\mathrm{Osp}(2n \vert 2m)$ & $SO^*(2n) \times \mathrm{Sp}_\mathbb{R}(2m)$ & $H^0(D_{\bar{0}}, \mathcal{F}) \otimes \bigwedge \mathbb{C}^{nm}$ \\ \hline
$PSQ(n)$ & $UPSQ(p,q)$ & $ H^0(D_{\bar{0}}, \mathcal{F}) \otimes \bigwedge \mathbb{C}^{pq}$ 

\end{tabular} \medskip

In the mixed case it is assumed that $D_{\bar{0}} = D_1 \times D_2$ where $D_1$ is hermitian holomorphic and $D_2$ is cycle connected and that $D$ is the $G_\mathbb{R}$-orbit through the neutral point of $Z$. Moreover it is again assumed that $\dim (\mathfrak{p} \cap \tau \mathfrak{p})_{\bar{1}} > 0$. Also, $\mathrm{SpO}(2m \vert n)$ is a group isomorphic to $\mathrm{Osp}(n \vert 2m)$, but with the orthogonal and symplectic factors interchanged. Consequently, for a flag domain of an (even) real form of $\mathrm{Osp}(n \vert 2m)$, the $SO$-factor of the base is hermitian holomorphic and in the case of $\mathrm{SpO}(2m \vert n)$, the $Sp$-factor of the base is hermitian holomorphic. 

\medskip

\noindent \begin{tabular}{p{1.7cm} | p{4.1cm} | p{2.7cm} | p{3cm}}
 
$G$ & $G_\mathbb{R}$ & condition & $H^0(D, \mathcal{O})$  \\ \hline

$SL_{n \vert m}(\mathbb{C})$ & $SU(p,q \vert r,s), \newline SU(p,q) \times SU(r,s) \times U(1)$ & $p \vert 0 \in \delta$ & $H^0(D_1, \mathcal{F}) \otimes \bigwedge \mathbb{C}^{pm}$ \\ \cline{3-4}
& & $p \vert m \in \delta$ & $H^0(D_1, \mathcal{F}) \otimes \bigwedge \mathbb{C}^{qm}$ \\ \cline{3-4}
& & $p \vert 0$ and $p \vert m \in \delta$ & $H^0(D_1, \mathcal{F}) \otimes \bigwedge \mathbb{C}^{nm}$ \\ \cline{3-4}
& & otherwise & $H^0(D_1, \mathcal{F})$ \\ \hline
$\mathrm{Osp}(n \vert 2m)$ & $\mathrm{Osp}(2, n-2 \vert 2m)$ & $2 \vert 0 \in \delta$ & $H^0(D_1, \mathcal{F}) \otimes \bigwedge \mathbb{C}^{2m}$ \\ \cline{3-4}
& & otherwise & $H^0(D_1, \mathcal{F})$ \\ \hline
$\mathrm{Osp}(2n \vert 2m)$ & $\mathrm{Osp}^*(2n \vert 2r, 2s) \newline SO^*(2n) \times \mathrm{Sp}_\mathbb{R}(2m)$ & $n \vert 0 \in \delta$ & $H^0(D_1, \mathcal{F}) \otimes \bigwedge \mathbb{C}^{nm}$ \\ \cline{3-4}
& & otherwise & $H^0(D_1, \mathcal{F}) $ \\ \hline
$\mathrm{SpO}(2m \vert n)$ & $\mathrm{SpO}(2m \vert p,q)$ & $n$ odd & $H^0(D_1,\mathcal{F})$ \\ \hline
$\mathrm{SpO}(2m \vert 2n)$ & $\mathrm{SpO}(2m \vert p,q) \newline \mathrm{Sp}_\mathbb{R}(2m) \times SO^*(2n)$ & $m \vert 0 \in \delta$ & $H^0(D_1, \mathcal{F}) \otimes \bigwedge \mathbb{C}^{km}$ \\ \cline{3-4}
& & otherwise & $H^0(D_1, \mathcal{F}) $  \\

\end{tabular} 

\chapter{Cycle spaces and the Double Fibration Transform}

In this chapter the classical notion of the cycle space is generalized to the supersymmetric case and a universal definition covering flag domains of both real forms and even real forms is given. This definition of cycle spaces in the supersymmetric context is not at all trivial, due to the fact that compact real forms and commuting involutions do not exist in abundance as in the classical case. 

Moreover the Double Fibration Transform relating the cohomology of vector bundles on the flag domains to sections of certain associated bundles on the cycle space is constructed in the supersymmetric setting. 
Then the questions of injectivity and image of the Double Fibration Transform are discussed. 

In the cases where the cycles are purely even the classical Bott-Borel-Weil Theorem is used
to obtain injectivity conditions for the Double Fibration Transform and to describe its target space. In order to obtain analogous results in the cases where the cycles have non-zero odd dimension, the current state of Bott-Borel-Weil theory for Lie superalgebras is reviewed. Moreover several results from that theory are used to demonstrate a general technique, which translates BBW type theorems into sufficient conditions for injectivity and a concrete characterization of its target.

\section{The classical setting}


Let $Z_{\bar{0}} = G_{\bar{0}}/P_{\bar{0}}$ be a flag manifold and $G_{\bar{0}\mathbb{R}}/L_{\bar{0}\mathbb{R}} = D_{\bar{0}} \subseteq Z_{\bar{0}}$ a flag domain.
Further let $K_{\bar{0}\mathbb{R}}$ be the maximal compact subgroup subgroup of $G_{\bar{0}\mathbb{R}}$ determined by the 
Cartan involution $\theta$ and $K_{\bar{0}} = K_{\bar{0}\mathbb{R}}^\mathbb{C}$ its complexification. Then $D_{\bar{0}}$ contains a unique
closed $K_{\bar{0}}$-orbit $C_0$, the base cycle. As it is closed $J_{\bar{0}} = \mathrm{Stab}_{G_{\bar{0}}}(C_0)$ is a closed subgroup of $G_{\bar{0}}$.
Let $(\mathcal{M}_Z)_{\bar{0}} = G_{\bar{0}}/J_{\bar{0}}$.

Then $\mathcal{M}_{\bar{0}} = \{ Y \in (\mathcal{M}_Z)_{\bar{0}} : Y \subseteq D_{\bar{0}} \}^\circ$  is the group-theoretic cycle space of $D_{\bar{0}}$.

It fits into a double fibration

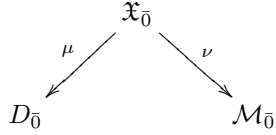
\[ \xymatrix{ & \mathfrak{X}_{\bar{0}} \ar[dl]_\mu \ar[dr]^\nu & \\ D_{\bar{0}} & & \mathcal{M}_{\bar{0}} } \]

Given a holomorphic vector bundle $\mathbb{E}$ on $D_{\bar{0}}$ there is a Double Fibration Transform 
relating the cohomology of the sheaf $\mathcal{O}(\mathbb{E})$ with sections of a certain associated
vector bundle $\mathcal{O}(\mathbb{E}^\prime)$ on $\mathcal{M}_{\bar{0}}$. The aim of this procedure is to obtain information about the possibly unknown cohomology groups $H^i(D_{\bar{0}}, \mathcal{O}(\mathbb{E}))$ using the transform and the cohomology groups
$H^0(\mathcal{M}_{\bar{0}}, \mathcal{O}(\mathbb{E})^\prime)$. For this to work the following conditions need to be fulfilled:

\begin{itemize}
 \item The Double Fibration Transforms needs to be injective.
 \item There has to be some characterization of its image inside $H^0(\mathcal{M}_{\bar{0}}, \mathcal{O}(\mathbb{E}^\prime))$.
 \item The fibres of $\mathbb{E}^\prime$ need to be known.
\end{itemize}

It turns out that classically the fibres of $\mathbb{E}^\prime$ are actually the cohomology groups $H^i(C_0, \mathcal{O}(\mathbb{E} \vert_{C_0}))$, so they are given by Bott-Borel-Weyl theory. 
The very same BBW theory also yields conditions for injectivity and the concrete description of the image is obtained using the resolution of $\mu^{-1} \mathcal{O}(\mathbb{E})$ by the relative holomorphic de Rham complex. 

The goal of this chapter is to discuss the notions of cycle space and double Fibration Transform in the superysmmetric setting.
As it turns out already the choices of the groups $G_\mathbb{R}$ and $K$, and consequently the definition of the cycle space,
are not at all obvious.  

\section{The Cycle spaces}

In order to define the cycle space in the supersymmetric case we first recall all the ingredients used in the construction of 
the classical cycle space:

\begin{itemize}
 \item The complex semisimple group $G_{\bar{0}}$
 \item The real form $G_{\bar{0}\mathbb{R}}$, defined by an antiholomorphic involution $\tau: G_{\bar{0}} \rightarrow G_{\bar{0}}$
 \item A Cartan involution $\theta$ compatible with $\tau$, i.e. $\tau\theta = \theta\tau$ and $-\kappa(X,\theta X) > 0$,
       where $\kappa: \mathfrak{g}_{\bar{0}\mathbb{R}} \times \mathfrak{g}_{\bar{0}\mathbb{R}} \rightarrow \mathbb{C}$ denotes the Killing form.
 \item A $G_{\bar{0}}$-flag manifold $Z_{\bar{0}}$, an open $G_{\bar{0}\mathbb{R}}$-orbit $D_{\bar{0}}$ and a closed $K_{\bar{0}}$-orbit $C_0$ in $D_{\bar{0}}$.
\end{itemize}

It turns out that in the supersymmetric case it is almost never possible to have all these available at the same time.
This is related to the fact there are very rarely compact real forms in the supersymmetric case.
Therefore, in order to define cycle spaces in the supersymmetric setting one needs to sacrifice some of the favorable
properties of the classical theory. 

The cases in which there are actually commuting involutions are the following:

\begin{enumerate}
 \item $G = SL_{n \vert m}(\mathbb{C})$ and $G_{\mathbb{R}} = SU(p, n-p \vert q, m-q)$. In that case 
$K = S(GL(p \vert q) \times GL(n-p \vert m-q))$ and $G_u = SU(n \vert m)$ is the compact real form.
 \item  $G = PSL_{n \vert n}(\mathbb{C})$ and $G_{\mathbb{R}} = US\Pi(n)$. Here $K = PSQ(n)$ and
$G_u$ is  $PSU(n \vert n)$.
 \item $G = PSQ(n)$ and $G_\mathbb{R} = UPSQ(p,q)$ with $K = PS(Q(p) \times Q(q))$ and $G_u = UPSQ(n)$ 
\end{enumerate}

For all other cases the requirements need to be weakened. In order to be able to make use of the classical theory for the bases, in particular of the existence of a unique base cycle, it is necessary to require $K_\mathbb{R}$ to be a maximal compact subsupergroup of $G_\mathbb{R}$. This leaves the following two main possibilities:

\begin{enumerate}
 \item The Cartan involution $\theta$ is replaced by the Cartan isomorphism $\theta: \mathfrak{g} \rightarrow \mathfrak{g}$
which coincides with the classical Cartan involution on $\mathfrak{g}_{\bar{0}}$ and satisfies $\theta^2(X) = -X$ for all $X \in \mathfrak{g}_{\bar{1}}$. It exists for all real forms of the basic classical Lie superalgebras, except for the two real forms $\ ^0\mathfrak{pq}(n)$ and $\mathfrak{us}\pi(n)$ of $\mathfrak{psl}_{n \vert n}(\mathbb{C})$, and for the real forms $\mathfrak{psq}_\mathbb{R}(n)$ and $\mathfrak{psq}_\mathbb{H}(k)$ of $\mathfrak{psq}(n)$.
 \item A $\mathbb{C}$-linear involution $\theta$ (an therefore $\mathfrak{k}$) is fixed and one considers an even real form of $\mathfrak{g}$, i.e. a $\mathbb{C}$-antilinear automorphism $\tau$  such that $\tau^2 \vert_{\mathfrak{g}_j} = (-1)^j \mathrm{id}_{\mathfrak{g}_j}$. This is possible for basic classical $\mathfrak{g}$ and all even real forms except for the even real form $\mathfrak{sl}_n(\mathbb{C})$ of $\mathfrak{psl}_{n \vert n}(\mathbb{C})$. This excludes the possibilities $\mathfrak{k} = \mathfrak{sp}(n)$ and $\mathfrak{k} = \mathfrak{psq}(n)$. 
\end{enumerate}

These two cases cover the two cases of flag domains defined in Chapter 1. 
The results on the non-existence of commuting automorphisms rely on the tables in \cite{Ser}.

\subsection{Flag domains of real forms}

In this case a real form $G_\mathbb{R}$ of $G$ and a flag domain $D \subseteq Z = G/P$ are given. Let $\theta$
be a Cartan involution or  a Cartan isomorphism commuting with $\tau$ and $K = \mathrm{Fix}(\theta)^\circ$. Then $D$
contains a unique closed $K$-orbit $C_0$. As it is closed, its stabilizer $J = \mathrm{Stab}_G(C_0)$ is a closed subsupergroup
of $G$.

Let $\mathcal{M}_Z = G/J$ and $\mathcal{M}$ be the open submanifold of $\mathcal{M}_Z$ with base

\[ \mathcal{M}_{\bar{0}} = \{ gC_0 \in (\mathcal{M}_Z)_{\bar{0}} : gC_0 \subseteq D \}^\circ \]

Then $\mathcal{M}$ is the cycle space of $D$.

\begin{rem}
\begin{enumerate}
\item If $ \mathrm{ord} \ \theta = 4$, then $K = K_{\bar{0}}, C_0 = C_{0{\bar{0}}}$ and $J = J_{\bar{0}}$ are purely even and the odd functions
         on $\mathcal{M}$ are given by the trivial bundle with fibre $\mathfrak{g}_{\bar{1}}^*$.
\item The base of the cycle space $\mathcal{M}$ need not equal the cycle space of the base $D_{\bar{0}}$. For example,
         if $G_\mathbb{R} = SU(p,q \vert r,s)$ and $D_{\bar{0}} = D_1 \times D_2$ with $D_1$ hermitian holomorphic and $D_2$ cycle-connected, 
then $\mathcal{M}_{\bar{0}} = G_{{\bar{0}},1}/K_{{\bar{0}},1} \times G_{{\bar{0}},2}/K_{{\bar{0}},2}$, but $\mathcal{M}_{D_{\bar{0}}} = G_{{\bar{0}},1}/(KS_\pm)_{{\bar{0}},1} \times G_{{\bar{0}},2}/K_{{\bar{0}},2}$.  
\end{enumerate}
\end{rem}

\subsection{Flag domains of even real forms}

Now a holomorphic involution $\theta: G \rightarrow G$ is given. Let $K = \mathrm{Fix}(\theta)$ and $G_\mathbb{R}$
be an even real form such that $K \cap G_\mathbb{R}$ is a maximal compact subgroup of $G_\mathbb{R}$.
Let $Z = G/P$ be a flag supermanifold. Then $K$ has finitely many closed orbits in $Z$ and for each of these closed orbits
$C_0$, its base $C_{0\bar{0}}$ is contained in a unique open $G_\mathbb{R}$-orbit $D_{\bar{0}}$. Let $D$ be the open submanifold of $Z$ with base $D_{\bar{0}}$ 
and $J = \mathrm{Stab}_G(C_0)$. It is a closed subsupergroup of $G$. Again let $\mathcal{M}_Z = G/J$
and $\mathcal{M}$ be the open subset with base $\mathcal{M}_{\bar{0}}$ as above.  Note that again the base of the cycle space
need not agree with the cycle space of the base. \newline

Using the notion of the universal domain from [FHW] the two definitions can be put together as follows:

\begin{def1}

Let $Z = G/P$ be a flag supermanifold, $G_\mathbb{R}$ a real form or an even real form of $G$ and $D \subseteq Z$
a flag domain. The group-theoretic cycle space of $D$, denoted $\mathcal{M}$ is the open submanifold of $\mathcal{M}_Z = G/J$ 
which has as its base the universal domain inside $(\mathcal{M}_Z)_{\bar{0}} = G_{\bar{0}}/J_{\bar{0}}$. 

\end{def1}

As the universal domain is Stein, $\mathcal{M}$ will always be a split supermanifold.

\section{The Double Fibration Transform}  

As in the classical case $D$ and $\mathcal{M}$ fit into a double fibration

\[ \xymatrix{ & \mathfrak{X} \ar[dl]_\mu \ar[dr]^\nu & \\ D & & \mathcal{M} } \]

Here $\mathfrak{X}$ is the open submanifold of $\widetilde{\mathfrak{X}} = G/(J \cap P)$ with base
the universal family $\mathfrak{X}_{\bar{0}} = \{ (z,C) \in D_{\bar{0}} \times \mathcal{M}_{\bar{0}} : z \in C \}$.

Let $\mathcal{E}$ be a supervector bundle on $D$.
The Double Fibration Transform relates the cohomology of $\mathcal{E}$ with sections of a certain associated bundle
on $\mathcal{M}$. Its construction is largely analogous to the classical case. However there are some important details
which turn out to be slightly different.

\subsection{Pullback}

The first step is to pull back the cohomology from $D$ to $\mathfrak{X}$ along $\mu$. As in the classical case, for every $r \geq 0$, $\mu$ induces a map $\mu_{\bar{0}}^r : H^r(D,\mathcal{E}) \rightarrow H^r(\mathfrak{X},\mu_{\bar{0}}^{-1}\mathcal{E})$.

The condition for these maps to be injective is the Buchdahl $q$-condition(see \ref{Buchdahl} and \cite{Bu}) which is purely topological. As the fibres of $\mu_{\bar{0}}$ are contractible as in the classical case(see  \cite{FHW}, Theorem 14.5.2 and Proposition 14.6.1 on page 212f.), it will trivially be satisfied.

Now let $\mu^*\mathcal{E} = \mu_{\bar{0}}^{-1}\mathcal{E} \otimes_{\mu_{\bar{0}}^{-1}\mathcal{O}_D} \mathcal{O}_\mathfrak{X}$ denote the pullback sheaf.  It is a holomorphic supervector bundle on $\mathfrak{X}$. The Extension of Scalars induces morphisms 
$i_r: H^r(\mathfrak{X}, \mu_{\bar{0}}^{-1}\mathcal{E}) \rightarrow H^r(\mathfrak{X}, \mu^*\mathcal{E})$. Let $j_r$ be the composition $   i_r \mu_{\bar{0}}^r$. Injectivity of these maps $j_r$ is equivalent to injectivity of the Double Fibration Transform itself as in the classical case. 

\subsection{Pushdown}

The second step is to push $H^r(\mathfrak{X}, \mu^* \mathcal{E})$ down from $\mathfrak{X}$ to $\mathcal{M}$ along $\nu$. 
As in the classical case $\nu_{\bar{0}}$ will always be a proper map and $\mathcal{M}$ a Stein supermanifold. 
Consequently one may make use of a generalization of the Grauert Direct Image Theorem due to Vaintrob:

\begin{satz}(see \cite{Va})

Let $f: X \rightarrow Y$ be a proper map of complex supermanifolds with $Y$ Stein and $\mathcal{S}$ a coherent 
$\mathcal{O}_X$-module. Then the $p^{th}$ direct image $R^p f_* \mathcal{S}$ is a coherent $\mathcal{O}_Y$-module
for all $p \geq 0$. 

\end{satz}

Moreover, as $\mathcal{M}$ is Stein, $H^q(\mathcal{M},R^p \nu_* \mu^* \mathcal{E}) = 0$ for all $q > 0, p \geq 0$.
Consequently the Leray spectral sequence collapses to yield isomorphisms $R^r \nu_* : H^r(\mathfrak{X}, \mu^* \mathcal{E}) \cong H^0(\mathcal{M}, R^r \nu_* \mu^* \mathcal{E})$.

\begin{def1}

The Double Fibration Transform is the composition

\[ \mathcal{P} = R^r \nu_* j_r : H^r(D,\mathcal{E}) \rightarrow H^0(\mathcal{M}, R^r \nu_* \mu^* \mathcal{E}) \] 

\end{def1}

\subsection{Injectivity of the DFT}

As was shown above, given that the fibres of $\mu_{\bar{0}}$ are contractible and $\nu_{\bar{0}}$ is a proper map onto a Stein space,
injectivity of the double fibration transform $\mathcal{P}$ is equivalent to injectivity of the coefficient map
$i_r: H^r(\mathfrak{X}, \mu^{-1}\mathcal{E}) \rightarrow H^r(\mathfrak{X}, \mu^*\mathcal{E})$. Analogous to the classical
case one may construct a resolution of $\mu^{-1}\mathcal{E}$ by $\mathcal{O}_\mathfrak{X}$-modules using the relative holomorphic de Rham complex:

\[ \xymatrix{ 0 \ar[r] & \mu^{-1}\mathcal{E} \ar[r] & \mu^*\mathcal{E} \ar[r] & \Omega_\mu^1 \mathcal{E} \ar[r] & \ldots } \]

Note that contrary to the classical case this is an unbounded complex as the odd differentials $d\xi_j$ commute.
It does however still give rise to two spectral sequences converging to the hypercohomology. In the classical case
a sufficient condition for these to yield vanishing of cohomology below the $s^{th}$ degree is $H^p(C, (\Omega_\mu^r\mathcal{E})\vert_C ) = 0$ for all $p < s, r > 0$. 
This is not a very strict condition in the classical case as only finitely many of the sheaves $\Omega_\mu^r$ are non-zero. In the supersymmetric case $\Omega_\mu^r$ is non-zero for all $r \geq 0$.
Therefore a slight weakening of the condition is necessary. 

\begin{bsp}

Let $G$ be any Lie supergroup, $G_\mathbb{R}$ a compact real form of $G_{\bar{0}}$ and $\theta: \mathfrak{g} \rightarrow \mathfrak{g},
\theta(X) = (-1)^{\vert X \vert} X$ for homogeneous elements, so $K = G_{\bar{0}}$. Furthermore let $Z = G/P$ be any $G$-flag manifold.
In that case $D = Z, \mathfrak{X} = G/(P \cap G_{\bar{0}}) = G/P_{\bar{0}}$ and $\mathcal{M} = G/G_{\bar{0}} = (\mathrm{pt}, \bigwedge \mathfrak{g}_{\bar{1}}^*)$.
Then $\mathfrak{X}$ is split and $\Omega_{\mu,red}$ is the sheaf of germs of sections of the homogeneous vector bundle
$G_{\bar{0}} \times_{P_{\bar{0}}} \mathfrak{p}_{\bar{1}}^*$. As $\mathfrak{p}_{\bar{1}}^*$ contains the dual of the Levi part $\mathfrak{l}_{\bar{1}}$ of $\mathfrak{p}_{\bar{1}}$, this bundle
is highly positive and vanishing for all $\Omega_\mu^r$ can only be achieved if $\mathfrak{l}_{\bar{1}} = 0$.

\end{bsp}

Recall that the two spectral sequences were given by $^\prime E_2^{p,q} =$ \newline $H^p(\mathfrak{X}, \mathcal{H}^q(\mathfrak{X}, \Omega^\bullet(\mathcal{O}(\mathbb{E}))))$ and  $^{\prime\prime}E_2^{p,q} = H_d^q(H^p(\mathfrak{X},\Omega^\bullet(\mathcal{O}(\mathbb{E}))))$ and the sufficient condition for injectivity was that $^{\prime\prime} E_2^{s,0}$
survives to the $^{\prime\prime}E_\infty$-page. This yields the following theorem:

\begin{satz}

Let $s \geq 0$ and $\Omega_\mu^r\mathcal{E}$ denote the sheaf of $\mu$-relative differential $r$-forms on $\mathfrak{X}$ 
with values in $\mu^* \mathcal{E}$. Let $C$ be a fibre of $\nu$ and suppose $H^p(C, (\Omega_\mu^r\mathcal{E})\vert_C ) = 0$
for all $p < s, r \leq s$. \\ Then $\mathcal{P}: H^s(D,\mathcal{E}) \rightarrow H^0(\mathcal{M}, \mathcal{R}^s \nu_* \mu^* \mathcal{E})$ is injective.

\end{satz}

\begin{proof}

Consider the two spectral sequences converging to the hypercohomology of the holomorphic de Rham complex. 
As the de Rham complex is exact $^\prime E_2^{p,q} = H^p(\mathfrak{X},\mu_0^{-1}\mathcal{E})$ for $q = 0$ and zero otherwise. Thus this spectral sequence collapses to yield the hypercohomology. 
Moreover one may identify $^{\prime \prime} E_1^{p,q} = H^p(\mathfrak{X}, \Omega_\mu^q(\mathcal{E}))$ and $^{\prime\prime} E_2^{s,q} = \mathrm{Ker} d_q / \mathrm{Im} d_{q-1}$. 
The vanishing condition ensures that all differentials mapping into $E_2^{s,0}$ or out of it are zero. 
Therefore $E_2^{s,0} = H^s(\mathfrak{X},\mu^*\mathcal{E})$ survives to the $E_\infty$-page and one may identify 
$H^s(\mathfrak{X},\mu_0^{-1}\mathcal{E}) = \mathrm{Ker} \ d_0$, in particular $j_s$ is injective. 

\end{proof}

The cohomology groups $H^p(C, (\Omega_\mu^r\mathcal{E})\vert_C )$ are not always known in the super case,  
it is however possible to obtain weaker conditions for injectivity by using classical vanishing results together
with the spectral sequence of Onishchik and Vishnyakova. Assume a supervector bundle $\mathcal{E}$ on $C$ is given
and let $E$ be the spectral sequence constructed in chapter. Then $E_1 = H^*(C, \mathrm{gr} \mathcal{E})$
and $E_\infty = \mathrm{gr} H^*(C, \mathcal{E})$ so the cohomology of $\mathrm{gr} \mathcal{E}$ is an upper bound
for the cohomology of $\mathcal{E}$, in particular vanishing of cohomology for $\mathrm{gr} \mathcal{E}$ implies 
vanishing of cohomology for $\mathcal{E}$:

\begin{satz}

Let $s \geq 0$ and $\Omega_\mu^r\mathcal{E}$ denote the sheaf of $\mu$-relative differential $r$-forms on $\mathfrak{X}$ 
with values in $\mu^* \mathcal{E}$. Let $C$ be a fibre of $\nu$ and suppose $H^p(C, \mathrm{gr} (\Omega_\mu^r\mathcal{E})\vert_C ) = 0$
for all $p < s, r \leq s$. \\ Then $\mathcal{P}: H^s(D,\mathcal{E}) \rightarrow H^0(\mathcal{M}, \mathcal{R}^s \nu_* \mu^* \mathcal{E})$ is injective.

\end{satz}

In view of this theorem the classical technique of tensoring with a sufficiently negative line bundle to achieve vanishing of cohomology below top degree generalizes to the supersymmetric case.

\begin{def1}[Notation]
 
Let $M(\Omega_\mu^{\leq s})$ be the set of all weights $\lambda$ such that the irreducible highest weight module $E_\lambda$
is contained in the fibre $\bigwedge^q (\mathfrak{p}/(\mathfrak{p} \cap \mathfrak{j}))^*$ of $\Omega_\mu^q$ for some
$q \leq s$. Its elements are of the form $\lambda = \sum_{i = 1}^{n_{\bar{0}}} \varepsilon_i \gamma_{0,i} + \sum_{j = 1}^{n_{\bar{1}}} k_j \gamma_{1,j}$
where $\varepsilon_i \in \{ \pm 1 \}, k_j \in \mathbb{Z}, \gamma_{l,i} \in \Sigma_l(\mathfrak{g},\mathfrak{h})$
and $n_{\bar{0}} + \sum_{j = 1}^{n_{\bar{1}}} \vert k_j \vert \leq s$. 

\end{def1}

Recall that in the classical case the total spaces of the sheaves $\Omega_\mu^p$ are actually homogeneous vector bundles.
This allows the characterization of the injectivity condition in terms of Bott-Borel-Weil theory(see Chapter 1).
In the supersymmetric case the sheaves $\Omega_\mu^p$ are homogeneous supervector bundles in the sense of \cite{AH}, but the BBW theory is not as readily available as in the classical case.
In the following sections the classical Bott-Borel-Weil-Theory and several analogous results for Lie superalgebras are used to obtain injectivity conditions for the Double Fibration Transform.

\section{Purely even cycles}

If the cycles within $D$ are purely even the classical BBW theory is available (compare with the results in section \ref{BBW}). This occurs in the following two cases:

\begin{enumerate}
 \item $K = G_{\bar{0}}$ given by the involution $\theta(X) = (-1)^{\vert X \vert} X$ and $G_\mathbb{R}$ is
       a compact even real form of $G$.
 \item $D$ is a flag domain of a real form $G_\mathbb{R}$ with Cartan isomorphism $\theta$.
\end{enumerate}

Moreover in those cases $\mathfrak{X}$ and $\mathcal{M}$ are both split supermanifolds with their global odd functions
given by the trivial bundle with fibre $\mathfrak{g}_{\bar{1}}^*$. 

\subsubsection{The case $K = G_{\bar{0}}$}

Consider the first case, i.e. $K = G_{\bar{0}}$ and $G_\mathbb{R}$ is an arbitrary even real form of $G$. This is an analogue
of the trivial choice of a compact real form in the classical case. The relevant spaces are given by $D_{\bar{0}} = \mathfrak{X}_{\bar{0}} = Z_{\bar{0}}$
and $\mathcal{M}_{\bar{0}} = \{pt\}$. Moreover the $\mu$-relative differential 1-forms are given by the homogeneous supervector bundle
$G \times_{P_{\bar{0}}} \mathfrak{p}_{\bar{1}}^*$. The unique cycle in $Z$ is the base $Z_{\bar{0}}$ and $\Omega_\mu^p \vert_C$ corresponds to the 
classical homogeneous bundle $\mathcal{F}(G_{\bar{0}} \times_{P_{\bar{0}}} S^p\mathfrak{p}_{\bar{1}}^*)$. If $\mathcal{E}$ is a super vector bundle on $Z$,
the injectivity of the pullback map is assured given that

\[ H^q(Z_0, \mathcal{F}(G_{\bar{0}} \times_{P_{\bar{0}}} S^p\mathfrak{p}_{\bar{1}}^*) \otimes \mathcal{E}_{red}) = 0 \ \forall q < \dim Z_{\bar{0}}, p \leq \dim Z_{\bar{0}} \]

If $\mathcal{E}_{red} = \mathcal{E}_\lambda$, then using the classical BBW theory the injectivity condition becomes

\[ \langle \lambda + \beta + \rho_{\bar{0}}, \gamma \rangle < 0 \ \forall \beta \in M(\Omega_\mu^{\leq s}), \gamma \in \Sigma((\mathfrak{r}_+)_{\bar{0}},\mathfrak{h}) \] 

Here $M(\Omega_\mu^{\leq s})$ is actually the set of all roots contributing to the truncated symmetric algebra $S(\mathfrak{p}_{\bar{1}}^*)^{\leq s}$.

\noindent Moreover the target space $H^0(\mathcal{M},\mathbb{R}^s \nu_* \mu^* \mathcal{E})$ can be identified with \\ $H^s(Z_{\bar{0}}, \mathcal{E}_{red}) \otimes \bigwedge \mathfrak{g}_{\bar{1}}^*$
which incidentally is isomorphic to $H^s(\mathfrak{X}, \mu^* \mathcal{E})$. So the Double Fibration transform is actually a map

\[ \mathcal{P} : H^s(Z, \mathcal{E}) \rightarrow H^s(Z_{\bar{0}}, \mathcal{E}_{red}) \otimes \bigwedge \mathfrak{g}_{\bar{1}}^* \]

\begin{bsp}

Let $Z = \mathrm{Gr}_{1 \vert 1}(\mathbb{C}^{2 \vert 2}) = PSL_{2 \vert 2}(\mathbb{C})/P, K = SL_2(\mathbb{C}) \times SL_2(\mathbb{C})$ 
and $G_\mathbb{R} = SU(2) \times SU(2)$. Then $D_{\bar{0}} = Z_{\bar{0}} = \mathfrak{X}_{\bar{0}} = \mathbb{P}^1 \times \mathbb{P}^1$ and $\mathcal{M}_{\bar{0}} = \{pt\}$.
Also let $\mathcal{E} = \mathcal{O}_D$. Its cohomology is given by

\[ H^0(D, \mathcal{O}_D) = \mathbb{C}, H^1(D, \mathcal{O}_D) = 0, H^2(D, \mathcal{O}_D) \cong \mathbb{C} \]  

\noindent where a possible generator of the latter group is given on $V_1$ by $\frac{\xi\eta}{z_1 z_2}$. The pullback map is not injective
as $H^*(\mathfrak{X}, \mathcal{O}_\mathfrak{X}) \cong H^*(\mathbb{P}^1 \times \mathbb{P}^1, \mathcal{F}) \otimes \mathfrak{g}_{\bar{1}}^*$
which is concentrated in degree $0$.

This shows how the pullback map is in a certain sense blind to the odd contributions to the cohomology when purely even cycles are considered.

\end{bsp}

\subsubsection{Flag domains of real forms}

Now consider the second case, i.e. $G_\mathbb{R}$ is a real form of $G$ and $K = K_{\bar{0}}$ is the fixed point set of the Cartan isomorphism.
Then $\mathfrak{X}_Z = G/(J_{\bar{0}} \cap P) = (G_{\bar{0}}/(J_{\bar{0}} \cap P_{\bar{0}}), \mathcal{F} \otimes \bigwedge \mathfrak{g}_{\bar{1}}^*)$ and 
$\mathcal{M}_Z = (G_{\bar{0}}/J_{\bar{0}}, \mathcal{F} \otimes \bigwedge \mathfrak{g}_{\bar{1}}^*)$, so the underlying manifolds of $\mathfrak{X}_{\bar{0}}$
and $\mathcal{M}_{\bar{0}}$ actually agree with the classical universal family and the classical cycle space associated to $D_{\bar{0}}$.

The sheaf $\Omega_\mu^1$ is given by $G \times_{(P_{\bar{0}} \cap J_{\bar{0}})} ((\mathfrak{p}_{\bar{0}}/\mathfrak{j}_{\bar{0}})^* + (\mathfrak{p}_{\bar{1}})^*)$.
Thus, as in the first case, the full symmetric algebra over $(\mathfrak{p}_{\bar{1}})^*$ contributes to the relative de Rham complex. 
The cycles are precisely the classical cycles and the injectivity condition is as follows:

\begin{satz}
 
Let $\mathcal{E}_\lambda$ be the super vector bundle on $D$ whose typical fibre is the irreducible
representation of $P$ with highest weight $\lambda$. Then the pullback map $j_s$ is injective if and only if

\[ \langle \lambda + \beta + \rho_\mathfrak{k}, \gamma \rangle < 0 \ \forall \beta \in M(\Omega_\mu^{\leq s}), \gamma \in \Sigma(\mathfrak{r}_+ \cap \mathfrak{k},\mathfrak{h}) \]

Moreover the fibre of the associated vector bundle $\mathcal{E}^\prime$ on $\mathcal{M}$ is $E_\Lambda$ for $\Lambda = w(\lambda + \rho_\mathfrak{k}) - \rho_{\mathfrak{k}}$, where
$w$ is the unique Weyl group element such that $w(\lambda + \rho_\mathfrak{k})$ is integral dominant.

\end{satz}

\section{Cycles of positive odd dimension I: The distinguished Borel case}

The Bott-Borel-Weil Theory in the supersymmetric case is not yet fully developed. The only case which is understood to a great extent is that of the distinguished Borel subalgebras:

\begin{def1}

Let $\mathfrak{g}$ be a basic classical simple Lie superalgebra. A Borel subalgebra $\mathfrak{b}^d$ of $\mathfrak{g}$
is called distinguished if and only if one of the following equivalent conditions is satisfied:

\begin{enumerate}
 \item The Dynkin diagram of $\Sigma(\mathfrak{g},\mathfrak{h})$ with respect to a Cartan subalgebra $\mathfrak{h} \subseteq \mathfrak{b}^d$
       contains exactly one odd root.
 \item For every parabolic subalgebra $\mathfrak{p}$ of $\mathfrak{g}$ containing $\mathfrak{b}^d$ one has $\mathfrak{l} \subseteq \mathfrak{g}_{\bar{0}}$.
\end{enumerate}

If $\mathfrak{g}$ is of type $Q$, there is only one conjugacy class of Borel subalgebras. Every representative of it is called a distinguished Borel subalgebra.
 
\end{def1}

\begin{bsp}

One possible distinguished Borel subalgebra of $\mathfrak{sl}_{n \vert m}$ consists of all upper triangular $(n+m) \times (n+m)$-matrices.
The respective flag supermanifold $G/B^d = Z(\delta)$ corresponds to the maximal flag type 

\[ \delta = 0 \vert 0 < 1 \vert 0 < \ldots < n \vert 0 < n \vert 1 < \ldots < n \vert m  \]

\noindent The - up to conjugacy - other possible choice is the set of all lower triangular matrices. It corresponds to the maximal flag type

\[ \delta = 0 \vert 0 < 0 \vert 1 < \ldots < 0 \vert m < 1 \vert m < \ldots < n \vert m \] 

\noindent The following can be said about the flag spaces $G/B^d$ using Vishnyakova's theorems:

\begin{enumerate}
 \item $Z = G/B^d$ is a split homogeneous space.
 \item $\mathcal{O}_{Z}$ is given by the trivial bundle with fibre $\mathfrak{g}_{\pm1}$, one of the two irreducible components of $\mathfrak{g}_{\bar{1}}$, in particular $H^0(Z, \mathcal{O}_{Z}) \cong \bigwedge \mathbb{C}^{nm}$
 \item For all parabolic subgroups $P$ containing $B^d$ one has either $G/P$ split and $H^0(G/P, \mathcal{O}_{G/P}) \cong \bigwedge \mathbb{C}^{nm}$ or $G/P$ non-split and $H^0(G/P, \mathcal{O}_{G/P}) = 0$. 
\end{enumerate} 

Also note that the real forms $SU(p,q \vert r,s)$ and $^0PQ(n)$ for $n = m$ have open orbits in $Z$, whereas the real orbits
of $SL_{n \vert m}(\mathbb{R})$ and $SL_{k \vert l}(\mathbb{H})$ with open base actually have minimal odd dimension. 
 
\end{bsp}

An important notion in the supersymmetric setting is typicality: 

\begin{def1}

Let $\mathfrak{g}$ be a classical reductive Lie superalgebra and $\mathfrak{p}$ a parabolic subalgebra.
The Levi subalgebra $\mathfrak{l}$ of $\mathfrak{p}$ is of typical type if all its finite-dimensional representations
are completely reducible. This implies that $\mathfrak{l}$ is isomorphic to a direct sum of reductive Lie algebras and 
Lie superalgebras isomorphic to $\mathfrak{osp}(1 \vert 2n)$. 

A weight $\lambda$  is called typical if $\langle \lambda, \gamma \rangle \neq 0$ for all anisotropic odd roots $\gamma$.

\end{def1}

The significance of typicality is that, unlike in the classical case,
$\langle \lambda , \alpha \rangle = 0$ for an anisotropic odd root does not imply the vanishing of all cohomology groups
$H^i(G/P, \mathcal{E}_\lambda)$. In fact atypicality is actually a necessary condition for the existence of more than one non-trivial 
cohomology group.

The cohomology groups of BBW theory are known for all basic classical Lie superalgebras of type I with distinguished Borel subalgebra 
and for the Lie superalgebras $\mathfrak{osp}(n \vert 2), D(2,1,\alpha), F(4)$ and $G(3)$ which are basic classical of type II.
So the target space of the Double Fibration Transform is known when the base cycle is $K/M$ where $\mathfrak{k}$ is one of
the Lie superalgebras listed above and $\mathfrak{m}$ is a parabolic subalgebra of $\mathfrak{k}$ with typical Levi part.
The list of possible choices for $K$ comprises the following cases:

\begin{enumerate}
\item $G = SL_{n \vert m}(\mathbb{C}), G_\mathbb{R} = SU(p,q \vert r,s)$ or $G_\mathbb{R} = S(U(p,q) \times U(r,s))$ and $K = S(GL(p \vert r) \times GL(q \vert s))$.  
\item $G = SL_{n \vert 2}(\mathbb{C}), G_\mathbb{R} = SL_n(\mathbb{R}) \times SL_1(\mathbb{H}) \times \mathbb{R}^{>0}$
and $K = \mathrm{Osp}(n \vert 2)$.
\item $G = SL_{2 \vert 2m}(\mathbb{C}), G_\mathbb{R} = SL_2(\mathbb{R}) \times SL_m(\mathbb{H}) \times \mathbb{R}^{>0}$
and $K = \mathrm{Osp}(2 \vert 2m)$.
\item $G = \mathrm{Osp}(2n \vert 2m), G_\mathbb{R} = SO^*(2n) \times Sp_{2m}(\mathbb{R})$ and $K = GL(n \vert m)$.
\item $G = \mathrm{Osp}(n+2 \vert 2m + 2), G_\mathbb{R} = SO(n,2) \times Sp(2,2m)$ and $K = \mathrm{Osp}(n,2) \times 
\mathrm{Osp}(2,2m)$
\item $G = \mathrm{Osp}(n+2 \vert 2m), G_\mathbb{R} = SO(n,2) \times Sp(2m)$ and $K = \mathrm{Osp}(n,2) \times 
\mathrm{Sp}(2m, \mathbb{C})$
\item $G = \mathrm{Osp}(n \vert 2m + 2), G_\mathbb{R} = SO(n,\mathbb{R}) \times Sp(2,2m)$ and $K = SO(n,\mathbb{C}) \times 
\mathrm{Osp}(2,2m)$
\end{enumerate}  

Note that this covers all cycle spaces whose bases are classical hermitian symmetric spaces and products of these(Cases 1,4 and 6).
Moreover this list could be extended in the future when a complete description of the BBW theory for $\mathfrak{g} = \mathfrak{osp}(m \vert 2n)(m > 2, n > 1)$ with distinguished Borel subalgebra is available. The condition that $\mathfrak{m}$ is a parabolic subalgebra of
$\mathfrak{k}$ with typical Levi part allows some freedom in the choice of the flag type $\delta$:

\begin{bsp}

Let $G = SL_{n \vert m}(\mathbb{C}), K = S(GL(p \vert r) \times GL(q \vert s))$ and $Z(\delta)$ a $G$-flag manifold. Then
the neutral point in $Z(\delta)$ has a $\mathfrak{k}$-stabilizer with typical Levi part, if and only if $p \vert r \in \delta, p \vert s \in \delta, q \vert r \in \delta$ or $q \vert s \in \delta$. 
Then the base cycle is isomorphic to $C_0 = (C_{\bar{0}0}, \mathcal{F} \otimes \bigwedge(\mathrm{Hom}(\mathbb{C}^p,\mathbb{C}^r)) \oplus \mathrm{Hom}(\mathbb{C}^q,\mathbb{C}^s))$, 
where $C_{\bar{0}0}$ is a classical base cycle of hermitian holomorphic type. 

\end{bsp}

The basic BBW type theorem that will be used to obtain the injectivity condition and the target for the Double Fibration transform is
the following:

\begin{satz}[\cite{Pen}]
 
Let $\mathfrak{g}$ be a classical Lie super algebra not of type $P$, $\mathfrak{p}$ a parabolic subgroup with Levi part $\mathfrak{l}$ of typical type. Also let $\lambda$ be an integral $\mathfrak{l}$-dominant typical weight.  Denote
$\Gamma_k(G/P, E_\lambda) = H^k(G/P, G \times_P E_\lambda^*)^*$.

\begin{enumerate}
 \item If $\lambda + \rho$ is singular, then $\Gamma_k(G/P, E_\lambda) = 0$ for all $k$.
 \item If $\lambda + \rho$ is regular, then there exists a unique $w \in W$ such that $\Lambda = w \cdot \lambda$ is
       integral dominant and $\Gamma_k(G/P, E_\lambda) = \delta_{k,l(w)} K_\Lambda$,
where $K_\Lambda$ is the Kac-module introduced in \cite{Kac2} .
\end{enumerate}
 
\end{satz}
 
Using this theorem the injectivity condition for the Double Fibration Transform and its target can now be expressed in terms of weights. First consider those cases in which $\mathfrak{k}$ is of type I. 
In that case, $\mathfrak{l}$ being of typical type implies that it is purely even. Consequently every $\mathfrak{l}$-weight is automatically typical and Penkov's theorem may be applied to all $\mathfrak{l}$-weights: 

\begin{satz}

Let $G$ be a complex simple Lie supergroup, $G_\mathbb{R}$ an even real form, $Z(\delta) = G/P$ a $G$-flag supermanifold, $D \subseteq Z(\delta)$ a $G_\mathbb{R}$-flag domain and $C = K/M$ be the base cycle. Assume that $K$ is of type I and $M$ has typical Levi part.
Let $\lambda \in \Sigma(\mathfrak{g},\mathfrak{h})$ be an integral weight, $E_\lambda$ be the irreducible $P$-module with highest weight $\lambda$ and $\mathcal{E}_\lambda = \mathcal{O}(G \times_P E_\lambda)$. If $\lambda$ is sufficiently negative,
that is $\langle \lambda + \beta + \rho_\mathfrak{k}, \gamma \rangle < 0$ for all $ \beta \in M(\Omega_\mu^{\leq s})$ 
and $\gamma \in \Sigma(\mathfrak{r}^+ \cap \mathfrak{k})$, then the the double fibration transform 
$\mathcal{P}: H^s(D, \mathcal{E}_\lambda) \rightarrow H^0(\mathcal{M}, \mathcal{E}_\lambda^\prime)$ is injective 
and the fibre of $\mathcal{E}_\lambda^\prime$ is the Kac-module $K_\Lambda, \Lambda = w \cdot \lambda$, 
where $w \in W$ is the unique Weyl group element such that $w \cdot \lambda$ is integral dominant. 
 
\end{satz} 
 
Now consider the basic classical Lie superalgebras $\mathfrak{k}$ of type II whose BBW theory for the distinguished Borel subalgebras is known. 
As the exceptional Lie superalgebras do not arise as fixed point sets of involutions of the classical non-exceptional Lie superalgebras,
this restricts to the case $\mathfrak{k} = \mathfrak{osp}(m \vert 2)$. If $m$ is even, then every Levi subalgebra of typical type is again purely even
so the last theorem applies verbatim. However, if $m$ is odd and $\mathfrak{l} \subseteq \mathfrak{m}$ is of typical type, then it may
contain a direct summand isomorphic to $\mathfrak{osp}(1 \vert 2)$, so $\mathfrak{l}$ contains a unique anisotropic odd simple root $\alpha$.
So typicality of a weight $\lambda$ is equivalent to $\langle \lambda, \alpha \rangle \neq 0$. 

The main difference to all cases considered so far is that if $\lambda$ is atypical, there might be two non-trivial cohomology groups
for $\mathcal{E}_\lambda$(compare \cite{C}, Theorem 10.1):

Recall from Chapter 1 that $\mathfrak{g}$ allows a natural $\mathbb{Z}$-grading $\mathfrak{g} = \mathfrak{g}_{-2} \oplus \mathfrak{g}_{-1} \oplus \mathfrak{g}_0 \oplus \mathfrak{g}_1 \oplus \mathfrak{g}_2$
with $\dim \mathfrak{g}_{\pm 2} = 1$ and the Weyl group satisfies $W(\mathfrak{g},\mathfrak{h}) = \mathbb{Z}/2 \times W(\mathfrak{g}_0,\mathfrak{h})$.
Let $\sigma$ be the generator of the $\mathbb{Z}/2$-factor and $w_0$ be the longest element of $W(\mathfrak{g}_0,\mathfrak{h})$. 
If $\Lambda$ is an atypical integral dominant root, it is possible that $\sigma \cdot \Lambda$ is again integral dominant. 
This is the only possible case in which there are two non-trivial cohomology groups $H^0(C,\mathcal{E}_\Lambda)$ and $H^1(C, \mathcal{E}_\Lambda)$.
Now, if $\lambda$ is a negative regular atypical root such that $w_0 \cdot \lambda = \Lambda$, then the two top-most
cohomology groups $H^s(C, \mathcal{E}_\lambda)$ and $H^{s-1}(C, \mathcal{E}_\lambda)$ will be nontrivial. This yields two possibly non-trivial
Double fibration transforms and the fibres of the target spaces $H^0(\mathcal{M}, \mathcal{R}^s\nu_*\mu^* \mathcal{E}_\lambda)$ and 
$H^0(\mathcal{M}, \mathcal{R}^{s-1}\nu_*\mu^* \mathcal{E}_\lambda)$ are twisted duals of each other.

The possible existence of two nontrivial cohomology groups also leads to stronger injectivity conditions:
Apart from the usual condition that $\sigma w_0 \cdot (\lambda + \beta)$ is integral dominant for all $\beta \in M(\Omega^{\leq s})$
one needs to impose the additional condition that $\lambda + \beta$ is either typical or $w_0 \cdot (\lambda + \beta)$ is not
integral dominant:
 
\begin{satz}

Let $G,G_\mathbb{R},Z(\delta)$ and $D$ as above. Assume that $K$ is of type II and $M$ has typical Levi part. Let $\lambda$ be an integral $\mathfrak{l}$-dominant weight. Moreover, assume $\langle \lambda + \beta + \rho_\mathfrak{k}, \gamma \rangle < 0$ for all $\beta \in M(\Omega^{\leq s}), \gamma \in \Sigma(\mathfrak{p}^c \cap \mathfrak{k})$ and either $\lambda + \beta$ typical or $\langle \lambda + \beta + \rho_{\mathfrak{k}}, \gamma_\sigma \rangle \not< 0$ for each $\beta \in M(\Omega^{\leq s})$, where $\gamma_\sigma$ is the unique root such that $\sigma = r_{\gamma_\sigma}$.

\begin{enumerate}
 \item  If $\lambda$ is typical or $\lambda$ is atypical and $w_0 \cdot \lambda$ not integral dominant, then the Double Fibration transform  
$\mathcal{P}: H^s(D, \mathcal{E}_\lambda) \rightarrow H^0(\mathcal{M}, \mathcal{E}_\lambda^\prime)$ is injective and 
the fibre of $\mathcal{E}_\lambda^\prime$ is the Kac module $K_\Lambda, \Lambda = \sigma w_0 \cdot \lambda$.
 \item If $\lambda$ is atypical and $w_0 \cdot \lambda$ is integral dominant, then there are two nontrivial 
       Double Fibration Transforms $\mathcal{P}^\prime: H^s(D, \mathcal{E}_\lambda) \rightarrow H^0(\mathcal{M}, \mathcal{E}_\lambda^\prime)$
       and $\mathcal{P}^{\prime\prime}:H^{s-1}(D, \mathcal{E}_\lambda) \rightarrow H^0(\mathcal{M}, \mathcal{E}_\lambda^{\prime\prime})$.
       Both transforms are injective and the fibres of $\mathcal{E}^\prime$ and $\mathcal{E}^{\prime\prime}$ are the Kac module $K_\Lambda$
       and its twisted dual $K_\Lambda^\vee$ respectively.
\end{enumerate}

\end{satz}

Note that $K_\Lambda^\vee$ is not a highest weight module, so this is the first occurence of a target space, which is not a highest weight space.

For $\mathfrak{k} = \mathfrak{osp}(m \vert 2n)(n > 1)$ with distinguished Borel subalgebra the BBW theory is yet unknown. 
But the results in the case $n = 1$ suggest that there might well be nontrivial cohomology groups arising from a direct summand
of the form $\mathfrak{osp}(1 \vert 2k)$ in the Levi subalgebra of $\mathfrak{m}$ as well. 

This concludes the discussion of the basic classical Lie superalgebras with distinguished Borel subalgebra.

\section{Cycles of positive odd dimension II: The general case}

If an arbitrary Borel subalgebra $\mathfrak{b} \subseteq \mathfrak{g}$ is given little is known about the BBW theory.
The results which are available require strong notions of genericity, so they are only available for weights lying
far away from the walls of the Weyl chambers. The following notion of genericity is introduced in \cite{C}:

\begin{def1}
 
Let $\Gamma^+$ be the set of all formal sums $\sum_{i \in I \subseteq \Sigma_1^+} \gamma_i$ of positive odd roots
and $\tilde{\Gamma}$ be the respective set of all formal sums of arbitrary odd roots.

\begin{enumerate}
 \item A weight $\lambda \in \mathfrak{h}^*$ is $\Gamma^+$-generic if all weights in $\lambda - \Gamma^+$ lie inside the
same Weyl chamber.
 \item A weight $\lambda \in \mathfrak{h}^*$ is $\tilde{\Gamma}$-generic if all weights in $\lambda - \tilde{\Gamma}$ lie inside the
same Weyl chamber.
 \item A weight $\lambda \in \mathfrak{h}^*$ is called generic if every weight in $\lambda - \Gamma^+$ is $\tilde{\Gamma}$-generic.
\end{enumerate} 

\end{def1}

Note that all of these three notions are invariant under the action of the Weyl group and that $\Gamma^+$-generic weights
are necessarily typical. For generic weights the cohomology groups $H^k(G/B,\mathcal{E}_\lambda)$ are described as follows:

\begin{satz}
 
Let $\mathfrak{g}$ be a basic classical Lie superalgebra, $\mathfrak{b}$ an arbitrary Borel subalgebra of $\mathfrak{g}$,
$\Lambda$ integral dominant and $\Gamma^+$-generic and $w \in W$. Then $\Gamma_k(G/B, E_{w \cdot \Lambda}) = K_\Lambda^{(\mathfrak{b})}[w]$
for $k = l(w)$ and zero else.

\end{satz}
 
Unlike in the classical case, the relation between the BBW theory of a Borel subalgebra $\mathfrak{b}$ and a parabolic subalgebra $\mathfrak{p} \supseteq \mathfrak{b}$ is not yet well understood in general. Some results have been obtained for relatively generic weights (compare \cite{C}, Section 8): Let $\mathfrak{b}$ be a Borel subalgebra and $\mathfrak{n}$ its nilradical. Then there is a parabolic subalgebra $\mathfrak{p}^\mathfrak{b} \subseteq \mathfrak{g}$, maximal among all parabolic subalgebras with typical Levi part that contain  $\mathfrak{b}$. Its Levi part is denoted $\mathfrak{l}^\mathfrak{b}$ and $\mathfrak{p}^\mathfrak{b} = \mathfrak{n} + \mathfrak{l}^\mathfrak{b}$. 

\begin{def1}

Let $\mathfrak{g}$ be a basic classical Lie superalgebra, $\mathfrak{b}$ a Borel subalgebra, $\lambda$ an
integral $\mathfrak{l}^\mathfrak{b}$-dominant weight and $S$ a set of integral weights. Then $\lambda$ is relatively
$S$-generic for $\mathfrak{b}$ if and only if every weight in $\lambda - S$, which is $\mathfrak{l}^\mathfrak{b}$-dominant,
lies in the same Weyl chamber as $\lambda$.

\end{def1}

For relatively generic weights, the BBW theory of $\mathfrak{p}^\mathfrak{b}$ and $\mathfrak{b}$ coincide:
  
\begin{satz}
 
If $\lambda$ is relatively $\Gamma^+$-generic then there is exactly one $w \in W_\mathfrak{b}^1 = W(\mathfrak{g},\mathfrak{h})/W(\mathfrak{l}_{\bar{0}}^\mathfrak{b}, \mathfrak{h})$ such that $w \cdot \lambda$
is integral dominant. \\ Moreover $\Gamma_k(G/B,E_\lambda) \cong \Gamma_k(G/P^\mathfrak{b},E_\lambda) = \delta_{k,l(w)} M$
with $\mathrm{ch} M = \mathrm{ch} K_{w \cdot \lambda}^{(\mathfrak{b})}$, i.e. $M$ has the same character as the Kac module.

\end{satz} 

These results can be used to obtain injectivity conditions for the Double Fibration Transform. They can be applied under
the strong sufficient condition that each $\lambda + \beta, \beta \in M(\Omega_\mu^{\leq s})$ should be generic.

\begin{satz}
 
Let $G,G_\mathbb{R},Z$ and $D$ as before and $\lambda$ be a sufficiently negative integral weight
such that $\lambda + \beta$ is generic for all $\beta \in M(\Omega_\mu^{\leq s})$. Then the Double Fibration 
Transform $\mathcal{P}: H^s(D, \mathcal{E}_\lambda) \rightarrow H^0(\mathcal{M},\mathcal{E}_\lambda^\prime)$ is
injective. If $P = B$, then the fibre of $\mathcal{E}_\lambda^\prime$ is the Kac module $K_\Lambda^{(\mathfrak{b})}$.

\end{satz}
   
If $P \neq B$ in the above theorem it is only known that the characters of the Kac module and the fibre of $\mathcal{E}_\lambda^\prime$
agree. 

This last theorem is the most general result that is available at this point. The reason for this is that for parabolic subalgebras $P$
whose Levi parts are not of typical type the relation between the cohomology groups for $G/B$ and $G/P$ is not yet known.
Moreover in the case of an arbitrary Borel subalgebra little is known about the BBW theory for weights which are not contained in the generic region.
This particularly includes atypical weights and the known results for $\mathfrak{k} = \mathfrak{osp}(n \vert 2)$ indicate that
there are very probably many examples featuring several non-vanishing cohomology groups. Note that the given example was also
the first occurence of a cycle with non-trivial odd structure.

To go on from the results in this thesis a very interesting case to study is that of the minimal $\Pi$-symmetric parabolic
subalgebra $\mathfrak{p}$ of $\mathfrak{g} = \mathfrak{psl}_{n \vert n}(\mathbb{C})$. It is given by

\[ \mathfrak{p} = \left\lbrace \begin{pmatrix} A & B \\ C & D \end{pmatrix} \in \mathfrak{g} ; A,B,C,D \in \mathfrak{b}_{std} \right\rbrace \] 

where $\mathfrak{b}_{std}$ is the standard Borel subalgebra of $\mathfrak{sl}_n(\mathbb{C})$. The corresponding flag space
is $G/P = Z(\delta), \delta = 0 \vert 0 < 1 \vert 1 < \ldots < n-1 \vert n-1 < n \vert n$ and the Levi subalgebra is 
$\mathfrak{l} \subseteq \mathfrak{l}_1 + \ldots + \mathfrak{l}_n$, where each $\mathfrak{l}_i$ is isomorphic to 
$\mathfrak{gl}_{1 \vert 1}(\mathbb{C})$. In particular $\mathfrak{l}$ is solvable which does render the representation
theory of $P$ highly complicated. In particular for a line bundle $\mathcal{L}$ on $G/P$ there is a high possibility of the existence
of several non-trivial cohomology groups.

The BBW theory for this particular parabolic subalgebra and the parabolic subalgebras containing it is of 
particular interest for the following two reasons:

\begin{itemize}
 \item Most examples of weakly measurable open orbits are contained in flag spaces $Z = G/Q$ where $Q$ contains $P$
 \item All open orbits for the real form $US\Pi(n)$ are contained in those flag spaces as well.
\end{itemize}

In particular, these interesting open orbits are directly connected to the BBW theory of the atypical blocks
and further progress on this aspect of BBW theory would greatly increase the understanding of the phenomenon of weak measurability.

\end{document}